\documentclass[english,a4paper,11pt]{article}
\usepackage[a4paper,hmargin=1.0in,vmargin=1.0in]{geometry}
\usepackage[round]{natbib}
\usepackage[usenames,dvipsnames]{xcolor}
\usepackage[linktocpage=true,pagebackref=true,colorlinks,allcolors=blue,bookmarks=true,bookmarksopen,bookmarksnumbered]{hyperref}

\usepackage{xurl}
\usepackage{sectsty}
\usepackage[english]{babel}
\usepackage[utf8]{inputenc}
\usepackage{amsfonts}
\usepackage{amssymb}
\usepackage{amsmath}
\usepackage{comment}
\usepackage{setspace}
\usepackage{bbm}
\usepackage{mathtools}
\usepackage{enumitem}
\usepackage{amsthm}
\usepackage{tikz}
\usepackage{booktabs}
\usepackage{array}
\usepackage{cleveref}
\usepackage{bm}
\usepackage{pdflscape}
\usepackage{adjustbox}
\usepackage{longtable}
\usepackage{caption}
\usepackage{multirow}
\usepackage{rotating}
\usepackage{graphicx}
\usepackage{breqn}
\usepackage{needspace}

\newcommand{\stkout}[1]{\ifmmode\text{\sout{\ensuremath{#1}}}\else\sout{#1}\fi}

\newtheoremstyle{DStheorem}
  {\topsep}
  {\topsep}
  {\itshape}
  {0pt}
  {\scshape}
  {.}
  { }
  {\thmname{#1}\thmnumber{ #2}\thmnote{ (#3)}}
\theoremstyle{DStheorem}
\newtheorem{theorem}{Theorem}[section]
\newtheorem{lemma}[theorem]{Lemma}
\newtheorem{claim}[theorem]{Claim}



\let\oldproofname=\proofname
\renewcommand{\proofname}{\rm\sc{\oldproofname}}

\sectionfont{\large} \subsectionfont{\normalsize}
\allowdisplaybreaks
\makeindex

\usepackage{xcolor}
\begin{document}

\begin{titlepage}

\title{Joint Inventory Placement, Assortment Personalization, and Order Fulfillment for Substitutable Products}
\author{Mikhail Fadin \quad Omar El Housni \quad Huseyin Topaloglu \\[1em] \normalsize School of Operations Research and Information Engineering \\ \normalsize Cornell Tech, New York, NY 10044, USA \\[0.3em] \normalsize \texttt{mf853@cornell.edu}, \texttt{oe46@cornell.edu}, \texttt{ht88@cornell.edu}}
\date{}
\maketitle

\pagenumbering{roman}
\thispagestyle{empty}

\begin{abstract}
Modern online retailers leverage networks of distributed warehouses to rapidly fulfill customer orders. We study a problem of jointly deciding (i) how to allocate inventories across warehouses in the network subject to warehouse capacity and product supply constraints, (ii) how to dynamically select personalized product assortments based on customer preferences and location, as well as real-time stock levels, and (iii) which warehouse to use to fulfill the product chosen by the customer to maximize expected profit. In our model, the firm chooses an inventory placement across warehouses at the start of the selling horizon. Customers of different types then arrive at discrete time periods over a finite time horizon according to a known distribution. When offered an assortment, each customer type chooses at most one product according to a discrete choice model. The firm then chooses a feasible warehouse from which to fulfill the chosen product, depleting its stock and earning a profit that depends on the product, customer type, and warehouse choice. For any fixed inventory placement, the optimal dynamic assortment and fulfillment policy can be characterized by a high-dimensional dynamic program, which is intractable in general. To address this, we develop approximation algorithms and policies with provable theoretical guarantees and strong empirical performance.  Under general probabilistic choice models, we establish asymptotically optimal approximation algorithms as warehouse capacities and product supplies scale. Under the multinomial logit (MNL) choice model, we obtain constant-factor approximation algorithms, with approximation factors ranging from 0.080 in the most general setting to 0.199 under stationarity and unlimited product supplies, which means that only warehouse capacity constraints are present. To the best of our knowledge, these results provide the first provably efficient algorithms for the joint inventory placement, assortment personalization, and order fulfillment problem. Computational experiments demonstrate that in practice our methods achieve near-optimal profits even for modest capacities and supply levels and, combined with simple additional heuristics, can perform even better and scale extremely well computationally.

\end{abstract}


\end{titlepage}

\pagenumbering{roman}
\thispagestyle{empty}
\tableofcontents

\newpage
\pagenumbering{arabic}
\setcounter{page}{1}
\section{Introduction}\label{sec:intro}
Leading e‑commerce platforms use intricate warehouse networks to strategically pre‑position inventory, flexibly customize offerings, and efficiently fulfill orders for customers spanning across multiple regions. The ability to tailor the product assortment based on customer information such as location, browsing history, or past purchases is a key source of competitive advantage. From the customer’s standpoint, personalized assortments align the displayed products with individual preferences and create a more satisfying shopping experience. From the retailer’s standpoint, personalized assortments enable strategic demand shaping, steering orders away from products in short supply in nearby warehouses and toward those with excessive stock.  Realizing these benefits requires coordinating several interdependent decisions. The retailer must decide how to allocate the products across warehouses, which assortment to present to each incoming customer, and from which warehouse to fulfill each chosen product, trading off shipping costs and inventory preservation for future demand. The resulting joint optimization problem captures the full cycle from inventory placement to customer choice and fulfillment.
To make the setting concrete, we now describe the sequence of decisions the retailer faces over the selling horizon. From the retailer's perspective, at the start of the selling horizon we choose how many units of each product to stock at every warehouse. During the selling horizon, different types of customers arrive stochastically. A customer type encodes any available information such as zip code, purchase history, or loyalty level. When a customer arrives, we observe their type and select a set of products to offer, based on current inventory levels at each warehouse, the customer's type, and time elapsed in the selling horizon. The customer then makes a purchase according to a probabilistic choice model and we fulfill the chosen product from a feasible warehouse of our choice, earning a profit that may depend not only on the product, but also on the customer type and warehouse location, reflecting varying shipping costs and personalized pricing effects. Profits may be negative when shipping costs exceed revenue. Our goal is to choose an initial stocking (inventory placement) plan and a policy for selecting assortments over time and making fulfillment decisions to maximize total expected profit over the selling horizon subject to supply and capacity constraints. Here, capacity constraints limit the total inventory stored at each warehouse, and supply constraints limit the total available units of each product across all warehouses. We take the total supply of each product as an exogenous input fixed before the selling season, for example by the team responsible for procurement decisions, so that the placement decision concerns how to distribute the given supply of each product across warehouses rather than how much to procure. We refer to this joint inventory \emph{placement}, assortment \emph{personalization}, and order \emph{fulfillment} problem as \ref{ppf}; see Section~\ref{sec:form} for the formal definition. We construct policies that maintain good performance under broad customer choice models and provide even stronger guarantees when choices follow the multinomial logit (MNL) model.

\paragraph{Fundamental challenges.}
Solving the \ref{ppf} problem comes with several fundamental challenges. First, computing the exact optimal assortment personalization and order fulfillment policy would require solving a large dynamic programming model that tracks inventory at every warehouse and is intractable for realistic problem sizes. Second, unlike most models, we study a setting where the profit of a sale depends simultaneously on the product purchased, the type of customer being served, and the warehouse that fulfills the order. This triadic heterogeneity is the principal source of technical difficulty because the same product can yield different profit depending on who buys it and where it is shipped from, and because each fulfillment decision depletes inventory at a specific warehouse and changes what can be profitably offered and fulfilled later. We must strategically balance allocation at each warehouse: stocking, assortment, and routing decisions are tightly interdependent over time, and the future marginal value of a unit of inventory, that is, the additional expected profit attributable to the last unit added over the remaining horizon, differs across products and fulfilling warehouses. This value also shifts as inventories deplete, because the worth of a unit depends on the availability of substitutes, of units at other warehouses, and on the geography-dependent shipping costs of serving customers from different warehouses, which makes the assortment and stocking decisions difficult. Tracking how inventory is depleted over time is substantially more involved than in models that assign revenue depending only on the product. Standard analysis shortcuts, such as uniformly scaling the fluid solution by a single common factor to control stock-out risk, no longer apply: even in the \emph{stationary} regime, i.e., when the customer arrival distribution is time-homogeneous, establishing performance guarantees requires nontrivial arguments to control how realized demand and purchases deviate from their fluid expectations across customer types and warehouses. In the general non-stationary regime, the difficulty is starker: natural randomized algorithms based on fluid relaxations, which suffice in simpler settings, can perform arbitrarily poorly because they ignore how profitability and feasibility shift as the horizon unfolds and inventories deplete. In addition, \emph{inventory-aware} policies that always fulfill products chosen by customers and are a widely used standard both in theory and industry are extremely hard to analyze directly in this setting due to complex intertwinement of inventory depletion and assortment and fulfillment decisions. This motivates us to first develop and analyze relaxed policies without the fulfillment guarantee, that is, without requiring that every offered product a customer chooses be fulfilled from available inventory, and then convert them into inventory-aware policies. Unlike standard settings, where removing unavailable products from an assortment only redirects demand toward other offered products and, because the revenue of a sale depends solely on the product sold, never lowers expected performance, such monotonicity is not apparent here, since profit also depends on the customer type and the fulfilling warehouse. Therefore, the conversion step requires a more sophisticated approach.

\subsection{Our Contributions}
We study the joint inventory placement, assortment personalization, and order fulfillment problem, \ref{ppf}, described above and develop algorithms and policies for it with provable performance guarantees and strong empirical performance. Our results are organized along two axes: whether customer arrivals are \emph{stationary}, meaning the arrival distribution of customer types is time-homogeneous over the selling horizon, or \emph{non-stationary}, and whether the underlying choice model is general or of the MNL type. Table~\ref{tab:main-results} summarizes the resulting approximation ratio guarantees for the \ref{ppf} problem. In particular, the general choice model row gives approximation ratios that approach one as the minimum capacity across warehouses and the minimum supply across products grow, while the MNL row reports explicit constant-factor guarantees, distinguishing the setting with only warehouse capacity constraints from the setting with both warehouse capacity and product supply constraints.

\begin{table}[htbp]
\centering
\caption{Approximation-ratio guarantees for the \ref{ppf} problem.}
\label{tab:main-results}
\footnotesize
\setlength{\tabcolsep}{4.2pt}
\renewcommand{\arraystretch}{1.5}
\begin{tabular}{l cc cc}
\toprule
\multirow{2}{*}{\textbf{Choice model}} & \multicolumn{2}{c}{\textbf{Stationary arrivals}} & \multicolumn{2}{c}{\textbf{Non-stationary arrivals}} \\
\cmidrule(lr){2-3}\cmidrule(lr){4-5}
 & WH & WH${+}$SPL & WH & WH${+}$SPL \\
\midrule
General
  & $1-2\sqrt{\tfrac{n}{K^{\min}}}$
  & $1-2\max\!\left(\sqrt{\tfrac{n}{K^{\min}}},\sqrt{\tfrac{L}{C_{\min}}}\right)$
  & $1-2\sqrt{\tfrac{nm}{K^{\min}}}$
  & $1-2\max\!\left(\sqrt{\tfrac{nm}{K^{\min}}},\sqrt{\tfrac{Lm}{C_{\min}}}\right)$ \\[2.6ex]
MNL
  & $0.199$ & $0.158$ & $0.101$ & $0.080$ \\
\bottomrule
\end{tabular}

\vspace{5pt}
\begin{minipage}{\textwidth}
{\footnotesize Here WH marks the setting with only warehouse capacity constraints and WH${+}$SPL the setting with both warehouse capacity and product supply constraints; $n$ is the number of products, $m$ the number of customer types, and $L$ the number of warehouses, while $K^{\min}=\min_{\ell\in\mathcal{L}}K^{\ell}$ is the minimum storage capacity across warehouses and $C_{\min}=\min_{i\in\mathcal{N}}C_i$ the minimum supply across products. Each entry is the fraction of the optimal expected profit our policies guarantee: in the general model these bounds approach $1$ as $K^{\min}$ and $C_{\min}$ grow, whereas in the MNL model they are the stated constants.\par}
\end{minipage}
\end{table}

\paragraph{Comparison to prior work.}
The guarantees in Table~\ref{tab:main-results} improve upon and unify several strands of prior work. The closest benchmark is \citet{baietal2025}, who study the single-warehouse problem with profits depending only on products. For general choice models they obtain a guarantee with a cubic-root convergence rate, which our results sharpen to a faster square-root rate while additionally accommodating multiple warehouses and customer-type-dependent profits. We are able to achieve this by introducing a novel joint construction that derives the stocking, assortment, and fulfillment decisions from a single fluid solution rather than sequentially. For the MNL model, \citet{baietal2025} give a $\tfrac{1}{4}(1-1/e)\approx 0.158$ constant, subsequently improved by a factor of two to $\approx 0.316$ by \citet{rajan2025}; our constants under stationary arrivals ($0.199$ without and $0.158$ with product supply constraints) are of comparable magnitude but hold in the substantially more general setting with multiple warehouses and profits that depend jointly on the product, the customer type, and the fulfilling warehouse. To handle this richer structure, we develop a sophisticated selective scaling procedure for fulfillment probabilities together with a triadic profit analysis.%

\noindent To obtain the guarantees for the \ref{ppf} problem, we build our results in several steps, each corresponding to a key component of the analysis. We now describe these components.

\paragraph{\bf Inventory-agnostic assortment and fulfillment policies (Section~\ref{sec:FixedStock}).}
In Section~\ref{sec:FixedStock}, we develop the key analytical and building blocks used throughout the paper. We treat the inventory placement as given, focusing on developing near-optimal assortment personalization and order fulfillment policies. We work with an \emph{inventory-agnostic} relaxation of the problem, in which the firm may offer out-of-stock products or decline to fulfill orders after customer choices. This relaxation is used as an analytical device: we produce clean policies that are later converted into feasible, inventory-aware ones. Within this relaxed environment, we construct three policies based on a fluid relaxation, which serves as an upper bound on the maximum expected profit. Section~\ref{sec:FixedStock} establishes that each of the three policies guarantees an expected profit that is at least a constant fraction of the fluid objective value. In the stationary regime, the \emph{Simple Inventory-Agnostic Policy} achieves a $(1-1/e)$-approximation guarantee relative to the fluid upper bound. At a high level, the policy uses the fluid solution to guide the assortment and routing decisions over time. The approximation guarantee proof tracks how routing and inventory co-evolve when profitability depends on both the arriving customer type and the chosen warehouse. In the non-stationary regime, this simple policy may perform arbitrarily badly. We therefore develop a novel profit-based scaling method that selectively suppresses low-value demand while preserving flows that can be routed to yield high profit. The scaling balances reductions across customer types and warehouse locations so that it remains resilient to future stock-outs. This approach significantly outperforms the naive strategy of uniformly scaling the fluid solution by a common factor chosen to control stock-out probability. The corresponding \emph{Advanced Inventory-Agnostic Policy} provides a $0.322$-approximation guarantee. Finally, we construct an asymptotically optimal policy, the \emph{Modified Inventory-Agnostic Policy}, which refines the scaling to drive the approximation ratio to one in the large inventory levels regime. These approximation guarantees serve as the foundation for the performance bounds proved under both general choice models and the MNL model.

\paragraph{\bf Asymptotically optimal guarantees under general choice models (Section~\ref{sec:joint}).}
Having developed assortment personalization and order fulfillment policies for a fixed inventory placement, we now bring the placement decision back into the picture and study the full joint problem. In this section, we develop approximation algorithms for the \ref{ppf} problem under a general choice model. Departing from the conventional two-stage approach, where one first selects an approximately optimal inventory placement and then uses an assortment and fulfillment policy with its own approximation guarantees, we base both the inventory placement decision and the dynamic assortment and fulfillment policy on an optimal solution of a joint fluid relaxation. Our construction leverages policies from Section~\ref{sec:FixedStock}. Unlike the analyses used in simpler single-warehouse settings such as \citet{baietal2025}, which bound inventory consumption as a fraction of the maximum stock, our analysis focuses on the absolute deviations between expected demand and the realized (random) cumulative inventory consumption across the horizon. This tighter coupling yields square-root convergence rates in the minimum warehouse capacity $K^{\min}$ and product supply $C_{\min}$, which is the benchmark in closely related revenue management settings, and significantly improves the cubic-root rate of \citet{baietal2025} obtained only for the single-warehouse setting with profits depending only on products.

\paragraph{\bf Constant-factor guarantees under the MNL choice model (Section~\ref{sec:mnljoint}).}
For the MNL choice model, we first optimize the inventory placement by introducing a surrogate objective that we prove to be monotone DR-submodular even with customer-type-dependent profits. This enables a provably good polynomial-time approximation using classical submodular optimization algorithms. We then apply assortment and fulfillment policies from Section~\ref{sec:FixedStock}. This framework yields strong approximation algorithms whose performance depends on whether arrivals are stationary and on whether the setting imposes only warehouse capacity constraints or both warehouse capacity and product supply constraints. Approximate numerical guarantees are showcased in Table~\ref{tab:main-results}.

\paragraph{\bf Conversion to inventory-aware policies (Section~\ref{sec:inventorybased}).}
As mentioned before, policies built in Section~\ref{sec:FixedStock} are inventory-agnostic: they may offer out-of-stock products or decline to fulfill orders after customer choices. To address this, Section~\ref{sec:inventorybased} presents a universal conversion that maps any inventory-agnostic policy from Section~\ref{sec:FixedStock} to a fully inventory-aware policy that never offers out-of-stock products and always fulfills any product it offers if chosen by the customer. Crucially, the conversion preserves expected profit. As a result, all guarantees proven in Sections~\ref{sec:joint} and~\ref{sec:mnljoint} translate directly into guarantees for feasible policies for the original \ref{ppf} problem. The key idea is to offer, in place of the assortment the inventory-agnostic policy would show, a randomized assortment over currently available products that matches the per-product purchase probabilities induced by that policy's intended assortment; since these probabilities are preserved exactly, so is the expected profit.

\paragraph{\bf Computational experiments (Section~\ref{sec:numerics}).}
To complement our theoretical guarantees, we conduct extensive computational experiments. These experiments validate our framework and show that our policies are highly effective in practice. The results demonstrate that our fluid-based joint optimization algorithms are empirically near-optimal, consistently achieving performance that approaches the fluid upper bound in large-scale settings. We also show that the empirical performance systematically improves as the problem scales. Here, scaling refers to increasing the horizon length and correspondingly scaling product inventory and warehouse capacities, which makes the fluid approximation more accurate and highlights the asymptotic behavior implied by our theoretical results.

\subsection{Related Literature}

Early work on joint assortment and inventory decisions dates back at least to \citet{honhon2010stockout} who study joint assortment and inventory decisions under a rank-based substitution model in which customers move down a customer-type-dependent preference list when higher-ranked products stock out. Several subsequent papers enrich the choice model at the cost of assuming uniform customer preferences. In \citet{topaloglu2013joint}, the retailer selects both stocking quantities and assortments over a finite selling horizon, and customers choose among the offered products according to the MNL model. \citet{mouchtaki2021mcc} consider joint assortment and inventory planning under the Markov chain choice model. These papers enlarge the class of choice models under which optimal joint assortment and inventory decisions can be tractably approximated but they do not incorporate customer-type-specific assortment personalization. By contrast, in a static two-stage assortment problem, \citet{elhousnitopaloglu2023} first select a cardinality-constrained set of products to carry and then, for a single random customer visit, offer a type-specific subset after observing the customer's type.

Our work is most closely connected to the relatively recent stream that couples offline stocking with personalized online offering, typically analyzed via a choice-based linear program. We generalize the setting from \citet{baietal2025}, where joint assortment optimization and stocking decision is studied with one warehouse and revenues dependent only on products. They introduce a novel framework for a constant factor approximation under the MNL choice model involving optimizing a DR-submodular surrogate to get near-optimal stocking quantities, which they then combine with a simple personalization policy with a provable constant-factor approximation guarantee. In addition, they construct the first asymptotically optimal approximation algorithm for general choice models with cubic-root convergence rate. \citet{rajan2025} improve the constant factor approximation guarantees of the assortment-personalization policy of \citet{baietal2025} and develop a more efficient stocking policy, improving the overall approximation guarantee by a factor of two. Building on that foundation, we study the same joint stocking-personalization paradigm but in a networked fulfillment setting where the per-sale profit depends jointly on the product, the arriving customer type, and the fulfillment location. We generalize the results above to this more complex setting and dramatically improve the approximation guarantee under general choice models.

Parallel work by \citet{HuangFeldmanZhang2024} studies the same joint assortment and inventory problem when resources are reusable; they show that similar surrogates and submodular structure yield constant-factor approximations in that setting under some usage-duration assumptions. Although our paper focuses on single-use units rather than reusable ones, their perspective underscores that the fluid-guided approach extends beyond the single-use case.

There is also a line of work that jointly optimizes assortment (or broader demand-shaping decisions) and fulfillment. \citet{lei2022framing} study a joint product-framing and fulfillment problem under the MNL choice model. They take the initial inventory distribution across multiple fulfillment centers as given. Over time, the retailer chooses how to display products and, should the customer choose some product, where to fulfill the order from. Their heuristic, derived from a deterministic relaxation, is shown to be asymptotically optimal. Their model highlights the benefit of coordinating demand shaping with fulfillment routing, though it does not consider offline placement decisions.

A related body of work studies order fulfillment in e-commerce networks in isolation: given realized demand, the retailer routes each accepted order across geographically dispersed fulfillment centers to control shipping and handling cost, often through linear-programming relaxations and rounding. \citet{acimovic2015making} propose a heuristic that makes real-time fulfillment decisions, using opportunity-cost estimates from a linear-programming relaxation to account for the value of reserving inventory for future demand, while \citet{jasin2015lp} develop an LP-based correlated rounding scheme with provable guarantees for multi-item fulfillment. These models take assortment, prices, and inventory placement as given and optimize routing alone; our setting embeds such a fulfillment decision but jointly optimizes it together with inventory placement and assortment personalization.

A growing body of literature studies joint placement and fulfillment. \citet{bai2025placement} analyze a multi-warehouse setting in which the retailer chooses an offline placement of inventory and then, upon each arrival, selects both a delivery promise, which influences conversion probability, and a fulfillment center. They obtain constant-factor and asymptotically optimal approximation guarantees. In related work, \citet{epstein2024placement} consider optimizing the placement of products across multiple warehouses in anticipation of a downstream online matching algorithm governing fulfillment. They show a constant-factor approximation to optimal joint placement-fulfillment performance for a fixed choice of the downstream online policy, so the remaining problem is to choose the offline placement that performs well under that given online policy. 

Separately, in work developed concurrently with and independently of ours, \citet{zuo2026joint} study joint inventory allocation and dynamic assortment optimization across multiple stores. Their joint model is a substantially restricted special case of ours: after the initial allocation, stores operate independently; each store serves a fixed number of customers who are homogeneous within the store; there is no customer-level personalization or fulfillment decision; and their joint results assume the MNL choice model. Exploiting this structure, they obtain guarantees for this special case and also study non-adaptive assortment policies.

Taken together, these strands of literature highlight the importance of coordinating assortment, inventory placement (stocking), and fulfillment decisions.  Existing models typically optimize at most two of these components at a time, rely on restrictive demand or network structures, or do not incorporate customer-type-dependent personalization.  Our work contributes by studying a very general joint inventory placement, assortment personalization, and order fulfillment setting and providing approximation guarantees under both the MNL and general choice models, often improving existing guarantees even though we consider a much more general problem setting.

\subsection{Outline}
In Section~\ref{sec:form}, we formulate our model. In Section~\ref{sec:FixedStock}, we develop assortment personalization and order fulfillment policies with approximation guarantees, which we then use in Sections~\ref{sec:joint} and~\ref{sec:mnljoint} to obtain approximation algorithms for the \ref{ppf} problem. In Section~\ref{sec:joint}, we provide approximation algorithms which are asymptotically optimal under a general choice model as product supplies and warehouse capacities scale. In Section~\ref{sec:mnljoint}, we build constant-factor approximation algorithms under the MNL choice model. In Section~\ref{sec:inventorybased}, we convert all of our assortment personalization and order fulfillment policies into inventory-aware ones; this conversion preserves the expected performance. In Section~\ref{sec:numerics}, we present computational experiments. In Section~\ref{sec:concl}, we conclude and discuss possible future research directions.

\section{Problem Formulation}\label{sec:form}

We consider a firm that sells $n$ products, indexed by $\mathcal{N} = \{1, \dots, n\}$, from a network of $L$ warehouses, indexed by $\mathcal{L} = \{1, \dots, L\}$. The firm serves $m$ distinct types of customers, indexed by $\mathcal{M} = \{1, \dots, m\}$. The profit from delivering product $i \in \mathcal{N}$ to a customer of type $j \in \mathcal{M}$ from warehouse $\ell \in \mathcal{L}$, denoted by $r_{ij}^{\ell}$, may depend on all three parameters, reflecting varying shipping costs which depend on the product itself, warehouse and customer location, incorporated in the customer type.

The problem unfolds over a finite selling horizon of $T$ discrete time periods, indexed by $\mathcal{T} = \{1, \dots, T\}$. In each time period $t \in \mathcal{T}$, a customer of type $j$ arrives with probability $\lambda_{jt}$; with probability $1 - \sum_{j \in \mathcal{M}} \lambda_{jt}$, no customer arrives. Arrivals are independent across time periods. Let $\tau_j = \sum_{t=1}^{T} \lambda_{jt}$ denote the expected number of type-$j$ customer arrivals over the horizon.

When a customer of type $j$ is offered an assortment of products $S \subseteq \mathcal{N}$, she chooses product $i \in S$ with probability $\phi_{ij}(S)$, given by a discrete choice model $\phi$. Choices of different customers are independent across time periods. Should the customer choose some product $i$, the firm must fulfill it from a warehouse $\ell$ that has positive inventory of product $i$, depleting the corresponding stock by one unit and earning a profit of $r_{ij}^{\ell}$, which is not required to be non-negative, as discussed in Section~\ref{sec:intro}. 

Throughout this work, we assume the choice model satisfies weak substitutability, meaning that adding a new product to an assortment does not increase the choice probability of existing products. Formally, for any assortment $S \subseteq \mathcal{N}$, any product $k \notin S$, and any product $i \in S$, we have $\phi_{ij}(S \cup \{k\}) \le \phi_{ij}(S)$. We also assume that for any assortment $S \subseteq \mathcal{N}$, and any product $k \notin S$, we have $\phi_{kj}(S) = 0$. These are standard properties satisfied by all random utility choice models (RUM), including the MNL model, which is of particular interest in our work.

The firm seeks to maximize its total expected profit over the selling horizon by making two sets of decisions: (i) a static stocking (inventory placement) decision at the beginning of the horizon, determining the quantity of each product to stock at each warehouse, and (ii) a dynamic policy for assortment personalization and order fulfillment. Assortment personalization refers to choosing which subset of products to offer to each arriving customer, while order fulfillment refers to deciding which warehouse ships a chosen product. The initial stocking decisions are subject to two types of constraints: each warehouse $\ell$ has a finite storage capacity of $K^{\ell}$ units, and each product $i$ has a total supply limit of $C_i$ units across all warehouses. The dynamic policy specifies which assortment to offer and from which warehouse to ship each product $i \in \mathcal{N}$, should it be chosen, for any customer arrival depending on the customer type, the time period, and the remaining inventory levels.

To formalize the dynamic policy problem for a fixed initial inventory vector, we construct a dynamic program. Let the state of the system at the beginning of time period $t$ be the vector of inventory levels $\boldsymbol{x} = (x_i^\ell : i \in \mathcal{N}, \ell \in \mathcal{L})$, where $x_i^\ell$ denotes the inventory level of product $i$ at warehouse $\ell$. In each time period $t$, if a customer of type $j$ arrives, the decision-maker takes an action $(S, f)$, where $S \subseteq \mathcal{N}$ is the assortment of products to offer and $f: S \to \mathcal{L}$ is a fulfillment function that maps each product in the assortment to a warehouse. A decision $(S,f)$ is feasible at the state $\boldsymbol{x}$ if there is positive inventory for each product at its assigned warehouse, i.e., $x_i^{f(i)} > 0$ for all $i \in S$. Let $I(\boldsymbol{x})$ denote the set of all such feasible decisions.

Let $J_t(\boldsymbol{x})$ be the maximum total expected profit from time periods $t, \dots, T$, given state $\boldsymbol{x}$ at the beginning of time period $t$. The value functions satisfy the following dynamic program:
\begin{align*}
J_t(\boldsymbol{x}) ={}& \sum_{j\in\mathcal{M}}\lambda_{jt} \max_{(S,f)\in I(\boldsymbol{x})} \left\{ \sum_{i\in S}\phi_{ij}(S)\left[r_{ij}^{f(i)}+J_{t+1}(\boldsymbol{x}-\boldsymbol{e}_i^{f(i)})\right] \right. \\
& \quad \left. +\left(1-\sum_{i\in S}\phi_{ij}(S)\right)J_{t+1}(\boldsymbol{x}) \right\} + \left(1-\sum_{j\in\mathcal{M}}\lambda_{jt}\right)J_{t+1}(\boldsymbol{x}),
\end{align*}
with the boundary condition $J_{T+1}(\boldsymbol{x}) = 0$ for all $\boldsymbol{x}$. Here, $\boldsymbol{e}_i^{\ell}$ denotes the unit vector corresponding to product $i$ at warehouse $\ell$. On the right side of the equation, a customer of type $j$ arrives in time period $t$ with probability $\lambda_{jt}$. If the firm offers assortment $S$ to her, she chooses product $i \in S$ with probability $\phi_{ij}(S)$ and one unit of product $i$ is depleted from warehouse $f(i)$, yielding a profit of $r_{ij}^{f(i)}$. The customer does not choose a product with probability $1-\sum_{i\in S}\phi_{ij}(S)$ and there is no customer arrival in time period $t$ with probability $1-\sum_{j\in\mathcal{M}}\lambda_{jt}$. In either of those cases, no inventory is consumed and the firm gains no profit.

By rearranging terms and factoring out $J_{t+1}(\boldsymbol{x})$, we can write the dynamic program more compactly as:
\begin{align*}
J_t(\boldsymbol{x}) = J_{t+1}(\boldsymbol{x}) + \sum_{j\in\mathcal{M}}\lambda_{jt} \max_{(S,f)\in I(\boldsymbol{x})}\left\{ \sum_{i\in S}\phi_{ij}(S)\left[r_{ij}^{f(i)} + J_{t+1}\left(\boldsymbol{x}-\boldsymbol{e}_i^{f(i)}\right) - J_{t+1}(\boldsymbol{x})\right] \right\}.
\end{align*}
The firm's overall problem is to choose an initial inventory vector $\boldsymbol{c} = (c_i^\ell : i \in \mathcal{N}, \ell \in \mathcal{L})$ and an inventory-aware assortment and fulfillment policy to maximize the total expected profit, subject to capacity and supply constraints. Therefore, the Joint Inventory Placement, Assortment Personalization, and Order Fulfillment Problem, which we refer to as  \ref{ppf}, can be stated as:
\begin{equation}\tag{\sf PPF}\label{ppf}
\text{OPT} = \max_{\boldsymbol{c} \in \mathbb{Z}_+^{nL}} \left\{ J_1(\boldsymbol{c}) \;\; : \;\; \sum_{i \in \mathcal{N}} c_i^\ell \le K^\ell \;\; \forall \ell \in \mathcal{L}, \;\; \sum_{\ell \in \mathcal{L}} c_i^\ell \le C_i \;\; \forall i \in \mathcal{N} \right\}.
\end{equation}

Here, $\text{OPT}$ denotes the optimal total expected profit of the \ref{ppf} problem, obtained by maximizing $J_1(\boldsymbol{c})$ over all feasible initial inventory vectors $\boldsymbol{c}$. The state space of the dynamic program grows exponentially with the number of products and warehouses, making the direct computation of $J_1(\boldsymbol{c})$ and the solution of the \ref{ppf} problem intractable for realistically sized problems. This computational challenge necessitates the development of approximation methods. Our approach, detailed in the subsequent sections, focuses on finding a near-optimal initial inventory placement $\boldsymbol{c}$ and a corresponding high-performance dynamic policy for assortment personalization and fulfillment. We assume the existence of a polynomial-time oracle for solving the single-period assortment optimization problem for the choice model: given a customer type $j$ and per-product rewards, it returns an assortment $S \subseteq \mathcal{N}$ that maximizes the customer's expected reward under $\phi$. This is a standard assumption satisfied by many choice models, including the multinomial logit model.

In the next sections, we separate solving the \ref{ppf} problem into making a near-optimal stocking decision $\boldsymbol{c}$ and building a near-optimal assortment personalization and fulfillment policy for it. However, instead of directly constructing a feasible policy, we first construct a policy for the {\it relaxed} problem, where the firm is not required to fulfill customer orders, which also implies that the firm is allowed to include out-of-stock products in offered assortments. These policies are then converted to those that fulfill any offered product if it is chosen by a customer, as required in the original problem, by using the technique from Section~\ref{sec:inventorybased}. The conversion preserves expected profit and thus maintains the performance guarantees.

Note that allowing the firm to decline fulfilling customer orders is equivalent to maintaining the fulfillment requirement but adding a virtual warehouse, labeled as $L+1$, with infinite stock of each product, and zero profit for all products and customer types. Declining an order in the original setting is equivalent to choosing to fulfill the order from the virtual warehouse in the relaxed setting. Throughout the paper, we consider this relaxed setting, denoting the extended set of warehouses as $\mathcal{L}_+ = \mathcal{L} \ \cup \{L+1\}$. We call assortment and fulfillment policies which are guaranteed not to use the virtual warehouse {\it inventory-aware}. These policies directly correspond to feasible policies in the original setting. Assortment and fulfillment policies without this guarantee are called {\it inventory-agnostic}. Working in this relaxed setting with warehouse set $\mathcal{L}_+$, let $J_1^*(\boldsymbol{c})$ be the optimal expected profit of the relaxed \ref{ppf} problem for a fixed placement $\boldsymbol{c}$, and let $\text{OPT}^* = \max_{\boldsymbol{c}} J_1^*(\boldsymbol{c})$ be its optimum over feasible placements. Restricting attention to inventory-aware policies in the relaxed setting gives exactly $J_1(\boldsymbol{c})$ and $\text{OPT}$, as defined earlier. Sections~\ref{sec:FixedStock}, \ref{sec:joint}, and \ref{sec:mnljoint} work with the relaxed, inventory-agnostic values $J_1^*(\boldsymbol{c})$ and $\text{OPT}^*$, and Section~\ref{sec:inventorybased} converts the resulting policies back to inventory-aware ones. By definition, $J_1(\boldsymbol{c}) \le J_1^*(\boldsymbol{c})$ and $\text{OPT} \le \text{OPT}^*$. In a slight abuse of notation, we refer to these relaxed problems simply as the assortment and fulfillment problem and the \ref{ppf} problem, dropping the word ``relaxed''.

In Section~\ref{sec:FixedStock}, we build inventory-agnostic assortment and fulfillment policies with provable performance guarantees. We then utilize these policies  to build approximation algorithms for the \ref{ppf} problem in Sections~\ref{sec:joint} and \ref{sec:mnljoint}. Finally, we show how to convert these policies into inventory-aware ones, while preserving the expected profit, in Section \ref{sec:inventorybased}.

\section{Inventory-Agnostic Assortment and Fulfillment Policies}\label{sec:FixedStock}

In this section, we focus on a key block in our approximation framework: developing inventory-agnostic assortment and fulfillment policies and establishing their performance guarantees. These policies are then converted to inventory-aware ones using the technique from Section \ref{sec:inventorybased}. As discussed in Section \ref{sec:form}, the dynamic program for the assortment and fulfillment problem is intractable to solve directly due to exponentially growing state space. Instead, we focus on developing approximately optimal policies. 

We treat the inventory placement $\boldsymbol{c}$ as given and fixed, establishing approximation guarantees that depend on whether customer arrivals are stationary and whether the minimum positive product inventory $c_{\min} = \min\{c_i^\ell : i\in\mathcal N,\ \ell\in\mathcal L,\ c_i^\ell>0\}$, the smallest number of units of any product stocked at any warehouse, is large. The all-zero placement is trivial, so we assume that $c_{\min}>0$. We use the lowercase $c_{\min}$ for this realized stocking quantity to distinguish it from the minimum product supply $C_{\min}$ and the minimum warehouse capacity $K^{\min}$. In the special case of stationary arrivals (i.e., $\lambda_{jt} = \lambda_j$ for all $t$), we develop a policy with the following performance guarantee.

\begin{theorem}[Performance Guarantee Under Stationarity]\label{thm:fixed stock stationary}
Under a general choice model and the stationarity assumption, we can compute a $\max\left\{1-\frac{1}{e}, 1-\frac{1}{\sqrt{c_{\min}}}\right\}$-approximate assortment and fulfillment policy in polynomial time.
\end{theorem}

Without the stationarity assumption, we establish the following result.

\begin{theorem}[Performance Guarantee Under Non-Stationary Arrivals]\label{thm:fixed stock general advanced}
Under a general choice model, we can compute a $\max\left\{0.322, 1-\frac{m+1}{\sqrt{c_{\min}}}\right\}$-approximate assortment and fulfillment policy in polynomial time.
\end{theorem}

As a reminder, we work in the relaxed setting with added virtual warehouse $L+1$, and both approximation guarantees are with respect to $J^*_1(\boldsymbol{c})$. Our approach consists of two key steps: (1) formulate a fluid approximation that provides an upper bound on the maximum expected profit, and (2) use its solution to design randomized inventory-agnostic policies.

\subsection{Fluid Approximation}
First, we relax the discrete stochastic problem of determining an optimal assortment and fulfillment policy into a continuous optimization problem where the decision variables represent expected flows rather than discrete decisions. Consider the following fluid approximation:
\begin{equation}\label{eq:fluid}
\text{LP}(\boldsymbol{c}) = \max_{\boldsymbol{w,y}\ge\bf{0}} \left\{ \begin{array}{l@{}l}
    \displaystyle\sum_{i\in\mathcal N}\sum_{j\in\mathcal M}\sum_{\ell\in\mathcal L_+} r^\ell_{ij}\,y^\ell_{ij}: \: & \displaystyle\sum_{j\in\mathcal M}y^\ell_{ij}\le c^\ell_i, \: \forall i\in\mathcal N,\;\ell\in\mathcal L; \\
    & \displaystyle\sum_{S\subseteq\mathcal N}\phi_{ij}(S)\,w_j(S) =\sum_{\ell\in\mathcal L_+}y^\ell_{ij}, \:  \forall i\in\mathcal N,\;j\in\mathcal M; \\
    & \displaystyle\sum_{S\subseteq\mathcal N}w_j(S)=\tau_j, \:  \forall j\in\mathcal M
\end{array} \right\}
\end{equation}

\noindent\textbf{Interpretation of Variables and Constraints.} In LP~\eqref{eq:fluid}, $y^{\ell}_{ij}$ represents the expected number of type-$j$ customers buying product $i$ served from warehouse $\ell$, and $w_j(S)$ represents the expected number of times we show assortment $S$ to type-$j$ customers. The first constraint ensures we cannot sell more in expectation than the available inventory, the second constraint ensures flow balance between expected assortment offerings and sales, and the third constraint ensures the total expected number of assortment offerings equals the expected number of arrivals for each customer type.

The number of decision variables in \eqref{eq:fluid} grows exponentially with the number of products but \eqref{eq:fluid} can be solved in polynomial time under the assortment optimization oracle assumption using column generation.  In the following lemma, proven in Appendix \ref{lpbound}, we formally establish that this LP provides an upper bound on the optimal expected profit.

\begin{lemma}\label{lem:lpbound}
For any inventory placement $\boldsymbol{c}$, we have $J^*_1(\boldsymbol{c}) \le \textup{LP}(\boldsymbol{c})$.
\end{lemma}

Throughout the paper, whenever we write $(\hat{\boldsymbol{w}}, \hat{\boldsymbol{y}})$, we refer to a feasible solution to \eqref{eq:fluid}. Without loss of generality, we further always assume $r^\ell_{ij} \hat{y}^\ell_{ij} \ge 0$ for all $i \in \mathcal{N},j \in \mathcal{M},\ell \in \mathcal{L}_+$; if $r^\ell_{ij} < 0$, we can set $\hat{y}^\ell_{ij} = 0$ and transfer its value to $\hat{y}^{L+1}_{ij}$, the variable corresponding to the expected sales of product $i$ to type-$j$ customers from virtual warehouse with corresponding profit of zero, without decreasing the objective value, while maintaining feasibility.

\subsection{Simple Inventory-Agnostic Assortment and Fulfillment Policy}
We prove Theorem~\ref{thm:fixed stock stationary} by constructing an inventory-agnostic probabilistic policy using a feasible solution to the fluid approximation \eqref{eq:fluid}, which we denote as $(\hat{\boldsymbol{w}}, \hat{\boldsymbol{y}})$. The policy's design follows a natural, two-stage randomization. First, upon the arrival of a type-$j$ customer, we sample an assortment $S$ from a distribution proportional to $\hat{w}_j(S)$. Second, if the customer chooses product $i$, we must select a warehouse for fulfillment. The variables $\hat{y}^\ell_{ij}$ from the fluid solution represent the target expected sales from each warehouse and serve as a natural guide for this decision; we therefore select warehouse $\ell$ with probability proportional to $\hat{y}^\ell_{ij}$.

Let $(\boldsymbol{w}^*, \boldsymbol{y}^*)$ be an optimal solution to \eqref{eq:fluid}. While to prove Theorem \ref{thm:fixed stock stationary} we simply use $(\hat{\boldsymbol{w}}, \hat{\boldsymbol{y}}) = (\boldsymbol{w}^*, \boldsymbol{y}^*)$, we use non-optimal feasible solutions when building an asymptotically optimal joint policy in Section \ref{sec:joint}. Moreover, as we show in Appendix \ref{app:simgen}, when using this policy in the case of non-stationarity, it is beneficial to use a scaled optimal solution: $(\hat{\boldsymbol{w}}, \hat{\boldsymbol{y}}) = \alpha(\boldsymbol{w}^*, \boldsymbol{y}^*)$, for a tuning parameter $\alpha \in (0, 1]$ (with a slight modification of increasing $\hat w_j(\varnothing)$ as needed to satisfy the third constraint in \eqref{eq:fluid}). Note that this scaled solution, after the modification, is still a feasible solution to \eqref{eq:fluid} but not necessarily an optimal one. Hence it is beneficial to formulate and analyze our policy in the more general case of using some feasible solution.

A key challenge in our approach is the risk of \textit{cannibalization}: myopically fulfilling an order might deplete stock that would be better reserved for a future, more profitable demand. As shown later in the paper this risk is significant in the non-stationary setting, where our policy can perform arbitrarily poorly when we choose $(\hat{\boldsymbol{w}}, \hat{\boldsymbol{y}})$ to be an unscaled optimal solution to \eqref{eq:fluid}. However, the stationarity assumption fundamentally changes the problem dynamics. It implies that the sequence of demands is probabilistically identical over time, which mitigates the risk of cannibalization, as there is no strategic benefit to reserving inventory for ``later" versus ``earlier" demands. This property allows our simple randomized policy to achieve a strong performance guarantee. We now formally define the policy.

\vspace{0.5\baselineskip}
\noindent\rule{\textwidth}{0.4pt}
\textbf{Policy 1: Simple Inventory-Agnostic Assortment and Fulfillment Policy}
\vspace{0.25\baselineskip}
\begin{itemize}[leftmargin=1.5em]
    \item \textbf{Input:} A feasible solution $(\hat{\boldsymbol{w}}, \hat{\boldsymbol{y}})$ to the fluid approximation \eqref{eq:fluid}.
    
    \item \textbf{Execution:} In each time period $t \in \mathcal{T}$, if a customer of type $j$ arrives:
        \begin{enumerate}[leftmargin=*,label=\arabic*.]
            \item \textbf{Offer Assortment:} Sample an assortment $\hat{S} \sim q_j$, where $q_j$ is the distribution over subsets $S \subseteq \mathcal{N}$ defined by $q_j(S) = \hat{w}_j(S)/\tau_j$, and offer $\hat{S}$ to the customer.
            
            \item \textbf{Fulfill Order:} If the customer chooses product $i \in \hat{S}$:
                \begin{enumerate}[leftmargin=*,label=(\alph*)]
                    \item \textbf{Select Warehouse:} Sample a warehouse $\ell \in \mathcal{L}_+$ from the fulfillment distribution $\rho_{ij}$ over $\mathcal{L}_+$ defined by:
                    \[ 
                    \rho_{ij}(\ell) = \frac{\hat{y}^\ell_{ij}}{\sum_{\ell' \in \mathcal{L}_+} \hat{y}^{\ell'}_{ij}}. 
                    \]
                    \item \textbf{Attempt Sale:} If the selected warehouse $\ell$ has a positive stock level for product $i$, fulfill the order. Otherwise, a stockout occurs, sell from the virtual warehouse instead.
                \end{enumerate}
        \end{enumerate}
\end{itemize}
\noindent\rule{\textwidth}{0.4pt}
\vspace{0.5\baselineskip}

\noindent Let us first analyze the expected demand generated by the policy. We define $Z^\ell_{ij}$ as the random variable for the total demand generated for product $i$ from (non-virtual) warehouse $\ell$ by type-$j$ customers over the selling horizon. Let $Z^\ell_i = \sum_{j \in \mathcal{M}} Z^\ell_{ij}$ be the total demand for product $i$ at warehouse $\ell$, which we call the \emph{$(i,\ell)$-demand}. Analogously, we let $\bar{Z}^\ell_{ij}$ denote the number of units of product $i$ actually sold to type-$j$ customers from warehouse $\ell$, and $\bar{Z}^\ell_i = \sum_{j \in \mathcal{M}} \bar{Z}^\ell_{ij}$ the total \emph{$(i,\ell)$-sales}; the fulfillment guarantees throughout this section bound $\bar{Z}^\ell_i$ from below, and we use this $(i,\ell)$-demand and $(i,\ell)$-sales shorthand from here on. The following lemma, proven in Appendix \ref{exp_dem_match} shows that the expected values of these random variables match the variables in the fluid solution. This result holds for both stationary and non-stationary arrivals.

\begin{lemma}[Expected Demand]\label{lem:expected_demand_match}
For any $i \in \mathcal{N}, j \in \mathcal{M}, \ell \in \mathcal{L}$, the expected demand under Policy 1 is given by $\mathbb{E}[Z^\ell_{ij}] = \hat{y}^\ell_{ij}$. Consequently, the total expected demand for product $i$ at warehouse $\ell$ is $\mathbb{E}[Z^\ell_i] = \sum_{j \in \mathcal{M}} \hat{y}^\ell_{ij}$.
\end{lemma}

\subsection {Proof of Theorem \ref{thm:fixed stock stationary}}
We now analyze the performance of Policy 1 under the stationarity assumption, i.e., $\lambda_{jt} = \lambda_j$ for all $t \in \mathcal{T}$. This assumption is pivotal because it implies that the demand process is probabilistically uniform over time, which greatly simplifies the analysis of expected profit lost to stockouts.

\subsubsection{Preliminary Lemmas}\label{prelem}
Since total expected profit decomposes into sum of expected profits from each product-warehouse pair, our analysis is going to focus on evaluating and bounding such expected profit. Let $\bar{Z}^\ell_{ij}$ be the random variable for the total number of units of product $i$ sold to type-$j$ customers from warehouse $\ell$. The total realized profit from this pair $(i,\ell)$ is the random variable $\bar{R}^\ell_i = \sum_{j \in \mathcal{M}} r^\ell_{ij} \bar{Z}^\ell_{ij}$. The following lemma provides an exact expression for $\mathbb{E}[\bar{R}^\ell_i]$, which is crucial for providing performance guarantees.

\begin{lemma}[Expected Profit Under Stationarity]\label{lem:expected_profit_formula}
Under the stationarity assumption, whenever $\mathbb{E}[Z_i^\ell]>0$, the expected total realized profit from product $i$ at warehouse $\ell$ is given by
\[
\mathbb{E}[\bar{R}^\ell_i] = \left(\sum_{j \in \mathcal{M}} r^\ell_{ij}\hat{y}^\ell_{ij}\right) \left( 1 - \frac{\mathbb{E}[(Z^\ell_i - c^\ell_i)^+]}{\mathbb{E}[Z^\ell_i]}\right).
\]
\end{lemma}

To prove this lemma, we first need to establish the conditional probability of an arriving customer's type, given that a demand occurs for product $i$ at warehouse $\ell$. Let $D_{jt}$ be the indicator that a customer of type $j$ arrives at time period $t$, and let $I^\ell_{it}$ be the indicator of a demand for product $i$ assigned to warehouse $\ell$ at time period $t$. 

\begin{claim}\label{lem:cond_prob}
Under stationarity, we have
\[
\mathbb{P}(D_{jt} = 1 \mid I^\ell_{it} = 1) = \frac{\hat{y}^\ell_{ij}}{\sum_{j' \in \mathcal{M}}\hat{y}^\ell_{ij'}}.
\]
\end{claim}
The proof of this claim can be found in Appendix \ref{claim3.6}
\begin{proof}[Proof of Lemma~\ref{lem:expected_profit_formula}]
In this proof, we treat product $i$ and warehouse $\ell$ as fixed and, for notational convenience, whenever we say ``demand'',``sale(s)'', or ``profit'', we only consider demand/sale(s) for product $i$ from warehouse $\ell$. Before we begin our proof, we define the following random variables, which are necessary for our proof.
\begin{itemize}
    \item $\bar{Z}^\ell_i = \sum_{j \in \mathcal{M}} \bar{Z}^\ell_{ij}$ is the total sales.
    \item $R^\ell_i$ is the total \textit{potential} profit if all demand is met.
    \item $R^\ell_{i,t}$ is the \textit{potential} profit from a demand at time period $t$. If the customer arriving at time period $t$ is of type $j$, the value of this variable is $r^{\ell}_{ij}$.
    \item $N^\ell_{ik}$ for $k=1, \dots, T$ is the time period at which the $k$-th demand occurs. If total demand is less than $k$, we set $N^\ell_{ik} = +\infty$.
\end{itemize}
Our goal is to find an expression for $\mathbb{E}[\bar{R}^\ell_i]$. We express the total realized profit as the total potential profit minus the total profit lost from demands that occur after inventory is depleted:
\[
\mathbb{E}[\bar{R}^\ell_i] = \mathbb{E}[R^\ell_i] - \mathbb{E}[\text{Total Lost Profit}].
\]
From Lemma~\ref{lem:expected_demand_match}, the expected potential profit is $\mathbb{E}[R^\ell_i] = \sum_{j \in \mathcal{M}} r^\ell_{ij} \mathbb{E}[Z^\ell_{ij}] = \sum_{j \in \mathcal{M}} r^\ell_{ij} \hat{y}^\ell_{ij}$. The main challenge is to evaluate the expected total lost profit. A loss only occurs if the total demand $Z^\ell_i$ exceeds the capacity $c^\ell_i$. Writing $k = Z^\ell_i$ for the realized total demand, when $k > c^\ell_i$ the demands from customers arriving at time periods $N^\ell_{i,c^\ell_i+1}, \dots, N^\ell_{ik}$ are unfulfilled. By the law of total expectation,
\[
\mathbb{E}[\text{Total Lost Profit}] = \sum_{k=c^\ell_i+1}^T \mathbb{P}(Z^\ell_i=k) \cdot \mathbb{E}\left[ \sum_{v=c^\ell_i+1}^k R^\ell_{i,N^\ell_{iv}} \biggm| Z^\ell_i=k \right].
\]
To evaluate the conditional expectation, we further condition on the exact ordered set of time periods $Q_k = \{q_{c^\ell_i+1}, \dots, q_k\}$ at which these unfulfilled demands occur. Let the event that these specific times occurred be $\mathcal{E}_{Q_k} := \{N^\ell_{iv}=q_v\}_{v=c^\ell_i+1}^k$. By the law of total expectation, averaging over all admissible orderings $Q_k$, we then have:
\begin{equation}\label{sum}
\mathbb{E}\left[ \sum_{v=c^\ell_i+1}^k R^\ell_{i,N^\ell_{iv}} \biggm| Z^\ell_i=k \right] = \sum_{Q_k} \mathbb{P}(\mathcal{E}_{Q_k} \mid Z^\ell_i=k) \sum_{v=c^\ell_i+1}^k \mathbb{E}\left[ R^\ell_{i,q_v} \mid Z^\ell_i=k, \mathcal{E}_{Q_k} \right].
\end{equation}
Under our policy and the stationarity assumption, the demand at any time period $t$ is an independent Bernoulli trial. As shown in the proof of Claim~\ref{lem:cond_prob}, the success probability is constant, $p = \mathbb{P}(I^\ell_{it}=1) = \frac{1}{T}\sum_{j \in \mathcal{M}} \hat{y}^\ell_{ij}$. The event $Z^\ell_i=k$ is equivalent to observing exactly $k$ successes in $T$ trials. Intuitively, because these trials are independent, conditioning on the total number of successes and the locations of \textit{all} successes does not alter the properties of a specific success. We now formalize this intuition and prove that for any $q_v \in Q_k$, the following equality holds:
\[
\mathbb{E}\left[R^\ell_{i,q_v} \mid Z^\ell_i=k, \mathcal{E}_{Q_k} \right] = \mathbb{E}\left[R^\ell_{i,q_v} \mid I^\ell_{i,q_v}=1 \right].
\]
To prove this equality rigorously, note that 
\[
\mathbb{E}\left[R^\ell_{i,q_v} \mid Z^\ell_i=k, \mathcal{E}_{Q_k} \right] = \sum_{j \in \mathcal{M}} r^\ell_{ij} \cdot \mathbb{P}(D_{j,q_v}=1 \mid Z^\ell_i=k, \mathcal{E}_{Q_k}).
\]
The event $\{Z^\ell_i=k, \mathcal{E}_{Q_k}\}$ tells us that a demand for product $i$ from warehouse $\ell$ occurs at time period $q_v$, i.e., $I^\ell_{i,q_v}=1$; it reveals nothing further about time period $q_v$, and any information it carries about the other time periods is irrelevant to the customer's type at $q_v$. Indeed, because the trials are independent across time periods and the customer's type at time period $q_v$ depends only on the random events at that time period, only the event $I^\ell_{i,q_v}=1$ is relevant to the distribution of the customer's type at $q_v$. Hence, we have
\[
\mathbb{P}(D_{j,q_v}=1 \mid Z^\ell_i=k, \mathcal{E}_{Q_k}) = \mathbb{P}(D_{j,q_v}=1 \mid I^\ell_{i,q_v}=1).
\] Substituting this back into the expectation gives:
\begin{align*}
\mathbb{E}\left[R^\ell_{i,q_v} \mid Z^\ell_i=k, \mathcal{E}_{Q_k} \right] &= \sum_{j \in \mathcal{M}} r^\ell_{ij} \cdot \mathbb{P}(D_{j,q_v}=1 \mid I^\ell_{i,q_v}=1) \\
&= \mathbb{E}\left[R^\ell_{i,q_v} \mid I^\ell_{i,q_v}=1 \right],
\end{align*} as desired.
This conditional expectation is the expected potential profit from a single demand, which, from Claim~\ref{lem:cond_prob}, is a constant value independent of time period:
\[
\mathbb{E}[R^\ell_{i,q_v} \mid I^\ell_{i,q_v}=1] = \sum_{j \in \mathcal{M}} r^\ell_{ij} \cdot \mathbb{P}(D_{j,q_v}=1 \mid I^\ell_{i,q_v}=1) =  \frac{\sum_{j \in \mathcal{M}} r^\ell_{ij}\hat{y}^\ell_{ij}}{\mathbb{E}[Z^\ell_i]}.
\]
Therefore, we have \[
\mathbb{E}\left[R^\ell_{i,q_v} \mid Z^\ell_i=k, \mathcal{E}_{Q_k} \right] = \frac{\sum_{j \in \mathcal{M}} r^\ell_{ij}\hat{y}^\ell_{ij}}{\mathbb{E}[Z^\ell_i]}.
\] Since this value is constant for each of the $k-c^\ell_i$ unfulfilled demands, the inner sum in \eqref{sum} becomes:
\[
\sum_{v=c^\ell_i+1}^k \mathbb{E}\left[ R^\ell_{i,q_v} \mid Z^\ell_i=k, \mathcal{E}_{Q_k} \right] = (k-c^\ell_i) \cdot \frac{\sum_{j \in \mathcal{M}} r^\ell_{ij}\hat{y}^\ell_{ij}}{\mathbb{E}[Z^\ell_i]}.
\]
Because this term is constant for any set of time periods $Q_k$, and since the probabilities sum to one ($\sum_{Q_k} \mathbb{P}(\mathcal{E}_{Q_k} \mid Z^\ell_i=k)=1$), the full expression simplifies to:
\[
\mathbb{E}\left[ \sum_{v=c^\ell_i+1}^k R^\ell_{i,N^\ell_{iv}} \biggm| Z^\ell_i=k \right] = (k-c^\ell_i) \cdot \frac{\sum_{j \in \mathcal{M}} r^\ell_{ij}\hat{y}^\ell_{ij}}{\mathbb{E}[Z^\ell_i]}.
\]
Substituting this back into our expression for the total expected lost profit, we get:
\begin{align*}
\mathbb{E}[\text{Total Lost Profit}] &= \sum_{k=c^\ell_i+1}^T \mathbb{P}(Z^\ell_i=k) (k-c^\ell_i) \frac{\sum_{j \in \mathcal{M}} r^\ell_{ij}\hat{y}^\ell_{ij}}{\mathbb{E}[Z^\ell_i]} \\
&= \frac{\sum_{j \in \mathcal{M}} r^\ell_{ij}\hat{y}^\ell_{ij}}{\mathbb{E}[Z^\ell_i]} \sum_{k=c^\ell_i+1}^\infty (k-c^\ell_i)\mathbb{P}(Z^\ell_i=k).
\end{align*}
The summation is the definition of $\mathbb{E}[(Z^\ell_i-c^\ell_i)^+]$. Finally, combining the potential and lost profit yields the lemma's statement:
\[
\mathbb{E}[\bar{R}^\ell_i] = \left(\sum_{j \in \mathcal{M}} r^\ell_{ij}\hat{y}^\ell_{ij}\right) - \left(\sum_{j \in \mathcal{M}} r^\ell_{ij}\hat{y}^\ell_{ij}\right)\frac{\mathbb{E}[(Z^\ell_i-c^\ell_i)^+]}{\mathbb{E}[Z^\ell_i]},
\]
which simplifies to the expression in the lemma.
\end{proof}

With the exact expression for the expected profit from Lemma~\ref{lem:expected_profit_formula}, the final step is to bound the fractional loss term $\mathbb{E}[(Z^\ell_i - c^\ell_i)^+] / \mathbb{E}[Z^\ell_i]$. We use the following three lemmas for this purpose.

\begin{lemma}[\cite{baietal2025}]\label{1e_sqrt}
Let $Z$ be a sum of independent Bernoulli random variables and let $c \ge \mathbb{E}[Z]$ be a positive integer. Then, $\mathbb{E}[(Z-c)^+] \le \frac{1}{\sqrt{c}}\mathbb{E}[Z]$.
\end{lemma}

\begin{lemma}\label{exp min bound}
For any inventory-agnostic policy under which the demand for product $i$ from warehouse $\ell$ arrives as a sequence of independent (but not necessarily identical) Bernoulli trials, the expected number of units of product $i$ sold from warehouse $\ell$, $\mathbb{E}[\bar{Z}_i^\ell]$, is lower-bounded by $\mathbb{E}[\bar{Z}_i^\ell] \ge c_i^\ell \left(1 - e^{-\mathbb{E}[Z_i^\ell]/c_i^\ell}\right)$.
\end{lemma}

\begin{lemma}\label{1e_e}
For the demand random variable $Z^\ell_i$ generated by Policy 1 with $(\boldsymbol{w}^*, \boldsymbol{y}^*)$, and for any initial stock $c_i^\ell \ge \mathbb{E}[Z_i^\ell]$, we have $\mathbb{E}[(Z^\ell_i - c^\ell_i)^+] \le \frac{1}{e} \mathbb{E}[Z^\ell_i]$.
\end{lemma}

Lemmas \ref{exp min bound} and \ref{1e_e} are closely related to Lemma 5 in \cite{rajan2025}; their proofs are provided in Appendix \ref{app:expmin} and \ref{app:1e}. We now have all the necessary components to conclude the proof of the theorem.

\subsubsection{Concluding the Proof of Theorem~\ref{thm:fixed stock stationary}}
First, as mentioned before, we set the feasible solution used in Policy 1 to be an optimal one. The total expected profit of the policy is the sum of the expected profit from each product-warehouse pair. Pairs with $\mathbb{E}[Z_i^\ell]=0$ contribute zero, so the sums below are restricted to pairs with $\mathbb{E}[Z_i^\ell]>0$. From Lemma~\ref{lem:expected_profit_formula}, we have:
\[
\mathbb{E}[\text{Total Profit}] = \sum_{i \in \mathcal{N}}\sum_{\ell \in \mathcal{L}}\left(\sum_{j \in \mathcal{M}} r^\ell_{ij}\hat{y}^\ell_{ij}\right) \left( 1 - \frac{\mathbb{E}[(Z^\ell_i - c^\ell_i)^+]}{\mathbb{E}[Z^\ell_i]}\right).
\]
We apply each tail bound separately. First, applying Lemma~\ref{1e_sqrt} provides one bound on the profit:
\[
\mathbb{E}[\text{Total Profit}] \ge \sum_{i \in \mathcal{N}}\sum_{\ell \in \mathcal{L}}\left(\sum_{j \in \mathcal{M}} r^\ell_{ij}\hat{y}^\ell_{ij}\right) \left(1 - \frac{1}{\sqrt{c_{\min}}}\right) = \left(1 - \frac{1}{\sqrt{c_{\min}}}\right) \text{LP}(\boldsymbol{c}).
\]
Second, applying Lemma~\ref{1e_e} provides another bound:
\[
\mathbb{E}[\text{Total Profit}] \ge \sum_{i \in \mathcal{N}}\sum_{\ell \in \mathcal{L}}\left(\sum_{j \in \mathcal{M}} r^\ell_{ij}\hat{y}^\ell_{ij}\right) \left(1 - \frac{1}{e}\right) = \left(1 - \frac{1}{e}\right) \text{LP}(\boldsymbol{c}).
\]
Since the policy's expected profit is guaranteed to be at least as large as both of these quantities, its performance relative to the fluid optimum is bounded below by their maximum. Given that $\text{LP}(\boldsymbol{c}) \ge J^*_1(\boldsymbol{c})$, this completes the proof.

\subsection{Proof of Theorem~\ref{thm:fixed stock general advanced}}\label{sec:proofadv}

Now we drop the stationarity assumption. In Section~\ref{sec:proofadv-partA} we construct a 0.322-approximation, while in Section~\ref{sec:proofadv-partB} we construct a $\left(1-\frac{m+1}{\sqrt{c_{\min}}}\right)$-approximation.

\subsubsection{Part A: The Constant Guarantee}\label{subsec:partA}\label{sec:proofadv-partA}
In the non-stationary setting, Policy 1 based on the optimal solution can perform arbitrarily poorly, as we demonstrate through an example in Appendix \ref{app:simgen}. There we also show that a uniform scaling of the optimal fluid solution, $(\hat{\boldsymbol{w}}, \hat{\boldsymbol{y}}) = \alpha(\boldsymbol{w}^*, \boldsymbol{y}^*)$ with $\alpha < 1$, while helping by reducing the cannibalization of high profit sales, can still provide a worst-case guarantee of only $1/4$, achieved at $\alpha = 1/2$, as shown in Lemma \ref{lem:simple_policy_bound}. This is because it reduces the target sales for high profit demands by the same factor as for low profit ones. The key idea behind our advanced policy is to be more selective. Instead of uniform scaling, we design a set of non-uniform, customer-product-specific fulfillment probabilities, $1-H^\ell_{ij}$. These factors are strategically calculated to reduce the likelihood of fulfilling low profit demands, thereby preserving scarce inventory for higher profit opportunities and leading to a better overall guarantee.

We now formalize the policy and the construction of $H^\ell_{ij}$.

\begin{sloppypar}
\vspace{0.5\baselineskip}
\noindent\rule{\textwidth}{0.4pt}
\textbf{Policy 2: Advanced Inventory-Agnostic Assortment and Fulfillment Policy}
\vspace{0.25\baselineskip}
\begin{itemize}[leftmargin=1.5em]
    \item \textbf{Input:} A feasible solution $(\hat{\boldsymbol{w}}, \hat{\boldsymbol{y}})$ to \eqref{eq:fluid} and pre-computed parameters $H^\ell_{ij} \in [0,1]$.
    \item \textbf{Execution:} In each time period $t \in \mathcal{T}$, if a customer of type $j$ arrives:
    \begin{enumerate}[leftmargin=*,label=\arabic*.]
        \item \textbf{Offer Assortment:} Sample an assortment $\hat{S} \sim q_j$, with $q_j$ the assortment distribution defined in Policy 1.
        \item \textbf{Fulfill Order:} If the customer chooses product $i \in \hat{S}$:
            \begin{enumerate}[leftmargin=*,label=(\alph*)]
                \item \textbf{Select Warehouse:} Choose a warehouse $\ell \in \mathcal{L}$ with probability $p_\ell$ given by:
                \[ p_\ell = \rho_{ij}(\ell)\,(1 - H^\ell_{ij}), \qquad \rho_{ij}(\ell) = \frac{\hat{y}^\ell_{ij}}{\sum_{\ell' \in \mathcal{L}_+} \hat{y}^{\ell'}_{ij}}. \]
                With the remaining probability $1-\sum_{\ell \in \mathcal{L}} p_\ell$, select the virtual warehouse.
                \item \textbf{Attempt Sale:} If the selected warehouse $\ell$ has a positive stock level for product $i$, fulfill the order. Otherwise, a stockout occurs, sell from the virtual warehouse instead.
            \end{enumerate}
    \end{enumerate}
\end{itemize}
\noindent\rule{\textwidth}{0.4pt}
\vspace{0.5\baselineskip}
\end{sloppypar}

\paragraph{Construction of the Scaling Parameters \texorpdfstring{$H^\ell_{ij}$}{H}.} Parameters $H^\ell_{ij} \in [0,1]$ are computed independently for each product-warehouse pair $(i,\ell) \in \mathcal{N} \times \mathcal{L}$. For a fixed pair $(i,\ell)$, we first order the customer types $j \in \mathcal{M}$ by their profit, such that $r^\ell_{ij_1} \le r^\ell_{ij_2} \le \dots \le r^\ell_{ij_m}$. For notational convenience, let $r = \sum_{k=1}^m r^\ell_{ij_k} \hat{y}^\ell_{ij_k}$ be the expected total potential profit  and $y = \sum_{k=1}^m \hat{y}^\ell_{ij_k}$ be the expected total demand for $(i,\ell)$ in the fluid solution. If $r=0$, the desired lower bound is zero and the claim is immediate; therefore, we may assume $r,y>0$. The construction uses two constants, $\alpha=0.3848$ and $\beta=0.524$, chosen to approximately maximize the resulting worst-case guarantee, as discussed at the end of this subsection.

First, we define a set of intermediate scaling parameters $h^\ell_{ij}$. Let $u$ be the smallest index for which the cumulative expected potential profit exceeds a fraction $\alpha$ of the total, i.e., $u = \min\{k \in \{1,\dots,m\} \mid \sum_{s=1}^k r^\ell_{ij_s} \hat{y}^\ell_{ij_s} > \alpha r\}$. We then set:
\[
h^\ell_{ij_k} = 
\begin{cases}
    1 & \text{if } k < u \\
    \frac{\alpha r - \sum_{s=1}^{u-1} r^\ell_{ij_s}\hat{y}^\ell_{ij_s}}{r^\ell_{ij_u}\hat{y}^\ell_{ij_u}} & \text{if } k = u \\
    0 & \text{if } k > u.
\end{cases}
\]
By design, these $h$-parameters precisely nullify the fulfillment probabilities for the least profitable demands such that their potential profit sums to $\alpha r$. To be precise, setting $H^\ell_{ij} = h^\ell_{ij}$ would scale the expected total potential profit exactly by a factor of $1-\alpha$ while reducing total expected demand as much as possible.

The three cases below split according to how much low-profit demand the scaling step suppresses. Recall that the $h$-parameters nullify the fulfillment probabilities for the least profitable demands, i.e., those below the profit threshold $u$ (the cutoff index at which the cumulative potential profit first reaches $\alpha r$); writing $\sum_{k=1}^m h^\ell_{ij_k}\hat{y}^\ell_{ij_k} = \beta'y$, the fraction $\beta'$ measures how much of the total expected demand this suppression accounts for. When $\beta'>\beta$, the suppression is large, so little demand survives relative to capacity and a stockout argument alone gives the guarantee. When $\beta'\le\beta$, the guarantee instead comes from the profit of the surviving demand: this profit is high enough on its own in one regime, whereas in the other we must lower the threshold to a smaller index $u'\le u$ and suppress demand in an all-or-nothing way, fully dropping the customer types below $u'$ and fully keeping the rest. These three regimes are formalized as Cases~1, 2, and~3 below.

It remains to define the final fulfillment-suppression parameters $H^\ell_{ij} \in [0,1]$, where $H^\ell_{ij}=1$ suppresses type $j$ for pair $(i,\ell)$ (its fulfillment probability is set to zero), $H^\ell_{ij}=0$ keeps it, and intermediate values suppress it proportionally; Cases~1 and~2 set $H^\ell_{ij}=h^\ell_{ij}$, which is in general fractional, while only Case~3 yields integral $H^\ell_{ij}\in\{0,1\}$. The two constants $\alpha=0.3848$ and $\beta=0.524$ are the fixed numbers introduced in the construction of the $h$-parameters above, and $\beta' = \tfrac{1}{y}\sum_{k=1}^m h^\ell_{ij_k}\hat{y}^\ell_{ij_k}$ is the fraction of the expected $(i,\ell)$-demand that those $h$-parameters already suppress. When $\beta'>\beta$, enough demand is suppressed on its own, and we keep $H^\ell_{ij}=h^\ell_{ij}$. When $\beta'\le\beta$, we instead decide which types to keep by comparing each type's per-sale profit $r^\ell_{ij}$ to a profit threshold and dropping the types that fall below it. The reference threshold is $\frac{\alpha r}{\beta' y}$, the average profit per unit of the demand suppressed by the $h$-parameters (a profit mass $\alpha r$ spread over a demand mass $\beta' y$). Whether the profit $r^\ell_{ij_u}$ of the type $j_u$ lies above or below this threshold determines whether we keep the $h$-based choice (Case 2) or re-select the cutoff (Case 3); to make these two options certify the same guarantee, we offset the threshold by the factors $1\pm\epsilon(\beta')$ built from the auxiliary functions
\[
\psi(x) = \frac{1-e^{x-1}}{1-e^{x/2-1}}, \qquad \epsilon(x) = \frac{1-\psi(x)}{1+\psi(x)}.
\]
Case 2 uses the raised threshold $(1+\epsilon(\beta'))\frac{\alpha r}{\beta'y}$ and Case 3 the lowered threshold $(1-\epsilon(\beta'))\frac{\alpha r}{\beta'y}$; in Case 3 we set $H^\ell_{ij_k}=1$ for exactly those low-profit types $j_k$ whose profit lies below the lowered threshold, and $H^\ell_{ij_k}=0$ for the rest. The offset $\epsilon(\beta')$ is chosen precisely so that the two resulting profit bounds, $(1+\epsilon(\beta'))\frac{\alpha r}{\beta'}(1-e^{\beta'-1})$ and $(1-\epsilon(\beta'))\frac{\alpha r}{\beta'}(1-e^{\beta'/2-1})$, are equal; their common value is $g(\beta')\,\alpha r$, where $g(x) = \frac{2}{1+\psi(x)}\frac{1-e^{x-1}}{x}$. We now make these three regimes precise as Cases~1, 2, and~3.
\vspace{0.2cm}

\noindent\textit{Case 1:} If $\beta' > \beta$, we set $H^\ell_{ij} = h^\ell_{ij}$ for all $j \in \mathcal{M}$.
\vspace{0.2cm}
    
\noindent\textit{Case 2:} If $\beta' \le \beta$ and $r^\ell_{ij_u} \ge (1+\epsilon(\beta'))\frac{\alpha r}{\beta'y}$, we set $H^\ell_{ij} = h^\ell_{ij}$ for all $j \in \mathcal{M}$.
\vspace{0.2cm}
    
\noindent\textit{Case 3:} If $\beta' \le \beta$ and $r^\ell_{ij_u} < (1+\epsilon(\beta'))\frac{\alpha r}{\beta'y}$, let $u'$ be the smallest index such that $r^\ell_{ij_{u'}} \ge (1-\epsilon(\beta'))\frac{\alpha r}{\beta'y}$. We set $H^\ell_{ij_k} = 1$ for $k < u'$, and $H^\ell_{ij_k} = 0$ for $k \ge u'$.

\paragraph{Performance Analysis.}
Here $Z^\ell_{ij}$ and $\bar{Z}^\ell_{ij}$ are the $(i,\ell)$-demand and $(i,\ell)$-sales from type-$j$ customers, now generated under Policy 2, in the notation introduced above. The expected profit of the policy from product-warehouse pair $(i,\ell)$ is equal to $ \sum_{j \in \mathcal{M}} r^\ell_{ij} \mathbb{E}[\bar{Z}^\ell_{ij}]$. The expected $(i,\ell)$-demand from type-$j$ customers is $\mathbb{E}[Z^\ell_{ij}] = \hat{y}^\ell_{ij}(1-H^\ell_{ij})$. We now show that for any $(i,\ell)$ the expected $(i,\ell)$ profit is lower-bounded by $0.322 \sum_{j \in \mathcal{M}} r^\ell_{ij} \hat{y}^\ell_{ij}$ by analyzing each of the three cases.

\vspace{0.5\baselineskip}
\noindent\textit{Case 1 ($ \beta' > \beta $):}
\vspace{0.2cm}

In this case, $H^\ell_{ij} = h^\ell_{ij}$. The total expected demand for $(i,\ell)$ is:
\[
\mathbb{E}[Z_i^\ell] = \sum_{j \in \mathcal{M}} \hat{y}^\ell_{ij}(1-h^\ell_{ij}) = y - \sum_{j \in \mathcal{M}} h^\ell_{ij}\hat{y}^\ell_{ij} = y(1 - \beta').
\]
Since $\beta' > \beta$ by the case condition, we have $\mathbb{E}[Z_i^\ell] < y(1-\beta)$. As the total fluid demand $y$ cannot exceed the capacity $c_i^\ell$, it follows that $\mathbb{E}[Z_i^\ell] < (1-\beta)c_i^\ell$. The probability of a stockout is therefore bounded: by Markov's inequality, $\mathbb{P}(Z_i^\ell \ge c_i^\ell) \le \frac{\mathbb{E}[Z_i^\ell]}{c_i^\ell} < 1-\beta$, which implies that the probability of not stocking out is $\mathbb{P}(Z_i^\ell < c_i^\ell) \ge \beta$.

Let $A_{it}^\ell$ be the indicator that product $i$ has non-empty stock at warehouse $\ell$ at the beginning of time period $t$. The expected sales to customer type $j$ can be expressed as:
\[ 
\mathbb{E}[\bar{Z}_{ij}^\ell] = \sum_{t \in \mathcal{T}} \lambda_{jt}\sum_{S \subseteq \mathcal{N}}\frac{\hat{w}_j(S)}{\tau_j}\phi_{ij}(S)\frac{\hat{y}^{\ell}_{ij}}{\sum_{b \in \mathcal{L}_+} \hat{y}^b_{ij}}(1-h^{\ell}_{ij}) \mathbb{E}[A^{\ell}_{it}]. 
\]
The probability of having stock at the beginning of time period $t$, $\mathbb{E}[A_{it}^\ell]$, is at least the probability of not stocking out over the entire horizon. Thus, $\mathbb{E}[A_{it}^\ell] \ge \mathbb{P}(Z_i^\ell < c_i^\ell) \ge \beta$. The whole summation in the expression for $\mathbb{E}[\bar{Z}_{ij}^\ell]$ without $\mathbb{E}[A_{it}^\ell]$ factor is equal to $\mathbb{E}[Z^\ell_{ij}]$, similarly to what we have proven before. This gives a lower bound on expected sales:
\[
\mathbb{E}[\bar{Z}^\ell_{ij}] \ge \mathbb{E}[Z^\ell_{ij}] \cdot \beta = \hat{y}^\ell_{ij}(1-h^\ell_{ij})\beta.
\]
The total expected profit for $(i,\ell)$ is therefore bounded as follows:
\begin{align*}
\sum_{j \in \mathcal{M}} r^\ell_{ij} \mathbb{E}[\bar{Z}^\ell_{ij}] &\ge \sum_{j \in \mathcal{M}} r^\ell_{ij} \hat{y}^\ell_{ij}(1-h^\ell_{ij})\beta
= \beta \left(\sum_{j \in \mathcal{M}} r^\ell_{ij}\hat{y}^\ell_{ij} - \sum_{j \in \mathcal{M}} r^\ell_{ij}h^\ell_{ij}\hat{y}^\ell_{ij}\right).
\end{align*}
By construction of the $h$-parameters, $\sum_{j \in \mathcal{M}} r^\ell_{ij}h^\ell_{ij}\hat{y}^\ell_{ij} = \alpha r = \alpha \sum_{j \in \mathcal{M}} r^\ell_{ij}\hat{y}^\ell_{ij}$. Substituting this in, we get:
\[
\sum_{j \in \mathcal{M}} r^\ell_{ij} \mathbb{E}[\bar{Z}^\ell_{ij}] \ge \beta(1-\alpha)\sum_{j \in \mathcal{M}} r^\ell_{ij}\hat{y}^\ell_{ij} \ge 0.322 \sum_{j \in \mathcal{M}} r^\ell_{ij}\hat{y}^\ell_{ij},
\]
which completes the proof for Case 1.

\vspace{0.5\baselineskip}
\noindent\textit{Case 2 ($ \beta' \le \beta $ and $ r^\ell_{ij_u} \ge (1+\epsilon(\beta'))\frac{\alpha r}{\beta'y} $):}
\vspace{0.2cm}

Here, $H^\ell_{ij} = h^\ell_{ij}$. By construction, $h^\ell_{ij_k}=1$ for all $k<u$, so the fulfillment probability for these low-profit customer types is zero. The profit is therefore derived only from customers of type $j_k$ for $k \ge u$. Since $r^\ell_{ij_k}$ is non-decreasing in $k$, the expected total profit is lower-bounded by noting that all sales generate a profit of at least $r^\ell_{ij_u}$:
\[
\sum_{j \in \mathcal{M}} r^\ell_{ij} \mathbb{E}[\bar{Z}^\ell_{ij}] = \sum_{k=u}^m r^\ell_{ij_k} \mathbb{E}[\bar{Z}^\ell_{ij_k}] \ge r^\ell_{ij_u} \sum_{k=u}^m \mathbb{E}[\bar{Z}^\ell_{ij_k}] = r^\ell_{ij_u} \mathbb{E}[\bar{Z}_i^\ell],
\]
where $\bar{Z}_i^\ell$ are total $(i,\ell)$-sales.
Using Lemma~\ref{exp min bound}, the fact that the function $f(x)=\frac{1-e^{-x}}{x}$ is decreasing, and that, like in Case 1, $\mathbb{E}\left[Z^{\ell}_i\right] = y(1-\beta')$ and therefore $\mathbb{E}\left[Z^{\ell}_i\right]/c^{\ell}_i \le 1 - \beta'$, we can bound the expected sales:
\[
\mathbb{E}[\bar{Z}_i^\ell] \ge c_i^\ell(1-e^{-\mathbb{E}[Z_i^\ell]/c_i^\ell}) \ge \mathbb{E}[Z_i^\ell] \cdot \frac{1-e^{-(1-\beta')}}{1-\beta'} = y(1-\beta')\frac{1-e^{-(1-\beta')}}{1-\beta'} = y(1-e^{\beta'-1}).
\]
Combining these results with the condition on $r^\ell_{ij_u}$ from the case definition gives:
\begin{align*}
\sum_{j \in \mathcal{M}} r^\ell_{ij} \mathbb{E}[\bar{Z}^\ell_{ij}] &\ge r^\ell_{ij_u} \cdot y(1-e^{\beta'-1}) \\
&\ge \left((1+\epsilon(\beta'))\frac{\alpha r}{\beta'y}\right) \cdot y(1-e^{\beta'-1}) \\
&= \left(\frac{1-\psi(\beta')}{1+\psi(\beta')} + 1\right)\frac{\alpha r}{\beta'} (1-e^{\beta'-1}) \\
&= \frac{2}{1+\psi(\beta')}\frac{\alpha r}{\beta'} (1-e^{\beta'-1}).
\end{align*}
Here, the second inequality follows from the case condition, and the first equality follows from the definition of $\epsilon$. The bound above can be written as $g(\beta') \cdot \alpha r$. A direct computation shows that $g(x)$ is a decreasing function for $x \in (0,1]$. Since the case condition is $\beta' \le \beta$, it follows that $g(\beta') \ge g(\beta)$. Therefore,
\[
\sum_{j \in \mathcal{M}} r^\ell_{ij} \mathbb{E}[\bar{Z}^\ell_{ij}] \ge g(\beta')\alpha r \ge g(\beta)\alpha r \ge 0.322 r = 0.322 \sum_{j \in \mathcal{M}} r^\ell_{ij}\hat{y}^\ell_{ij}.
\]
This concludes the proof for Case 2.

\vspace{0.5\baselineskip}
\noindent\textit{Case 3 ($ \beta' \le \beta $ and $ r^\ell_{ij_u} < (1+\epsilon(\beta'))\frac{\alpha r}{\beta'y} $):}
\vspace{0.2cm}

In this case, we set $H^\ell_{ij_k}=1$ for $k < u'$ and $H^\ell_{ij_k}=0$ for $k \ge u'$, zeroing out the demand from the lowest profit customers (types $j_1, \dots, j_{u'-1}$). The remaining profit is driven by customers of type $j_{u'}$ and higher. An important property in this case is that since the profits $r^\ell_{ij_k}$ are ordered, the construction of $h^\ell_{ij_k}$ implies a lower bound on $r^\ell_{ij_u}$:
\[ \alpha r = \sum_{k=1}^u r^\ell_{ij_k} h^\ell_{ij_k} \hat{y}^\ell_{ij_k} \le r^\ell_{ij_u} \sum_{k=1}^u h^\ell_{ij_k} \hat{y}^\ell_{ij_k} = r^\ell_{ij_u} \beta' y \implies r^\ell_{ij_u} \ge \frac{\alpha r}{\beta' y}. \]
This directly implies $u' \le u$. We first provide a claim about the amount of demand reduction from the intermediate $h$-parameters, proven in Appendix \ref{app:claimadv}.
\begin{claim}\label{claimadv}
$\sum_{s=u'}^u h^\ell_{ij_s}\hat{y}^\ell_{ij_s} \ge \frac{1}{2}\beta'y$.
\end{claim}

The total expected demand under the new parameters $H$ is $\mathbb{E}[Z_i^\ell] = \sum_{k=u'}^m \hat{y}^\ell_{ij_k}$. Let this sum be equal to $(1-\beta'')y$. To bound this term, we use the result from the claim. By definition, $\sum_{k=1}^m h^\ell_{ij_k} \hat{y}^\ell_{ij_k} = \beta'y$. Since $h^\ell_{ij_k}=1$ for $k < u$ (and $u' \le u$), this can be decomposed as:
\[
\beta'y = \sum_{k=1}^{u'-1} \hat{y}^\ell_{ij_k} + \sum_{k=u'}^u h^\ell_{ij_k} \hat{y}^\ell_{ij_k}.
\]
The claim states that $\sum_{k=u'}^u h^\ell_{ij_k} \hat{y}^\ell_{ij_k} \ge \beta'y/2$. This implies that the expected demand reduction from the lowest profit customer types is bounded: $\sum_{k=1}^{u'-1} \hat{y}^\ell_{ij_k} \le \beta'y/2$. The remaining expected demand is therefore:
\[
(1-\beta'')y = y - \sum_{k=1}^{u'-1} \hat{y}^\ell_{ij_k} \ge y - \beta'y/2 = y(1-\beta'/2).
\]
In particular, this gives $\beta''\le\beta'/2$. Because the customer types are ordered by increasing per-sale profit and Case 3 retains exactly the types $j_{u'},\dots,j_m$, every remaining sale earns at least $r^\ell_{ij_{u'}}$, the per-sale profit of the least profitable retained type $j_{u'}$. Following the logic of the Case 2 analysis, we have:
\[
\sum_{j \in \mathcal{M}} r^\ell_{ij} \mathbb{E}[\bar{Z}^\ell_{ij}] \ge r^\ell_{ij_{u'}} \mathbb{E}[\bar{Z}_i^\ell] \ge r^\ell_{ij_{u'}} y(1-e^{\beta''-1}) \ge r^\ell_{ij_{u'}} y(1-e^{\beta'/2-1}).
\]
The last inequality holds because $\beta''\le\beta'/2$, as established above, and $x \mapsto 1-e^{x-1}$ is decreasing, so the smaller argument $\beta''$ yields the larger value. Using the condition on $r^\ell_{ij_{u'}}$ from this case's definition, we obtain the final bound:
\begin{align*}
\sum_{j \in \mathcal{M}} r^\ell_{ij} \mathbb{E}[\bar{Z}^\ell_{ij}] &\ge (1-\epsilon(\beta'))\frac{\alpha r}{\beta'y} \cdot y(1-e^{\beta'/2-1}) 
= g(\beta')\alpha r  \ge g(\beta)\alpha r \ge 0.322r,
\end{align*}
where $g$ is a decreasing function for $x \in (0,1]$, as stated in Case~2. This concludes the proof for Case 3.

\vspace{0.3cm}

When Policy 2 is applied to an optimal solution of \eqref{eq:fluid}, summing the pairwise guarantees gives a total expected profit of at least $0.322 \cdot \text{LP}(\boldsymbol{c}) \ge 0.322 J^*_1(\boldsymbol{c})$.

\vspace{0.1cm}

{\bf Remark.} Parameters $\alpha$ and $\beta$ were chosen by a fine grid search to approximately maximize $\min(\beta(1-\alpha), g(\beta)\alpha)$, with the goal of obtaining the best possible uniform performance guarantee across all three cases.

\subsubsection{Part B: The Asymptotic Guarantee}\label{subsec:partB}\label{sec:proofadv-partB}

We now construct a $\left(1-\frac{m+1}{\sqrt{c_{\min}}}\right)$-approximate policy. 

\paragraph{Policy with Customer-Specific Capacities.}
The core idea for this case is  to create virtual, customer-specific inventory capacities, $c^\ell_{ij}$, that are slightly larger than the expected $(i,\ell)$ demands from that customer type, $\hat{y}^\ell_{ij}$. This allows us to use a powerful concentration inequality (Lemma~\ref{main ineq}) to show that stockouts for any specific customer type are rare.

The asymptotic guarantee is nontrivial only if $\sqrt{c_{\min}}>m+1$, so we assume this condition below; otherwise, Policy 2 applies.

First, we scale down the optimal fluid solution $(\boldsymbol{w}^*, \boldsymbol{y}^*)$ by a factor that depends on $c_{\min}$:
\[
(\hat{\boldsymbol{w}}, \hat{\boldsymbol{y}}) = \left(1 - \frac{m}{\sqrt{c_{\min}}}\right) (\boldsymbol{w}^*, \boldsymbol{y}^*).
\]
Here, as before, with scaling we also make a slight modification of increasing $\hat w_j(\varnothing)$ as needed to preserve feasibility. Next, for each triple $(i,j,\ell)$ with $i \in \mathcal{N}, j \in \mathcal{M}, \ell \in \mathcal{L}$, we define a virtual capacity $c^\ell_{ij} = \lfloor \hat{y}^\ell_{ij} + \sqrt{c_{\min}} \rfloor$ if $c_i^\ell>0$, and set $c^\ell_{ij}=0$ otherwise. Next, we show that these capacities satisfy the original physical constraints. This is immediate when $c_i^\ell=0$. Otherwise, $c_i^\ell\ge c_{\min}$, and since $\lfloor x \rfloor \le x$, we have:
\begin{align*}
\sum_{j \in \mathcal{M}} c^\ell_{ij} &\le \sum_{j \in \mathcal{M}} (\hat{y}^\ell_{ij} + \sqrt{c_{\min}}) = \left(\sum_{j \in \mathcal{M}} \hat{y}^\ell_{ij}\right) + m\sqrt{c_{\min}} \\
&= \left(1 - \frac{m}{\sqrt{c_{\min}}}\right)\left(\sum_{j \in \mathcal{M}} y^{*\ell}_{ij}\right) + m\sqrt{c_{\min}} \\
&\le \left(1 - \frac{m}{\sqrt{c_{\min}}}\right)c^\ell_i + m\sqrt{c_{\min}} \\
&= c^\ell_i - \frac{m c^\ell_i}{\sqrt{c_{\min}}} + m\sqrt{c_{\min}} \le c^\ell_i.
\end{align*}
The policy we introduce is a modification of Policy 1, using these new virtual capacities.

\vspace{0.5\baselineskip}
\noindent\rule{\textwidth}{0.4pt}
\textbf{Policy 3: Modified Simple Inventory-Agnostic Assortment and Fulfillment Policy}
\vspace{0.25\baselineskip}
\begin{itemize}[leftmargin=1.5em]
    \item \textbf{Assumption:} For each product $i$, physical warehouse $\ell$, and customer type $j$, there is a virtual stock level of $c^\ell_{ij}$.
    \item \textbf{Execution:} In each time period $t$, if a type-$j$ customer arrives, the policy follows the same steps as Policy 1 with $(\hat{\boldsymbol{w}}, \hat{\boldsymbol{y}})$, but a (non-virtual) sale is made only if the chosen warehouse $\ell$ has a positive stock of product $i$ for customer type $j$. This stock is depleted by one in the event of sale.
\end{itemize}
\noindent\rule{\textwidth}{0.4pt}

\paragraph{Performance Analysis.} The analysis of this policy relies on the following concentration inequality, proven in Appendix \ref{app:conc}.

\begin{lemma}\label{main ineq}
Let $Z$ be a non-zero sum of independent $[0,1]$-valued random variables. If $c >\mathbb{E}[Z]$, then $\frac{\mathbb{E}[(Z-c)^+]}{\mathbb{E}[Z]} \le \frac{1}{4(c -\mathbb{E}[Z])}$.
\end{lemma}

Under the modified policy, sales are fulfilled from the customer-specific capacities. The expected profit is:
\[
\mathbb{E}[\text{Total Profit}] = \sum_{i \in \mathcal{N}}\sum_{j\in \mathcal{M}}\sum_{\ell \in \mathcal{L}} r^\ell_{ij} \mathbb{E}[\bar{Z}^\ell_{ij}] = \sum_{i \in \mathcal{N}}\sum_{j\in \mathcal{M}}\sum_{\ell \in \mathcal{L}} r^\ell_{ij} \left(\mathbb{E}[Z^\ell_{ij}] - \mathbb{E}[(Z^\ell_{ij}-c^\ell_{ij})^+]\right).
\]
Similar to Lemma \ref{lem:expected_demand_match}, we have $\mathbb{E}[Z^\ell_{ij}]=\hat{y}^\ell_{ij}$. We can rewrite the expected total profit as:
\[
\mathbb{E}[\text{Total Profit}] = \sum_{i \in \mathcal{N}}\sum_{j\in \mathcal{M}}\sum_{\ell \in \mathcal{L}} r^\ell_{ij}\hat{y}^\ell_{ij} \left(1 - \frac{\mathbb{E}[(Z^\ell_{ij}-c^\ell_{ij})^+]}{\hat{y}^\ell_{ij}}\right),
\]
where terms with $\hat{y}^\ell_{ij}=0$ contribute zero and are omitted.
We use Lemma~\ref{main ineq} to bound the multipliers. Note that $c^\ell_{ij} - \mathbb{E}[Z^\ell_{ij}] = \lfloor \hat{y}^\ell_{ij} + \sqrt{c_{\min}} \rfloor - \hat{y}^\ell_{ij} \ge \hat{y}^\ell_{ij} + \sqrt{c_{\min}} - 1 - \hat{y}^\ell_{ij} = \sqrt{c_{\min}}-1$. For $c_{\min} \ge 4$ (and otherwise the bound is trivial), this is at least $\sqrt{c_{\min}}/2$. Applying the lemma:
\[
\frac{\mathbb{E}[(Z^\ell_{ij}-c^\ell_{ij})^+]}{\mathbb{E}[Z^\ell_{ij}]} \le \frac{1}{4(c^\ell_{ij}-\mathbb{E}[Z^\ell_{ij}])} \le \frac{1}{4(\sqrt{c_{\min}}-1)} \le \frac{1}{2\sqrt{c_{\min}}}.
\]
Substituting this bound back into the expected profit expression and using the definition of $\hat{\boldsymbol{y}}$ gives:
\begin{align*}
\mathbb{E}[\text{Total Profit}] &\ge \sum_{i \in \mathcal{N}}\sum_{j\in \mathcal{M}}\sum_{\ell \in \mathcal{L}} r^\ell_{ij}\hat{y}^\ell_{ij} \left(1 - \frac{1}{2\sqrt{c_{\min}}}\right) \\
&= \left(1 - \frac{1}{2\sqrt{c_{\min}}}\right)  \sum_{i \in \mathcal{N}}\sum_{j\in \mathcal{M}}\sum_{\ell \in \mathcal{L}}  r^\ell_{ij}\hat{y}^\ell_{ij} \\
&= \left(1 - \frac{1}{2\sqrt{c_{\min}}}\right)\left(1 - \frac{m}{\sqrt{c_{\min}}}\right)  \sum_{i \in \mathcal{N}}\sum_{j\in \mathcal{M}}\sum_{\ell \in \mathcal{L}}  r^\ell_{ij}y^{*\ell}_{ij}.
\end{align*}
Since $\sum_{i \in \mathcal{N},j \in \mathcal{M},\ell \in \mathcal{L}} r^\ell_{ij}y^{*\ell}_{ij} = \text{LP}(\boldsymbol{c})$, and for any $a,b\ge0$ we have $(1-a)(1-b) \ge 1-a-b$, we deduce
\[
\left(1 - \frac{1}{2\sqrt{c_{\min}}}\right)\left(1 - \frac{m}{\sqrt{c_{\min}}}\right) \ge 1 - \frac{m}{\sqrt{c_{\min}}} - \frac{1}{2\sqrt{c_{\min}}} > 1 - \frac{m+1}{\sqrt{c_{\min}}}.
\]
This gives $\mathbb{E}[\text{Total Profit}] \ge (1-\frac{m+1}{\sqrt{c_{\min}}}) \text{LP}(\boldsymbol{c}) \ge (1-\frac{m+1}{\sqrt{c_{\min}}}) J^*_1(\boldsymbol{c})$, which yields the desired performance guarantee.

\section{Approximation Algorithms for \texorpdfstring{\ref{ppf}}{PPF} Under a General Choice Model}\label{sec:joint}

We now turn to the main joint optimization problem: \ref{ppf}. In this section, we develop approximation algorithms for the \ref{ppf} problem under a general customer choice model. The performance guarantees depend on whether the customer arrival process is stationary. We use the notation $K^{\min} = \min_{\ell \in \mathcal{L}} K^{\ell}$ and $C_{\min} =\min_{i \in \mathcal{N}}C_i$. The approximation algorithms we develop are asymptotically optimal as $K^{\min}$ and $C_{\min}$ scale.

\begin{theorem} \label{joint stationarity}
For the \ref{ppf} problem under a general choice model and the assumption of stationary arrivals ($\lambda_{jt} = \lambda_{j}$), we can compute a 
$1-2\max\left(\sqrt{\frac{n}{K^{\min}}}, \sqrt{\frac{L}{C_{\min}}}\right)$-approximation in polynomial time.
\end{theorem}

\begin{theorem} \label{joint general}
For the \ref{ppf} problem under a general choice model with non-stationary arrivals, we can compute a $1-2\max\left(\sqrt{\frac{nm}{K^{\min}}}, \sqrt{\frac{Lm}{C_{\min}}}\right)$-approximation in polynomial time.
\end{theorem}

Our approach is to first solve a fluid relaxation of the joint problem to find a near-optimal stocking plan, and then run one of the inventory-agnostic policies from Section~\ref{sec:FixedStock} using a scaled optimal solution to the (joint) fluid relaxation and chosen stocking quantities.

\subsection{A Fluid Relaxation for \texorpdfstring{\ref{ppf}}{PPF}}
We relax the integer constraints on the stocking $\boldsymbol{c}$ and combine that with the fluid approximation $\text{LP}(\boldsymbol{c})$ from Lemma~\ref{lem:lpbound}. This yields the following linear program:
\begin{equation}\label{eq:relaxed_joint}
\text{LP}^* = \max_{\boldsymbol{c}, \boldsymbol{w}, \boldsymbol{y} \ge \bf{0}} \quad \left\{  \sum_{i \in \mathcal{N}}\sum_{j\in \mathcal{M}}\sum_{\ell \in \mathcal{L}_+}  r^\ell_{ij}y^\ell_{ij} \;\; \middle| \;\; \text{s.t. constraints in \eqref{eq:fluid} and constraints on } \boldsymbol{c} \right\}.
\end{equation}

Here, the constraints on $\boldsymbol{c}$ are those of the \ref{ppf} problem with integrality relaxed: $\sum_{i \in \mathcal{N}} c^\ell_i \le K^\ell$ for all $\ell \in \mathcal{L}$ and $\sum_{\ell \in \mathcal{L}} c^\ell_i \le C_i$ for all $i \in \mathcal{N}$. We sometimes refer to this linear program as joint fluid approximation (relaxation) or, more concisely, joint fluid. It can be solved in polynomial time under the assumption of an assortment optimization oracle. 

Let $(\boldsymbol{c}^*, \boldsymbol{w}^*, \boldsymbol{y}^*)$ be an optimal solution to the joint fluid \eqref{eq:relaxed_joint}. It follows from Lemma \ref{lem:lpbound} that this value serves as an upper bound on the optimal expected profit. Similarly to Section~\ref{sec:FixedStock}, without loss of generality, we assume $r^\ell_{ij} {y}^{*\ell}_{ij} \ge 0$ for all $i \in \mathcal{N},j \in \mathcal{M},\ell \in \mathcal{L}_+$; if $r^{\ell}_{ij} < 0$, we can set ${y}^{*\ell}_{ij} = 0$ and transfer its value to $y^{*L+1}_{ij}$, without decreasing the objective value, while maintaining feasibility.

\subsection{Proof of Theorem \ref{joint stationarity}}

To build a policy with desired performance, we follow the plan outlined in the beginning of the section.
\vspace{0.5\baselineskip}

\noindent\textbf{Choosing the Stocking and Dynamic Policy.}
For the approximation guarantee to be non-trivial, we assume $K^{\min} \ge 4n$ and $C_{\min} \ge 4L$. We define a scaling factor $\alpha$ and a rounding term $\gamma$ as follows:
\[
\alpha = 1-\max\left(\sqrt{\frac{n}{K^{\min}}},\sqrt{\frac{L}{C_{\min}}}\right), \quad \gamma = \min\left(\sqrt{\frac{K^{\min}}{n}},\sqrt{\frac{C_{\min}}{L}}\right).
\]
We construct our feasible stocking plan $\hat{\boldsymbol{c}}$ by scaling and rounding up the fluid solution:
\[
\hat{c}^\ell_i = \lfloor \alpha c^{*\ell}_i + \gamma \rfloor \quad \text{for all } i \in \mathcal{N}, \ell \in \mathcal{L}.
\]
We verify feasibility of our stocking plan in Appendix \ref{app:jfs}. The assortment and fulfillment policy we use is Policy 1 with $(\hat{\boldsymbol{w}}, \hat{\boldsymbol{y}}) = \alpha(\boldsymbol{w}^*, \boldsymbol{y}^*)$ (as before, with a slight modification of increasing $\hat w_j(\varnothing)$ as needed to maintain feasibility).

\paragraph{Performance Analysis.}
We now analyze the expected profit of running Policy 1    with the scaled optimal joint fluid solution $(\hat{\boldsymbol{w}}, \hat{\boldsymbol{y}}) = \alpha(\boldsymbol{w}^*, \boldsymbol{y}^*)$ and the feasible capacities $\hat{\boldsymbol{c}}$. Note that $\alpha(\boldsymbol{w}^*, \boldsymbol{y}^*)$ is always a feasible solution for \eqref{eq:fluid} with $\hat{\boldsymbol{c}}$. We prove the desired performance guarantee using Lemma~\ref{main ineq}.

\noindent From Lemma~\ref{lem:expected_profit_formula}, which holds under the assumption of stationary arrivals, it follows that the expected profit is:
\[
\mathbb{E}[\text{Total Profit}] = \sum_{i \in \mathcal{N}}\sum_{ \ell \in \mathcal{L}} \left(\sum_{j \in \mathcal{M}} r^\ell_{ij}\hat{y}^\ell_{ij}\right) \left( 1 - \frac{\mathbb{E}[(Z^\ell_i - \hat{c}^\ell_i)^+]}{\mathbb{E}[Z^\ell_i]}\right),
\]
where $Z^\ell_i$ is the random variable for total demand for product $i$ at warehouse $\ell$, as defined in Section~\ref{sec:FixedStock}. To apply Lemma~\ref{main ineq}, we need to lower-bound the gap $\hat{c}^\ell_i - \mathbb{E}[Z_i^\ell]$. Applying Lemma~\ref{lem:expected_demand_match}, we have $\mathbb{E}[Z_i^\ell] = \sum_{j} \hat{y}^\ell_{ij} = \alpha \sum_{j} y^{*\ell}_{ij} \le \alpha c^{*\ell}_i$. Therefore,
\[
\hat{c}^\ell_i - \mathbb{E}[Z_i^\ell] \ge \lfloor \alpha c^{*\ell}_i + \gamma \rfloor - \alpha c^{*\ell}_i \ge (\alpha c^{*\ell}_i + \gamma - 1) - \alpha c^{*\ell}_i = \gamma - 1.
\]
Since our non-triviality assumption implies $\gamma \ge 2$, we have $\gamma-1 \ge \gamma/2$. Let $\delta = \max\left(\sqrt{\frac{n}{K^{\min}}}, \sqrt{\frac{L}{C_{\min}}}\right)$ for notational simplicity. From Lemma~\ref{main ineq} we get
\[
\frac{\mathbb{E}[(Z^\ell_i - \hat{c}^\ell_i)^+]}{\mathbb{E}[Z_i^\ell]} \le \frac{1}{4(\hat{c}^\ell_i - \mathbb{E}[Z_i^\ell])} \le \frac{1}{4(\gamma-1)} \le \frac{1}{2\gamma} = \frac{1}{2}\delta.
\]
Substituting this bound into the profit formula and using the fact that $\sum_{i \in \mathcal{N}, j \in \mathcal{M}, \ell \in \mathcal{L}} r^\ell_{ij}\hat{y}^\ell_{ij} = \alpha \cdot \text{LP}^*$, we get
\begin{align*}
\mathbb{E}[\text{Total Profit}] &\ge \sum_{i \in \mathcal{N}}\sum_{ j \in \mathcal{M}}\sum_{ \ell \in \mathcal{L}} r^\ell_{ij}\hat{y}^\ell_{ij} \left(1 - \frac{\delta}{2}\right) \\
&= \left(1 - \frac{\delta}{2}\right) \alpha \cdot \text{LP}^* \\
&= \left(1 - \frac{\delta}{2}\right) \left(1 - \delta \right) \text{LP}^* \\
&\ge \left(1 - 2\delta\right) \text{LP}^*.
\end{align*}
Substituting the definition of $\delta$ back gives the final guarantee. Since $\text{LP}^*$ is an upper bound on the optimal expected profit, this completes the proof. 

\subsection{Proof of Theorem \ref{joint general}}

The proof for the non-stationary case follows a similar structure to the stationary case. However, we use Policy 3 and different scaling and rounding parameters, accounting for the greater uncertainty inherent in non-stationary arrivals.

\paragraph{Choosing the Stocking and Dynamic Policy.}
For the approximation to be non-trivial, we assume that $K^{\min} \ge 4mn$ and $C_{\min} \ge 4mL$. We define a new scaling factor $\alpha$ and rounding term $\gamma$:
\[
\alpha = 1-\max\left(\sqrt{\frac{mn}{K^{\min}}},\sqrt{\frac{mL}{C_{\min}}}\right), \quad \gamma = \min\left(\sqrt{\frac{K^{\min}}{mn}},\sqrt{\frac{C_{\min}}{mL}}\right).
\]
We construct an aggregate stocking plan $\hat{\boldsymbol{c}}$ and customer-specific virtual capacities $\hat{c}^\ell_{ij}$ as follows:
\[
\hat{c}^\ell_i = \lfloor \alpha c^{*\ell}_i + m\gamma \rfloor \quad \text{and} \quad \hat{c}^\ell_{ij} = \lfloor \hat{y}^\ell_{ij} + \gamma \rfloor.
\]
We verify feasibility of our stocking plan in Appendix \ref{app:jfg}. The assortment and fulfillment policy we use is Policy 3 with $(\hat{\boldsymbol{w}}, \hat{\boldsymbol{y}}) = \alpha(\boldsymbol{w}^*, \boldsymbol{y}^*)$ (as usual, with a slight modification of increasing $\hat w_j(\varnothing)$ as needed to maintain feasibility).  This policy operates using a feasible fluid solution for assortment and fulfillment probabilities, but makes each sale contingent on the corresponding customer-specific virtual capacity $\hat{c}^\ell_{ij}$ being positive. Note that $\alpha(\boldsymbol{w}^*, \boldsymbol{y}^*)$ is always a feasible solution for \eqref{eq:fluid} with $\hat{\boldsymbol{c}}$.

\paragraph{Performance Analysis.}
The expected profit of Policy 3 is the sum of expected profits across all combinations of product, customer type, and warehouse. For a given tuple $(i,j,\ell)$, the expected profit is $r^\ell_{ij} \mathbb{E}[\bar{Z}^\ell_{ij}]$, where $\bar{Z}^\ell_{ij}$ are corresponding sales, which are equal to demand capped by the customer-specific capacity $\hat{c}^\ell_{ij}$. We can write the total expected profit as:
\[
\mathbb{E}[\text{Total Profit}] = \sum_{i \in \mathcal{N}}\sum_{ j \in \mathcal{M}}\sum_{ \ell \in \mathcal{L}} r^\ell_{ij} \left(\mathbb{E}[Z^\ell_{ij}] - \mathbb{E}[(Z^\ell_{ij} - \hat{c}^\ell_{ij})^+]\right).
\]
Here, $Z^\ell_{ij}$ is the random variable for total demand from type-$j$ customers for product $i$ at warehouse $\ell$. Following the same logic as in the proof of Lemma~\ref{lem:expected_demand_match}, the expected demand is $\mathbb{E}[Z^\ell_{ij}] = \hat{y}^\ell_{ij}$. We can therefore rewrite the expected profit as:
\[
\mathbb{E}[\text{Total Profit}] = \sum_{i \in \mathcal{N}}\sum_{ j \in \mathcal{M}}\sum_{ \ell \in \mathcal{L}} r^\ell_{ij}\hat{y}^\ell_{ij} \left(1 - \frac{\mathbb{E}[(Z^\ell_{ij} - \hat{c}^\ell_{ij})^+]}{\hat{y}^\ell_{ij}}\right).
\]
We use Lemma~\ref{main ineq} again to bound the fractional loss for each $(i,j,\ell)$ tuple. The gap term is:
\[
\hat{c}^\ell_{ij} - \mathbb{E}[Z_{ij}^\ell] = \lfloor \hat{y}^\ell_{ij} + \gamma \rfloor - \hat{y}^\ell_{ij} \ge (\hat{y}^\ell_{ij} + \gamma - 1) - \hat{y}^\ell_{ij} = \gamma - 1.
\]
For the bound to be non-trivial, we must have $\gamma \ge 2$, so $\gamma-1 \ge \gamma/2$. Applying Lemma~\ref{main ineq}:
\[
\frac{\mathbb{E}[(Z^\ell_{ij} - \hat{c}^\ell_{ij})^+]}{\mathbb{E}[Z_{ij}^\ell]} \le \frac{1}{4(\hat{c}^\ell_{ij} - \mathbb{E}[Z_{ij}^\ell])} \le \frac{1}{4(\gamma-1)} \le \frac{1}{2\gamma}.
\]
Let $\delta = \max\left(\sqrt{\frac{nm}{K^{\min}}}, \sqrt{\frac{Lm}{C_{\min}}}\right)$. Then $\frac{1}{2\gamma} = \frac{1}{2}\delta$. Substituting this into the expected total profit formula:
\begin{align*}
\mathbb{E}[\text{Total Profit}] &\ge \sum_{i \in \mathcal{N}}\sum_{j \in \mathcal{M}}\sum_{\ell \in \mathcal{L}} r^\ell_{ij}\hat{y}^\ell_{ij} \left(1 - \frac{\delta}{2}\right) \\
&= \left(1 - \frac{\delta}{2}\right) \alpha \cdot \text{LP}^* \\
&= \left(1 - \frac{\delta}{2}\right) \left(1 - \delta \right) \text{LP}^* \\
&\ge \left(1 - 2\delta\right) \text{LP}^*.
\end{align*}
Substituting the definition of $\delta$ back gives the final guarantee. Since $\text{LP}^*$ is an upper bound on the optimal profit, this completes the proof. 
\section{Constant Factor Approximation for \texorpdfstring{\ref{ppf}}{PPF} Under the MNL Choice Model}\label{sec:mnljoint}

In this section, we show that in the case of the widely-used MNL choice model, we can develop a constant-factor approximation algorithm for the full joint optimization problem, \ref{ppf}. The structure inherent in the MNL model allows for a crucial simplification of the fluid relaxation. This, in turn, enables us to formulate an auxiliary optimization problem whose objective function can be shown to be DR-submodular. By leveraging known results for maximizing submodular functions under matroid constraints, we are able to establish a constant-factor performance guarantee.

We begin by formally defining the MNL model in our setting. For each customer type $j \in \mathcal{M}$, there is a set of non-negative preference weights $\{v_{ij} \mid i \in \mathcal{N}\}$. The probability that a type-$j$ customer chooses product $i$ from an offered assortment $S$ is given by: $\phi_{ij}(S) = \frac{v_{ij}}{1 + \sum_{i' \in S} v_{i'j}}.$
The probability of choosing the no-purchase option is $\frac{1}{1 + \sum_{i' \in S} v_{i'j}}$.

Let us briefly sketch our strategy in building the approximation. We approximate $\text{LP}(\boldsymbol{c})$ with a DR-submodular surrogate function and then apply well-known algorithms to optimize that function subject to capacity and supply constraints. This allows us to find approximately optimal stocking quantities for which we can use some of our policies from Section~\ref{sec:FixedStock}. Our focus in this section, therefore, will be on finding the surrogate function and proving its DR-submodularity. First, however, we formalize our key results in this section.

\begin{theorem}[Constant-Factor Approximations for MNL]\label{thm:jointconst}
For the \ref{ppf} problem under the MNL choice model, there exists a polynomial-time algorithm that computes a $\beta$-approximation, where the factor $\beta$ depends on the problem assumptions as follows:
\begin{itemize}[leftmargin=2em]
    \item In the general case: $\beta = 0.0805 - \epsilon$.
    \item If product supplies are unlimited ($C_{\min} = +\infty$): $\beta = 0.161(1 - 1/e)$.
    \item If customer arrivals are stationary: $\beta = \frac{1}{4}(1 - 1/e) - \epsilon$.
    \item If arrivals are stationary and product supplies are unlimited: $\beta = \frac{1}{2}(1 - 1/e)^2 $.
\end{itemize}
Here, $\epsilon > 0$ can be an arbitrarily small positive constant and its inverse, $1/\epsilon$, is treated as a constant in the complexity analysis.
\end{theorem}

The following lemma, proven in Appendix \ref{app:eqlp} is key to our approach. It establishes that under the MNL model, the complex fluid approximation from \eqref{eq:fluid} is equivalent to a much simpler and more tractable linear program. This program is then modified to get a surrogate with the desired properties, similar to the approach of \citet{baietal2025}.

\begin{lemma}[LP Equivalence for MNL]\label{lem:mnl_lp_equiv}
Under the MNL choice model, the optimal value of the fluid approximation \eqref{eq:fluid}, denoted as $\textup{LP}(\boldsymbol{c})$, is equal to the optimal value of the following linear program:
\begin{equation}\label{eq:mnl_lp_compact}
\max_{\boldsymbol{y\ge0}, \boldsymbol{y}_0\boldsymbol{\ge0}} \left\{ \begin{array}{l@{\quad}l}
    \displaystyle\sum_{i \in \mathcal{N}}\sum_{j \in \mathcal{M}}\sum_{\ell \in \mathcal{L}_+} r^\ell_{ij}y^\ell_{ij} : & \displaystyle\sum_{j \in \mathcal{M}} y^\ell_{ij} \le c^\ell_i, \quad \forall i \in \mathcal{N}, \ell \in \mathcal{L}; \\
    & \displaystyle\sum_{i \in \mathcal{N}}\sum_{\ell \in \mathcal{L}_+} y^\ell_{ij} + y_{0j} \le \tau_j, \quad \forall j \in \mathcal{M}; \\
    & \displaystyle\sum_{\ell \in \mathcal{L}_+} y^\ell_{ij} \le v_{ij} y_{0j}, \quad \forall i \in \mathcal{N}, j \in \mathcal{M}
\end{array} \right\}
\end{equation}
where $y_{0j}$ represents the expected number of times a type-$j$ customer leaves without making a purchase.
\end{lemma}

\subsection{A DR-Submodular Surrogate for the Fluid Optimum}

The core of our approximation algorithm is to reframe the stocking optimization problem as a maximization of a DR-submodular function, for which efficient approximation algorithms are known. However, the fluid optimum $\text{LP}(\boldsymbol{c})$ itself is known not to be DR-submodular even in a single warehouse setting, as discussed in~\citep{baietal2025}. Following the surrogate construction of \citet{baietal2025}, adapted here to customer-type-dependent profits, we introduce a closely related surrogate function $f_{app}(\boldsymbol{c})$, which approximates $\text{LP}(\boldsymbol{c})$ and possesses the required DR-submodularity. Intuitively, $f_{app}$ caps each customer type's expected sales and each product's MNL attraction at half their natural maxima, restoring diminishing returns in stocking while preserving a constant fraction of the fluid value.

Let $f_{app}(\boldsymbol{c})$ be the optimal value of the following linear program:
\begin{equation}\label{eq:f_app}
\max_{\boldsymbol{y\ge0}} \left\{ \begin{aligned}
    & \sum_{i \in \mathcal{N}} \sum_{j \in \mathcal{M}} \sum_{\ell \in \mathcal{L}_+} r^{\ell}_{ij}\,y^{\ell}_{ij} : && \sum_{j \in \mathcal{M}} y^{\ell}_{ij} \le c^{\ell}_i, \quad \forall i \in \mathcal{N}, \ell \in \mathcal{L}; \\
    & && \sum_{i \in \mathcal{N}}\sum_{\ell \in \mathcal{L}_+} y^{\ell}_{ij} \le \frac{\tau_j}{2}, \quad \forall j \in \mathcal{M}; \\
    & && \sum_{\ell \in \mathcal{L}_+} y^{\ell}_{ij} \le \frac{v_{ij}\,\tau_j}{2}, \quad \forall i \in \mathcal{N}, j \in \mathcal{M}
\end{aligned} \right\}.
\end{equation}
This function has two crucial properties, stated in the following lemmas.

\begin{lemma}[Approximation Property of $f_{app}$]\label{lem:fapp_approx}
For any inventory placement $\boldsymbol{c}$, the surrogate function $f_{app}(\boldsymbol{c})$ provides a $1/2$-approximation to the fluid optimum, i.e.,
\[ \frac{1}{2} \text{LP}(\boldsymbol{c}) \le f_{app}(\boldsymbol{c}) \le \text{LP}(\boldsymbol{c}). \]
\end{lemma}
The proof is deferred to Appendix~\ref{app:fapp_approx}.

\begin{lemma}[DR-Submodularity of $f_{app}$]\label{lem:fapp_dr}
The function $f_{app}(\boldsymbol{c})$ is DR-submodular in $\boldsymbol{c}$. That is, for any integer vectors $\boldsymbol{c} \ge \boldsymbol{b} \ge \bf{0}$ and a unit integer vector $\boldsymbol{e}^{\ell}_i$, we have $$f_{app}(\boldsymbol{c}) + f_{app}(\boldsymbol{b} + \boldsymbol{e}_i^\ell) \ge f_{app}(\boldsymbol{c} + \boldsymbol{e}_i^\ell) + f_{app}(\boldsymbol{b}).$$
\end{lemma}

The proof of Lemma~\ref{lem:fapp_dr}, which requires a detailed analysis of the dual formulation, is presented in Appendix \ref{app:dr}. We are now ready to construct the approximation algorithm.

\subsection{The Algorithm for Theorem \ref{thm:jointconst}}

In this section we formalize the algorithm with approximation guarantees matching those in the theorem statement. The algorithm proceeds in two main stages:
\begin{enumerate}
    \item \textbf{Offline Planning:} First, we find a near-optimal stocking plan $\boldsymbol{c}_{\text{alg}}$ by approximately maximizing the DR-submodular surrogate function $f_{app}(\boldsymbol{c})$ subject to the capacity and supply constraints.
    \item \textbf{Online Execution:} Second, we run the appropriate assortment and fulfillment policy from Section \ref{sec:FixedStock}, guided by the fluid solution corresponding to this stocking plan $\boldsymbol{c}_{\text{alg}}$.
\end{enumerate}

\paragraph{Framework For Offline Planning: Maximizing a Submodular Set Function.}
To leverage standard results from submodular optimization, we first translate our problem from maximizing a function over the integer lattice to maximizing a set function. We create a ground set $E$ of discrete ``inventory units.'' For each product-warehouse pair $(i,\ell)$, we create a set $E_{i\ell}$ of $\min(K^{\ell}, C_i, T)$ identical units. The ground set is the disjoint union $E = \bigcup_{i,\ell} E_{i\ell}$. Any subset $S \subseteq E$ corresponds to a stocking plan $\boldsymbol{c}$ with $c_i^\ell = |S \cap E_{i\ell}|$. Note that we are allowed to limit each stocking quantity $c^{\ell}_i$ by $\min(K^{\ell}, C_i, T)$ without decreasing the maximum possible expected profit due to feasibility constraints and the fact that stocking more than $T$ units cannot increase expected profit, because no more than $T$ customers arrive over the selling horizon. This constraint ensures that the size of the constructed ground set is polynomial in the original input size.

We define a set function $g: 2^E \to \mathbb{R}_+$ corresponding to $f_{app}(\boldsymbol{c})$ as $g(S) = f_{app}(\boldsymbol{c}_S)$, where $\boldsymbol{c}_S$ is the vector of stockings induced by the set $S$. Because $f_{app}(\boldsymbol{c})$ is monotone and DR-submodular (Lemma~\ref{lem:fapp_dr}), the corresponding set function $g(S)$ is monotone and submodular. The capacity and supply constraints on $\boldsymbol{c}$ can now be modeled as matroid constraints on the set $S$.
\begin{itemize}
    \item \textbf{General Case:} The warehouse capacity constraints, $\sum_{i \in \mathcal{N}} c_i^\ell \le K^\ell$, define a partition matroid $\mathcal{M}_K$ on $E$. The product supply constraints, $\sum_{\ell \in \mathcal{L}} c_i^\ell \le C_i$, define a second partition matroid $\mathcal{M}_C$. The problem is to maximize a monotone submodular function $g(S)$ subject to the intersection of two matroids: $S \in \mathcal{I}(\mathcal{M}_K) \cap \mathcal{I}(\mathcal{M}_C)$. This problem admits a $(1/2-\epsilon)$-approximation algorithm, see \cite{lee2010submodular}.

    \item \textbf{Unlimited Product Supplies ($C_{\min} = +\infty$):} In this case, only the warehouse capacity constraints remain. The problem simplifies to maximizing a monotone submodular function subject to a single partition matroid constraint. This problem admits a randomized $(1-1/e)$-approximation, as well as a deterministic $(1-1/e-\epsilon)$-approximation, see \cite{calinescu2011maximizing}, \cite{BuchbinderFeldman2024}.
\end{itemize}

\paragraph{Deriving the Performance Guarantees.}

The final approximation factor $\beta$ is the product of three distinct performance guarantees: (1) the factor for the submodular maximization algorithm used in the planning stage, (2) the factor from Lemma~\ref{lem:fapp_approx} relating $f_{app}(\boldsymbol{c})$ to the true fluid optimum $\text{LP}(\boldsymbol{c})$, and (3) the performance guarantee of the chosen assortment and fulfillment policy from Section \ref{sec:FixedStock}. Multiplying these factors in each case yields exactly the guarantees stated in the theorem statement as shown in Appendix \ref{app:pg}.

\section{From Inventory-Agnostic to Inventory-Aware Policies}\label{sec:inventorybased}

The policies developed in the previous section are inventory-agnostic: they make assortment and fulfillment decisions based on a pre-determined fluid solution, only checking for stock at the final step. In this section, we show how to convert these policies into more practical inventory-aware policies that only offer products that are in stock, while provably preserving the same performance guarantees. 

The conversion relies on the following key result, which is an application of Proposition~3 in \cite{feng2021near} formulated in our notation; we include the derivation for completeness. It allows us to dynamically adjust a randomized assortment offering to account for a restricted set of available products, without altering the marginal choice probabilities for those available products.

\begin{lemma}\label{lem:dist_conversion}
Let $\phi_i(S)$ be a customer choice model satisfying weak substitutability. Let $p$ be a probability distribution over assortments $S \subseteq \mathcal{N}$ and let $P_i = \mathbb{E}_{S \sim p}[\phi_i(S)]$ for each product $i$. Assume that we can sample from $p$ in polynomial time. Then, for any set $\hat{S} \subseteq \mathcal{N}$, there exists a distribution $p'$ over subsets of $\hat{S}$ such that an assortment can be sampled from it in polynomial time, and for all products $i \in \hat{S}$, the marginal choice probability is preserved: $\mathbb{E}_{S \sim p'}[\phi_i(S)] = P_i$.
\end{lemma}
\begin{proof}
First, we sample an assortment $\bar{S}$ from $p$. Then we apply Proposition 3 from \cite{feng2021near} with $p_i = \phi_i(\bar{S}) \: \: \forall i \in \mathcal{N}$ and $  S = \hat{S} \cap \bar{S},$ sampling an assortment $\tilde{S}$. This sampling procedure is polynomial in time and induces a distribution $p'$ on subsets of $\hat{S}$ such that $\mathbb{E}_{S \sim p'}[\phi_i(S)] = P_i \: \: \forall i \in \hat S$ (by the law of total expectation).
\end{proof}

\subsection{Inventory-Aware Simple Policy}
We now present the inventory-aware counterpart to Policy 1.

\vspace{0.5\baselineskip}
\needspace{18\baselineskip}
\noindent\rule{\textwidth}{0.4pt}
\textbf{Policy 1-IA: Inventory-Aware Simple Policy}
\vspace{0.25\baselineskip}
\begin{itemize}[leftmargin=1.5em]
    \item \textbf{Input:} A feasible solution $(\hat{\boldsymbol{w}}, \hat{\boldsymbol{y}})$ to the fluid approximation \eqref{eq:fluid}.
    \item \textbf{Execution:} In each time period $t \in \mathcal{T}$, if a customer of type $j$ arrives:
    \begin{enumerate}[leftmargin=*,label=\arabic*.]
        \item \textbf{Pre-select Fulfillment Warehouses:} For each product $i \in \mathcal{N}$, independently sample a potential fulfillment warehouse $\ell_i \sim \rho_{ij}$, the fulfillment distribution over $\mathcal{L}_+$ defined in Policy 1.
        
        \item \textbf{Determine Available Set:} Identify the set of products that are currently in stock at their pre-selected, non-virtual warehouse: $S_t = \{i \in \mathcal{N} \mid \ell_i \in \mathcal{L} \text{ and } x_i^{\ell_i}(t) > 0\}$.
        
        \item \textbf{Construct and Offer a Feasible Assortment:}
            \begin{enumerate}[leftmargin=*,label=(\alph*)]
                \item Find a probability distribution $p'_{j,S_t}$ over assortments $S \subseteq S_t$ that preserves each available product's purchase probability, so that every available product $i \in S_t$ has the same purchase probability under $p'_{j,S_t}$ as under the original inventory-agnostic policy.
                \item Offer an assortment $\tilde{S} \sim p'_{j,S_t}$ to the customer.
            \end{enumerate}
        
        \item \textbf{Fulfill Order:} If the customer chooses product $i \in \tilde{S}$, fulfill the order from its pre-selected warehouse $\ell_i$. The corresponding stock is depleted by one.
    \end{enumerate}
\end{itemize}
\noindent\rule{\textwidth}{0.4pt}
\vspace{0.1cm}

Such a distribution $p'_{j,S_t}$ exists by Lemma~\ref{lem:dist_conversion}, applied with available set $\hat{S} = S_t$, type-$j$ choice model $\phi_{ij}$, and input distribution $q_j(S) = \hat{w}_j(S)/\tau_j$, which is the assortment distribution used by the inventory-agnostic policy. For every available product $i \in S_t$, it satisfies
\[  \mathbb{E}_{S \sim p'_{j,S_t}}[\phi_{ij}(S)] = \mathbb{E}_{S \sim q_j}[\phi_{ij}(S)] = \frac{\sum_{\ell' \in \mathcal{L}_+} \hat{y}^{\ell'}_{ij}}{\tau_j}.  \]
The first equality is the marginal-preservation guarantee of Lemma~\ref{lem:dist_conversion}; the second substitutes $q_j(S) = \hat{w}_j(S)/\tau_j$ and applies the flow-balance constraint $\sum_{S \subseteq \mathcal{N}} \phi_{ij}(S)\,\hat{w}_j(S) = \sum_{\ell \in \mathcal{L}_+} \hat{y}^{\ell}_{ij}$ of the fluid approximation~\eqref{eq:fluid}.

\noindent This policy has the same performance as the original inventory-agnostic one, which we formalize in the following theorem, proven by a coupling argument in Appendix \ref{app:ib}.

\begin{theorem}\label{thm:ib}
The expected total profit of the Inventory-Aware Simple Policy (Policy 1-IA) is equal to the expected total profit of Policy 1.
\end{theorem}

\subsection{Other Inventory-Aware Policies}

While we explicitly construct the inventory-aware counterpart only for Policy 1, same technique applies for other policies we introduced. To avoid duplication, in this subsection we briefly state the conversion for other policies.

\paragraph{Advanced Inventory-Aware Policy (2-IA)}
The only difference between Policy 2-IA and Policy 1-IA is that pre-select probabilities are scaled by $(1-H^\ell_{ij})$. The performance analysis remains the same with a minor modification of multiplying appropriate expressions by $(1-H^{\ell}_{ij})$.

\paragraph{Modified Inventory-Aware Policy (3-IA)}
The only difference between Policy 3-IA and Policy 1-IA is that stock availability check is done against the customer-specific virtual capacities $c^\ell_{ij}$. The performance analysis remains the same with a minor modification of including customer type parameter in stock indicators.

\subsection{Special Case: The MNL Model}
Under the MNL choice model, formally defined in the beginning of Section~\ref{sec:mnljoint}, the conversion can be done more directly without solving a linear program at each step for any of the policies.  While all other steps are done exactly as for a general choice model, in Step 3 instead of applying the algorithm from Lemma~\ref{lem:dist_conversion}, we simply sample $\tilde S$ from  $$p'_{j,S_t}(S) = \frac{\hat w_j(S) \cdot (1+v_j(S \cap S_t))}{\tau_j \cdot (1+v_j(S))} \quad \forall S \subseteq \mathcal{N}: S \neq \varnothing, \qquad p'_{j,S_t}(\varnothing) = 1 - \sum_{S \neq \varnothing} p'_{j,S_t}(S),$$ and offer $\tilde S \cap S_t$, where $v_j(A) = \sum_{i \in A} v_{ij}$ is the total preference weight of the products in set $A$. To see why the marginal choice probabilities are preserved, first note that $p'_{j,S_t}$ is a valid distribution: because $S \cap S_t \subseteq S$ we have $v_j(S \cap S_t) \le v_j(S)$, so each weight satisfies $p'_{j,S_t}(S) \le \hat w_j(S)/\tau_j$, and since $\sum_{S} \hat w_j(S) = \tau_j$ the residual mass assigned to $\varnothing$ is nonnegative. Now fix an available product $i \in S_t$ and consider the two sources of randomness, the sampled set $S \sim p'_{j,S_t}$ and the customer's choice from the offered set $S \cap S_t$. Since we offer $S \cap S_t$, product $i$ is chosen with probability $v_{ij}\,\mathbbm{1}[i \in S \cap S_t]/(1+v_j(S \cap S_t))$, and only realizations with $i \in S$ contribute because $i \in S_t$ is fixed. Averaging over $S$ using the definition of $p'_{j,S_t}(S)$, the probability that $i$ is chosen equals
$$\sum_{S:\,i \in S} \frac{\hat w_j(S)\,(1+v_j(S \cap S_t))}{\tau_j\,(1+v_j(S))}\cdot\frac{v_{ij}}{1+v_j(S \cap S_t)} = \sum_{S:\,i \in S} \frac{\hat w_j(S)}{\tau_j}\cdot\frac{v_{ij}}{1+v_j(S)} = \mathbb{E}_{S \sim p}[\phi_{ij}(S)],$$
the marginal probability required by Lemma~\ref{lem:dist_conversion}. Intuitively, trimming the sampled set $S$ down to the available set $S \cap S_t$ shrinks the MNL normalizer from $1+v_j(S)$ to $1+v_j(S \cap S_t)$; the weight $p'_{j,S_t}(S)$ carries the compensating factor $(1+v_j(S \cap S_t))/(1+v_j(S))$ that cancels this change exactly, leaving every available product's marginal choice probability equal to $\mathbb{E}_{S \sim p}[\phi_{ij}(S)]$. Because this correcting factor is available in closed form, no linear program is solved at any step, providing a highly efficient implementation.

\section{Computational Experiments}\label{sec:numerics}

In this section, we conduct a series of computational experiments to evaluate the empirical performance of our proposed algorithms. The primary goals of these experiments are (i) to validate our theoretical performance guarantees on a range of synthetically generated data, (ii) to compare the algorithms against a standard benchmark and enhance their performance via simple heuristics, and (iii) to understand the practical trade-offs between theoretical guarantees, and real-world computational complexity and performance. We first detail our experimental design, then describe the algorithms tested, and finally present and discuss the computational findings.

\subsection{Experimental Setup}
Every test instance contains $n=100$ products, $m=50$ customer types, $L=10$ warehouses, and is characterized by randomly generated model inputs, unique for each instance. In our experiments, each customer type is represented by a location, i.e., a point on the map. This representation is what makes the triadic (product, customer type, warehouse) profit heterogeneity clean to formulate: the distances from a customer type's location to the warehouses define the shipping costs, while the customer type's product preferences are carried by the customer type itself. Customer type and warehouse locations are defined by coordinates sampled uniformly from the square $[0, 10] \times [0, 10]$. The revenue $r_i$ for each product $i$ is sampled independently from $\text{Uniform}[10, 20]$, and the profit for fulfilling an order for product $i$ for a type-$j$ customer from warehouse $\ell$ is $r_{ij}^\ell = r_i - \text{dist}(\ell, j)$, where $\text{dist}(\ell, j)$ is the Euclidean distance between warehouse $\ell$ and customer type $j$'s location. The selling horizon consists of $T$ time periods, where we experiment with different values of $T$.

Customer choices follow the MNL model. To introduce heterogeneity, we first determine a consideration set $\mathcal{C}_j$ for each type $j$ by sampling a size $L_j \sim \text{Uniform}\{10, \dots, 40\}$ and then choosing a subset of $\mathcal{N}$ of that size uniformly at random. For each product $i \in \mathcal{C}_j$, we sample a raw preference weight $\tilde{v}_{ij} \sim \text{Uniform}[1, 10]$, and set $\tilde{v}_{ij}=0$ for $i \notin \mathcal{C}_j$. For half of the customer types, we reorder these raw weights in reverse order of the revenues $r_i$ to model price sensitivity, meaning that these customers assign higher preference to lower-priced (lower-revenue) products and thus are less willing to purchase more expensive ones. We fix the no-purchase weight at $v_{0j}=1$ and use the normalized preference weights
\begin{equation*}
v_{ij} = \frac{\tilde{v}_{ij}}{\frac{P_0}{1-P_0}\sum_{k\in\mathcal{C}_j}\tilde{v}_{kj}} \qquad \text{for each } i \in \mathcal{C}_j,
\end{equation*}
which ensures that, when the full consideration set $\mathcal{C}_j$ is offered, the no-purchase probability equals $P_0$, a parameter we vary across experiments.

To model non-stationarity, we make customer types with smaller consideration sets, i.e., ``pickier'' customers more likely to arrive later in the horizon. The arrival probability for type $j$ at time period $t$ is given by $\lambda_{jt} \propto \exp(-\gamma L_j(t - T/2))$, normalized so that $\sum_{j \in \mathcal{M}} \lambda_{jt} = 1$ in every period $t$. We choose the parameter $\gamma$ such that the minimum time-averaged market share across all customer types, $\min_{j \in \mathcal{M}} \frac{1}{T}\sum_t \lambda_{jt}$, equals $\theta_{\min}$.

To set realistic capacity and supply levels, which allow us to test our policies in structurally different settings,  we first calculate a baseline total demand $D_i$ for each product by assuming that, for each type $j \in \mathcal{M}$, the myopic optimal assortment $\widetilde S_j
=\arg\max_{S\subseteq\mathcal{N}}
\sum_{i\in S}r_{ij}\,\phi_{ij}(S)
$ is offered to each incoming customer of that type,
where $r_{ij} = \max_{\ell \in \mathcal{L}}r^{\ell}_{ij}$ is the best profit product $i$ can earn from a type-$j$ customer across all warehouses. Assuming each type-$j$ customer is shown its myopic optimal assortment $\widetilde S_j$, $D_i = \sum_{j \in \mathcal{M}} \tau_j\, \phi_{ij}(\widetilde S_j)$ is the expected number of purchases of product $i$. We then set the total product supply as $C_i = D_i / \rho$ and the capacity of each warehouse as $K^\ell = (\sum_{i \in \mathcal{N}} C_i / L) \cdot \alpha$. The parameters $\rho$ and $\alpha$ control the supply and capacity tightness, respectively. 

Varying $P_0 \in \{0.1,0.3\}$, $\theta_{\min} \in \{\frac{1}{200}, \frac{1}{100}\}$, $\rho \in \{1.0, 1.4, 2.0 \}$, and $\alpha \in \{0.6, 1.0, 1.4 \}$ yields 36 unique parameter configurations. We generate one instance for each configuration for $T = 8{,}000$ and $T = 100{,}000$. In additional experiments, generating multiple instances per parameter configuration produced qualitatively similar results, with comparable quantitative performance, at substantially higher computational cost; we therefore limit the scope to one instance per configuration. Customer type and warehouse locations, revenues, preference weights and consideration sets are generated separately for each instance.

\subsection{Tested Algorithms}

We test the empirical performance of five algorithms. Four closely follow the theoretical algorithms we propose: one from Section \ref{sec:mnljoint} and three variations of those from Section \ref{sec:joint}. These four are enhanced for efficient implementation; relative to the theoretical versions, we simplify the stocking optimization for computational tractability and experiment with different inventory-aware conversion heuristics. The fifth is a standard benchmark that relies on myopic optimization. Each algorithm has three components: (i) an offline stocking decision, (ii) a ``raw'' inventory-agnostic assortment and fulfillment policy, and (iii) a conversion rule that makes the raw policy inventory-aware.  We now provide a description of each algorithm.

\paragraph{DualGreedy and SubSampling (DGSS).}
This policy combines our algorithm from Section~\ref{sec:mnljoint} with the inventory-awareness technique from Section \ref{sec:inventorybased}. To lower computational complexity and improve empirical performance significantly, we make some simplifications.
\begin{itemize}[leftmargin=*]
    \item \textit{Stocking Decision}: To find a near-optimal stocking plan ${\boldsymbol{c}}_{alg}$, we approximately maximize the surrogate function $f_{app}(\boldsymbol{c})$ using a computationally efficient \textit{dual-greedy} heuristic. This heuristic iteratively allocates inventory to the product-warehouse pair with the highest corresponding dual variable in the LP \eqref{eq:f_app}. While this approach currently does not have a theoretical guarantee, it is based on the two following observations. First, it is a classical result that a simple greedy algorithm allows to optimize a monotone DR-submodular function on non-negative integer lattice subject to 2 partition matroid constraints with $1/3$-approximation. Second, $f_{app}(\boldsymbol{c}+{\boldsymbol{e}}^{\ell}_i) - f_{app}(\boldsymbol{c})$ can be approximated by $\frac{\partial f_{app}}{\partial c^{\ell}_i}(\boldsymbol{c})$, where it exists, which is, in turn, equal to the value of dual variable corresponding to $c^{\ell}_i$ in the LP formulation of $f_{app}(\boldsymbol{c})$.
    
    \item \textit{``Raw'' Personalization and Fulfillment Policy}: The raw policy is Policy 1 from Section \ref{sec:FixedStock} with optimal solution $(\boldsymbol{w}^*, \boldsymbol{y}^*)$ to LP \eqref{eq:fluid} with the stocking plan $\boldsymbol{c}_{alg}$. Note that while this policy has a strong theoretical guarantee of $(1-1/e)$-approximation in the case of stationarity, it does not have theoretical guarantees in the general case, which we are testing. However, as is common with algorithms like Policy 2, which artificially lower sales of products to provide a better theoretical guarantee, in practice it is more sensible not to perform such reduction. We will see that this approach pays off for our setting later in this section.
    
    \item \textit{Inventory-Aware Conversion}: The ``raw'' policy is made inventory-aware using the subsampling technique from Section \ref{sec:inventorybased}, which provably preserves the expected profit.
\end{itemize}

\paragraph{Rounding and SubSampling (RSS).}
This is a simplified version of the algorithms from Section~\ref{sec:joint}. We first solve the joint fluid LP \eqref{eq:relaxed_joint} to obtain an optimal fractional stocking plan $\boldsymbol{c}^*$. We then convert $\boldsymbol{c}^*$ into a feasible integer plan $\boldsymbol{c}_{alg}$ by randomized rounding: each product $i$ is stocked at the floor $\lfloor c^{*\ell}_i\rfloor$ in every warehouse $\ell$, and the few remaining units needed to reach $\lceil\sum_\ell c^{*\ell}_i\rceil$ (capped by the supply $C_i$) are assigned to warehouses sampled without replacement with probability proportional to the fractional parts $c^{*\ell}_i-\lfloor c^{*\ell}_i\rfloor$; any warehouse exceeding its capacity $K^\ell$ then drops its added units with the smallest fractional parts until feasible, so that both $\sum_{i\in\mathcal{N}} c^\ell_i\le K^\ell$ and $\sum_{\ell\in\mathcal{L}} c^\ell_i\le C_i$ hold. The assortment and fulfillment policy is the same as in DGSS.

\paragraph{Rounding and Drop-if-Unavailable (RDU).}
This is a modification of RSS which uses the same stocking algorithm and ``raw'' assortment and fulfillment policy but replaces subsampling with a simpler, more intuitive inventory-awareness rule. In each time period $t$, the algorithm determines the set of available products, namely those with positive stock in at least one warehouse where the profit is positive. It then removes all unavailable products from the ``raw'' assortment. If the customer chooses product $i$, the order is fulfilled from an available, positive-profit warehouse $\ell$ with probability proportional to the fluid variable $\hat{y}_{ij}^\ell$. Unlike subsampling, this inventory-awareness rule carries no proven performance guarantee.

\paragraph{Rounding and Drop-at-Selected-Warehouse (RDWH).}
This is another modification of RSS which uses the same stocking algorithm and ``raw'' assortment and fulfillment policy but another inventory-awareness rule. In each time period $t$, the policy first pre-selects a potential fulfillment warehouse $\ell_i$ for each product $i$ according to the fulfillment probabilities from the fluid solution. It then removes all products which are out of stock at their pre-selected warehouse from the ``raw'' assortment. If a customer makes a purchase, the order is fulfilled from the pre-selected warehouse. Unlike subsampling, this inventory-awareness rule carries no proven performance guarantee.

\paragraph{Benchmark (BM).}
This is a static benchmark policy based on myopic optimization. The stocking plan is determined by solving a mixed-integer program that allocates inventory to best serve the demand that would be generated by offering a fixed, myopically optimal assortment $\widetilde{S}_j$, defined earlier in this section, to each customer type $j$. Concretely, for each customer type $j$, the benchmark fixes the assortment offered to all customers of this type to its myopic optimal assortment $\widetilde{S}_j$, so that type $j$ generates an expected demand of $\tau_j\,\phi_{ij}(\widetilde{S}_j)$ units of product $i$, the expected number of type-$j$ customers who choose $i$ if we always show them $\widetilde{S}_j$. It then solves the MIP similar to the LP \eqref{eq:relaxed_joint} but with this fixed demand and the stocking levels $\boldsymbol{c}$ restricted to integers, placing inventory where it best serves that demand subject to the warehouse capacities $K^\ell$ and product supplies $C_i$. The resulting optimal $\boldsymbol{c}_{alg}$ is the benchmark stocking plan. The ``raw'' policy is static: always offer the myopic assortment $\widetilde{S}_j$ to a type-$j$ customer. The inventory-awareness conversion is also straightforward. In each time period $t$, the policy offers the subset of the myopic assortment $\widetilde{S}_j$, consisting of products that are currently in stock at any warehouse with a positive corresponding profit for the current customer type. Fulfillment is routed greedily to an available warehouse with the highest profit.

\subsection{Computational Results and Discussion}

We present the results of our computational experiments in Table~\ref{tab:perf_8000_100000}. Each row corresponds to one test instance, identified in the leftmost column by its parameter tuple $(P_0, \theta_{\min}, \rho, \alpha)$, where $P_0$ is the no-purchase probability, $\theta_{\min}$ the minimum time-averaged market share, and $\rho$ and $\alpha$ control the supply and capacity tightness, respectively. Expected profit of each policy is estimated by averaging over $100$ simulations for each policy. The performance of each policy is reported as a percentage of the upper bound on the optimal profit, which is obtained by solving the joint fluid LP \eqref{eq:relaxed_joint}. 

{ \scriptsize
 \setlength{\tabcolsep}{2pt}
 \begin{longtable}{p{2.4cm}ccccc c ccccc}
   \caption{Performance Comparison Across Different Policies and Instances (T = 8{,}000 vs.\ 100{,}000)}\\
   \label{tab:perf_8000_100000}\\
   \toprule
   \multirow{2}{*}{\shortstack[l]{\textbf{Instance}\\(P$_0,\theta_{\min},\rho,\alpha$)}}
     & \multicolumn{5}{c}{\textbf{T = 8{,}000}}
     & 
     & \multicolumn{5}{c}{\textbf{T = 100{,}000}} \\
   \cmidrule(r){2-6} \cmidrule(l){8-12}
     & \textbf{BM}
     & \textbf{DGSS}
     & \textbf{RSS}
     & \textbf{RDU}
     & \textbf{RDWH}
     & \hspace{1cm}
     & \textbf{BM}
     & \textbf{DGSS}
     & \textbf{RSS}
     & \textbf{RDU}
     & \textbf{RDWH} \\
   \midrule
 \endfirsthead

   \multicolumn{12}{@{}l}{\small\sl continued from previous page}\\
   \toprule
   \multirow{2}{*}{\shortstack[l]{\textbf{Instance}\\(P$_0,\theta_{\min},\rho,\alpha$)}}
     & \multicolumn{5}{c}{\textbf{T = 8{,}000}}
     & 
     & \multicolumn{5}{c}{\textbf{T = 100{,}000}} \\
   \cmidrule(r){2-6} \cmidrule(l){8-12}
     & \textbf{BM}
     & \textbf{DGSS}
     & \textbf{RSS}
     & \textbf{RDU}
     & \textbf{RDWH}
     & \hspace{1cm}
     & \textbf{BM}
     & \textbf{DGSS}
     & \textbf{RSS}
     & \textbf{RDU}
     & \textbf{RDWH} \\
   \midrule
 \endhead

   \midrule
   \multicolumn{12}{r}{\small\sl continued on next page}\\
 \endfoot

   \midrule
   \textbf{Average}
     & 88.51 & 81.16 & 88.31 & 92.47 & 92.74
     & 
     & 89.48 & 88.84 & 96.37 & 95.74 & 98.05 \\
   \bottomrule
 \endlastfoot

 (0.1,0.005,1.0,0.6)  & 90.05 & 89.08 & 90.40 & 94.75 & 94.44 & & 90.35 & 95.16 & 97.22 & 97.90 & 98.52 \\
 (0.1,0.005,1.0,1.0)  & 96.24 & 61.93 & 90.54 & 95.80 & 95.30 & & 98.63 & 65.37 & 97.30 & 98.80 & 98.69 \\
 (0.1,0.005,1.0,1.4)  & 95.94 & 63.58 & 91.17 & 96.00 & 95.56 & & 98.23 & 66.99 & 97.45 & 98.80 & 98.73 \\
 (0.1,0.005,1.4,0.6)  & 88.73 & 89.43 & 89.82 & 93.74 & 93.55 & & 89.15 & 96.09 & 96.82 & 95.73 & 98.20 \\
 (0.1,0.005,1.4,1.0)  & 91.79 & 78.97 & 89.80 & 94.69 & 94.98 & & 92.16 & 84.93 & 96.93 & 97.27 & 98.70 \\
 (0.1,0.005,1.4,1.4)  & 90.29 & 80.44 & 91.14 & 94.48 & 95.28 & & 90.79 & 85.70 & 97.40 & 97.65 & 98.81 \\
 (0.1,0.005,2.0,0.6)  & 83.27 & 86.89 & 89.58 & 94.52 & 93.66 & & 83.16 & 94.62 & 96.72 & 96.20 & 98.32 \\
 (0.1,0.005,2.0,1.0)  & 81.69 & 88.93 & 90.87 & 93.40 & 95.32 & & 81.80 & 95.22 & 97.18 & 96.14 & 98.60 \\
 (0.1,0.005,2.0,1.4)  & 86.23 & 87.81 & 89.61 & 92.04 & 94.16 & & 86.47 & 95.70 & 96.89 & 94.60 & 98.42 \\
 (0.1,0.01,1.0,0.6)   & 83.64 & 88.52 & 91.02 & 94.43 & 94.30 & & 83.66 & 94.36 & 97.41 & 97.48 & 98.53 \\
 (0.1,0.01,1.0,1.0)   & 94.68 & 65.92 & 90.52 & 95.49 & 95.00 & & 96.32 & 69.89 & 97.27 & 98.68 & 98.61 \\
 (0.1,0.01,1.0,1.4)   & 95.76 & 60.75 & 90.55 & 95.04 & 95.35 & & 98.57 & 64.41 & 97.35 & 98.58 & 98.72 \\
 (0.1,0.01,1.4,0.6)   & 81.47 & 88.06 & 89.85 & 94.32 & 93.61 & & 81.72 & 94.16 & 96.81 & 96.69 & 98.32 \\
 (0.1,0.01,1.4,1.0)   & 90.40 & 79.13 & 89.52 & 94.87 & 94.49 & & 90.93 & 85.56 & 96.86 & 96.75 & 98.53 \\
 (0.1,0.01,1.4,1.4)   & 92.41 & 79.68 & 90.31 & 94.53 & 95.17 & & 92.66 & 85.65 & 96.81 & 96.49 & 98.70 \\
 (0.1,0.01,2.0,0.6)   & 82.36 & 88.57 & 88.85 & 92.20 & 93.08 & & 82.53 & 95.52 & 96.37 & 95.27 & 98.13 \\
 (0.1,0.01,2.0,1.0)   & 83.82 & 87.94 & 88.21 & 91.22 & 94.32 & & 84.11 & 95.30 & 96.33 & 94.84 & 98.52 \\
 (0.1,0.01,2.0,1.4)   & 86.85 & 87.83 & 88.76 & 91.96 & 94.04 & & 87.19 & 95.74 & 96.76 & 95.05 & 98.43 \\
 (0.3,0.005,1.0,0.6)  & 89.13 & 85.15 & 86.84 & 92.05 & 90.92 & & 90.03 & 94.89 & 96.05 & 95.75 & 97.55 \\
 (0.3,0.005,1.0,1.0)  & 94.05 & 74.94 & 87.60 & 93.45 & 92.12 & & 97.11 & 82.10 & 96.48 & 98.14 & 97.85 \\
 (0.3,0.005,1.0,1.4)  & 94.03 & 70.55 & 87.68 & 92.89 & 92.67 & & 97.84 & 77.32 & 96.43 & 97.86 & 97.92 \\
 (0.3,0.005,1.4,0.6)  & 84.99 & 84.75 & 86.36 & 90.52 & 89.67 & & 85.38 & 94.51 & 95.63 & 94.40 & 97.22 \\
 (0.3,0.005,1.4,1.0)  & 91.64 & 85.11 & 87.77 & 92.78 & 92.22 & & 92.22 & 94.34 & 95.90 & 96.08 & 97.89 \\
 (0.3,0.005,1.4,1.4)  & 92.08 & 84.86 & 87.04 & 92.32 & 91.78 & & 92.98 & 94.44 & 95.95 & 96.50 & 97.75 \\
 (0.3,0.005,2.0,0.6)  & 84.73 & 83.82 & 86.08 & 87.80 & 89.07 & & 85.29 & 95.29 & 95.65 & 91.21 & 97.38 \\
 (0.3,0.005,2.0,1.0)  & 84.62 & 82.98 & 85.07 & 89.15 & 90.14 & & 85.11 & 93.72 & 95.08 & 92.64 & 97.43 \\
 (0.3,0.005,2.0,1.4)  & 85.53 & 84.59 & 86.79 & 90.65 & 90.92 & & 85.68 & 94.31 & 95.77 & 92.82 & 97.66 \\
 (0.3,0.01,1.0,0.6)   & 89.42 & 85.47 & 87.28 & 92.22 & 90.57 & & 90.66 & 95.25 & 96.26 & 96.70 & 97.47 \\
 (0.3,0.01,1.0,1.0)   & 92.73 & 72.37 & 86.15 & 92.52 & 91.45 & & 95.38 & 79.98 & 96.07 & 97.89 & 97.66 \\
 (0.3,0.01,1.0,1.4)   & 93.62 & 70.14 & 86.87 & 92.60 & 92.01 & & 97.06 & 76.80 & 96.25 & 97.86 & 97.83 \\
 (0.3,0.01,1.4,0.6)   & 86.19 & 84.36 & 87.10 & 90.30 & 90.14 & & 86.58 & 94.67 & 95.63 & 93.34 & 97.39 \\
 (0.3,0.01,1.4,1.0)   & 88.60 & 84.77 & 86.15 & 91.22 & 91.20 & & 89.13 & 93.57 & 95.76 & 94.84 & 97.66 \\
 (0.3,0.01,1.4,1.4)   & 87.98 & 84.33 & 86.14 & 90.54 & 91.18 & & 88.87 & 94.38 & 95.42 & 95.05 & 97.59 \\
 (0.3,0.01,2.0,0.6)   & 83.12 & 84.94 & 86.35 & 88.11 & 89.86 & & 83.52 & 95.00 & 95.69 & 91.48 & 97.27 \\
 (0.3,0.01,2.0,1.0)   & 83.44 & 82.03 & 84.65 & 85.45 & 90.21 & & 84.28 & 93.57 & 93.78 & 88.75 & 97.46 \\
 (0.3,0.01,2.0,1.4)   & 85.02 & 83.12 & 86.67 & 90.81 & 90.96 & & 85.63 & 93.72 & 95.78 & 92.55 & 97.55 \\
\end{longtable}
}
\vspace{0.5\baselineskip}

Our analysis of the computational results, presented in Table \ref{tab:perf_8000_100000}, highlights several key insights. First, the policies using a stocking plan derived from the joint fluid LP \eqref{eq:relaxed_joint}, namely RSS, RDU, and RDWH, are highly efficient: RDU and RDWH outperform the benchmark on average at both horizons, while RSS is on par with the benchmark at the shorter horizon and substantially exceeds it at the longer one. This provides strong empirical support for our joint fluid-based optimization approach from Section~\ref{sec:joint}. This effectiveness is particularly pronounced as the problem scales, which aligns with our asymptotic theoretical guarantees in Section \ref{sec:joint}. For instance, the average performance of the RSS algorithm, the closest to our asymptotic theoretical algorithms, jumps by 8 percentage points (from 88.31\% to 96.37\%) as the horizon $T$ increases from $8{,}000$ to $100{,}000$, while the static benchmark remains flat. This is consistent with the theoretical guarantee that, as the problem scales, the optimum converges to our fluid approximations, which our policies effectively exploit.

\enlargethispage{\baselineskip}
We also observe a distinction between the two main theoretical approaches. The DGSS policy, implementing the DR-submodular framework from Section \ref{sec:mnljoint} via a dual-greedy heuristic, is outperformed by the rounding-based methods, which correspond to the algorithms from Section \ref{sec:joint}. The results suggest that the main point of weakness of DGSS is the substantial instance-to-instance variation in the objective value attained by the submodular optimization approximation algorithm. This suggests that while the DR-submodular $f_{app}$ function is a powerful theoretical tool, the combination of a $1/2$-approximation in the surrogate (Lemma \ref{lem:fapp_approx}) and the need for complex submodular optimization makes it less robust in practice than the joint fluid-based approach.

At the same time, relative to the benchmark and to its theoretical guarantees, DGSS still performs excellently, far exceeding the guarantees from Theorem \ref{thm:jointconst} and having average performance similar to the benchmark on $T=100{,}000$. Keeping in mind that synthetic data construction inherently favors the benchmark due to underlying synergy (the supply levels $C_i = D_i/\rho$ and warehouse capacities are calibrated from the baseline demand $D_i$ generated by the same myopic optimal assortments $\widetilde S_j$ that the benchmark offers), these results are very encouraging. 

In addition, we gain practical insights from comparing the inventory-aware conversion heuristics. For $T=100{,}000$, the theoretically-sound subsampling-based algorithm (RSS) achieves an excellent 96.37\%. However, the RDWH algorithm, which pre-selects a fulfillment decision according to the fluid solution, performs even better at 98.05\%. This is in line with the intuitive understanding that subsampling tends to be overly conservative in practice, causing the policy to show the empty assortment too often to prevent overselling. At the same time, the RDU algorithm, which only filters products that are out of stock everywhere, while having a strong performance for smaller $T$,  performs worst of the three when $T$ grows to $100{,}000$, indicating that it tends to oversell products from distant warehouses. 

Finally, note that while to get theoretical guarantees without the stationarity assumption we need to artificially reduce demand by scaling fluid solutions, and there are instances when algorithms with unscaled input perform arbitrarily badly, in practice it is best not to apply such demand reduction. Overall, the outstanding performance of the joint fluid-based algorithms, combined with excellent performance of the submodular surrogate approach, suggests that carefully coordinating the inventory placement, assortment personalization, and order fulfillment decisions is both theoretically grounded and effective in practice.

\section{Conclusion}\label{sec:concl}

In this work, we studied the joint inventory placement, assortment personalization, and order fulfillment problem motivated by online retailers serving geographically dispersed customers and operating networks of warehouses. To the best of our knowledge, we provide the first performance guarantees for this joint problem with distributed warehouses and profit heterogeneity incorporating customer characteristics and fulfillment locations.
For general choice models, our algorithms are asymptotically optimal as the minimum warehouse capacity $K^{\min}$ and minimum product supply $C_{\min}$ grow, attaining a square-root convergence rate that notably improves on the cubic-root rate of \citet{baietal2025}, whose guarantee applies only to the much less general single-warehouse setting with profits depending only on products. We provide constant-factor approximation algorithms for the MNL choice model and our advanced assortment and fulfillment policy for the non-stationary case introduces a novel, sophisticated selective demand reduction method, which is far more efficient than the classical scaling approach based on Markov's inequality.

Several interesting directions for future research emerge from our work. First, while we provide the first constant-factor guarantees for the joint problem under MNL, improving these constants remains an important direction. Particularly, one can try enhancing advanced assortment and fulfillment policy or introducing a completely new approach without artificial demand reduction. Second, a significant open question is whether a constant-factor approximation exists for a general choice model under assortment oracle assumption, bridging the gap between our asymptotic results and the MNL case. Third, one can study upper bounds on possible polynomial approximation ratio conditioned on $\text{P} \neq \text{NP}$. Fourth, one could work on extending the results to a setting when products are reusable, for example when products return to the firm in their original state after a random duration (see, for example, \cite{HuangFeldmanZhang2024}). Finally, our model assumes customer types are known upon arrival. Another interesting direction would be to develop robust algorithms that can handle uncertainty or misidentification of customer types, which is a practical challenge for online platforms.

\addcontentsline{toc}{section}{Bibliography}
\bibliographystyle{plainnat}
\bibliography{biblio}

\appendix
\appendix
\makeatletter
\gdef\thesection{\@Alph\c@section}
\renewcommand{\@seccntformat}[1]{\appendixname\ \csname the#1\endcsname:~} 
\makeatother
\section{Proof of Lemma \ref{lem:lpbound}}\label{lpbound}
Consider any feasible assortment and fulfillment policy. Let $Y^{\ell}_{ij}$ denote the number of type-~$j$ customers who purchase product $i$ from warehouse $\ell$, and let $\hat{y}^{\ell}_{ij} = \mathbb{E}[Y^{\ell}_{ij}]$. Similarly, let $W_j(S)$ denote the number of times assortment $S$ is shown to type-$j$ customers, and let $\hat{w}_j(S) = \mathbb{E}[W_j(S)]$. We show that $(\hat{\boldsymbol{w}}, \hat{\boldsymbol{y}})$ is feasible for \eqref{eq:fluid} and its objective value equals the policy's expected profit.
\vspace{0.5\baselineskip}

\noindent\textit{Verifying the first constraint of \eqref{eq:fluid}:} Since a policy cannot sell more than the available inventory, we have $\sum_{j \in \mathcal{M}}Y^{\ell}_{ij} \le c^{\ell}_i$ (almost surely) for all $i \in \mathcal{N}, \ell \in \mathcal{L}$. Taking expectations yields $\sum_{j \in \mathcal{M}}\hat{y}^{\ell}_{ij} \le c^{\ell}_i$.
\vspace{0.5\baselineskip}

\noindent\textit{Verifying the second constraint of \eqref{eq:fluid}:} Let $Y_{ij} = \sum_{\ell \in \mathcal{L}_+}Y_{ij}^{\ell}$ be the total sales of product $i$ to type-$j$ customers. Let $I_{ijt}$ be the indicator that a type-$j$ customer arrives at time period $t$ and chooses product $i$. Also, let $W_{jt}(S)$ be the indicator that a type-$j$ customer arrives at time period $t$ and is offered assortment $S$. The expected sales of product $i$ to type-$j$ customers can be written as
\begin{align*}
\sum_{\ell\in\mathcal{L}_+}\hat{y}^{\ell}_{ij}
&= \mathbb{E}[Y_{ij}]
= \sum_{t\in\mathcal{T}}\mathbb{E}[I_{ijt}] \\
&= \sum_{t\in\mathcal{T}}\sum_{S\subseteq\mathcal{N}}\mathbb{P}(I_{ijt}=1\mid W_{jt}(S)=1)\mathbb{P}(W_{jt}(S)=1) \\
&= \sum_{t\in\mathcal{T}}\sum_{S\subseteq\mathcal{N}}\phi_{ij}(S)\mathbb{E}[W_{jt}(S)] = \sum_{S\subseteq\mathcal{N}}\phi_{ij}(S)\hat{w}_{j}(S).
\end{align*}

\noindent\textit{Verifying the third constraint of \eqref{eq:fluid}:} Since some assortment is shown to every arriving customer, $\sum_{S \subseteq \mathcal{N}} W_j(S)$ equals the total number of type-$j$ arrivals. Taking expectations and using $\tau_j = \sum_{t \in \mathcal{T}} \lambda_{jt}$ yields $\sum_{S \subseteq \mathcal{N}}\hat{w}_j(S) = \tau_j$.
\vspace{0.5\baselineskip}

\noindent Since $W_j(S) \ge 0$ and $Y^{\ell}_{ij} \ge 0$, we have $\hat{w}_j(S) \ge 0$ and $\hat{y}^{\ell}_{ij} \ge 0$. Thus, $(\hat{\boldsymbol{w}}, \hat{\boldsymbol{y}})$ is feasible for \eqref{eq:fluid}. The policy's expected profit is $\mathbb{E}\left[\sum_{i\in\mathcal N}\sum_{j\in\mathcal M}\sum_{\ell\in\mathcal L_+}r^{\ell}_{ij}Y^{\ell}_{ij}\right] = \sum_{i\in\mathcal N}\sum_{j\in\mathcal M}\sum_{\ell\in\mathcal L_+}r^{\ell}_{ij}\hat{y}^{\ell}_{ij},$ which is the objective value of \eqref{eq:fluid} at this solution.

\section{Proof of Lemma \ref{lem:expected_demand_match}}\label{exp_dem_match}
The expected demand $\mathbb{E}[Z^\ell_{ij}]$ is the sum over all time periods of the probability of a type-$j$ customer arriving, being offered an assortment $S$, choosing product $i$, and having that choice assigned to warehouse $\ell$. Therefore, we have
\[
\mathbb{E}[Z^\ell_{ij}] = \sum_{t\in\mathcal{T}} \lambda_{jt} \sum_{S\subseteq\mathcal{N}} \left(\frac{\hat{w}_j(S)}{\tau_j}\right) \phi_{ij}(S) \left(\frac{\hat{y}^\ell_{ij}}{\sum_{\ell' \in \mathcal{L}_+} \hat{y}^{\ell'}_{ij}}\right).
\]
Factoring out terms that do not depend on $t$ and using the definition $\tau_j = \sum_{t\in\mathcal{T}} \lambda_{jt}$, the expression simplifies to
\[
\mathbb{E}[Z^\ell_{ij}] = \frac{\tau_j}{\tau_j} \left(\sum_{S\subseteq\mathcal{N}} \hat{w}_j(S) \phi_{ij}(S)\right) \left(\frac{\hat{y}^\ell_{ij}}{\sum_{\ell' \in \mathcal{L}_+} \hat{y}^{\ell'}_{ij}}\right).
\]
From the flow conservation constraint in the fluid approximation \eqref{eq:fluid}, we know that $\sum_{S\subseteq\mathcal{N}} \hat{w}_j(S) \phi_{ij}(S) = \sum_{\ell' \in \mathcal{L}_+} \hat{y}^{\ell'}_{ij}$. Substituting this into the equation above yields
\[
\mathbb{E}[Z^\ell_{ij}] = \left(\sum_{\ell' \in \mathcal{L}_+} \hat{y}^{\ell'}_{ij}\right) \left(\frac{\hat{y}^\ell_{ij}}{\sum_{\ell' \in \mathcal{L}_+} \hat{y}^{\ell'}_{ij}}\right) = \hat{y}^\ell_{ij}.
\]
The second statement in the lemma follows directly by summing over $j \in \mathcal{M}$.
\section{Proof of Technical Lemmas from Section \ref{prelem}}
In this appendix, we treat product $i$ and a warehouse $\ell$ as fixed and, for notational convenience, whenever we say ``demand'',``sale(s)'', or ``profit'', we only consider demand/sale(s) for product $i$ from warehouse $\ell$.
\subsection{Proof of Claim \ref{lem:cond_prob}}\label{claim3.6}
We first compute joint probability of a type-$j$ arrival and a demand at a given time period $t$.
\begin{align*}
\mathbb{P}(D_{jt} = 1, I^\ell_{it} = 1) &= \lambda_j \sum_{S\subseteq\mathcal{N}} \left(\frac{\hat{w}_j(S)}{\tau_j}\right) \phi_{ij}(S) \left(\frac{\hat{y}^\ell_{ij}}{\sum_{\ell' \in \mathcal{L}_+} \hat{y}^{\ell'}_{ij}}\right) \\
&= \frac{\lambda_j}{\tau_j} \left(\sum_{S\subseteq\mathcal{N}} \hat{w}_j(S) \phi_{ij}(S)\right) \left(\frac{\hat{y}^\ell_{ij}}{\sum_{\ell' \in \mathcal{L}_+} \hat{y}^{\ell'}_{ij}}\right).
\end{align*}
From stationarity it follows that $\tau_j = T \lambda_j$, so $\lambda_j/\tau_j = 1/T$. Furthermore, by the flow constraint in \eqref{eq:fluid}, $\sum_{S\subseteq\mathcal{N}} \hat{w}_j(S) \phi_{ij}(S) = \sum_{\ell' \in \mathcal{L}_+} \hat{y}^{\ell'}_{ij}$. Substituting these gives
\[
\mathbb{P}(D_{jt} = 1, I^\ell_{it} = 1) = \frac{1}{T} \left(\sum_{\ell' \in \mathcal{L}_+} \hat{y}^{\ell'}_{ij}\right) \left(\frac{\hat{y}^\ell_{ij}}{\sum_{\ell' \in \mathcal{L}_+} \hat{y}^{\ell'}_{ij}}\right) = \frac{\hat{y}^\ell_{ij}}{T}.
\]
Next, we find the marginal probability of a demand at time period $t$ by summing over all customer types:
\[
\mathbb{P}(I^\ell_{it} = 1) = \sum_{j' \in \mathcal{M}} \mathbb{P}(D_{j't} = 1, I^\ell_{it} = 1) = \frac{1}{T}\sum_{j' \in \mathcal{M}}\hat{y}^\ell_{ij'}.
\]
The conditional probability is the ratio of the joint and marginal probabilities, which yields the desired result.

\subsection{Proof of Lemma~\ref{exp min bound}}\label{app:expmin}

Let $p^\ell_{it} = \mathbb{P}(\text{demand for }(i,\ell)\text{ at time period } t)$. The expected total  demand is $\mathbb{E}[Z_i^\ell] = \sum_{t \in \mathcal{T}} p^\ell_{it}$. To lower-bound the expected sales, we compare our policy to a more constrained process. Consider a system with $c_i^\ell$ distinct bins, each holding one unit of product $i$. When a demand occurs at time period $t$, we select one of the $c_i^\ell$ bins uniformly at random to fulfill it. A sale is made only if the chosen bin is not empty. By a coupling argument on a common demand sequence, this ``randomized bin'' policy has expected sales less or equal than those of our original policy: it fails to sell to an incoming demand whenever the randomly chosen bin is already empty, even if other units of product $i$ remain in stock at warehouse $\ell$, whereas our original policy sells as long as any unit remains.

The probability that a specific bin is not chosen at time period $t$ is $1 - p^\ell_{it}/c_i^\ell$. The probability it is never chosen over the horizon is $\prod_{t \in \mathcal{T}}(1 - p^\ell_{it}/c_i^\ell)$. Thus, the probability it is chosen at least once is $1 - \prod_{t \in \mathcal{T}}(1 - p^\ell_{it}/c_i^\ell)$. The total expected sales for the bin policy is $c_i^\ell$ times this probability. Since our policy sells at least as much as the bin policy by the coupling argument, its expected sales satisfy the following chain of inequalities:
\begin{align*}
\mathbb{E}[\bar{Z}_i^\ell] &\ge c_i^\ell \left(1 - \prod_{t \in \mathcal{T}}\left(1 - \frac{p^\ell_{it}}{c_i^\ell}\right)\right) \\
&\ge c_i^\ell \left(1 - \left(\frac{1}{T}\sum_{t \in \mathcal{T}}\left(1 - \frac{p^\ell_{it}}{c_i^\ell}\right)\right)^T\right) \\
&= c_i^\ell \left(1 - \left(1 - \frac{\sum_{t \in \mathcal{T}} p^\ell_{it}}{T c_i^\ell}\right)^T\right) \\
&\ge c_i^\ell \left(1 - e^{-\sum_{t \in \mathcal{T}} p^\ell_{it} / c_i^\ell}\right) = c_i^\ell \left(1 - e^{-\mathbb{E}[Z_i^\ell]/c_i^\ell}\right).
\end{align*}
Here, the second inequality is the AM-GM inequality applied to the $T$ nonnegative factors $1 - p^\ell_{it}/c_i^\ell$, and the final inequality follows from the fact that $(1-x/n)^n \le e^{-x}$ for $x \ge 0, n \in \mathbb{N}$.

\subsection{Proof of Lemma~\ref{1e_e}}\label{app:1e}
From Lemma~\ref{exp min bound}, we have $\mathbb{E}[\bar{Z}_i^\ell] \ge c_i^\ell(1 - e^{-\mathbb{E}[Z_i^\ell]/c_i^\ell})$. Let $x = \mathbb{E}[Z_i^\ell]/c_i^\ell$. Since we assume $c_i^\ell \ge \mathbb{E}[Z_i^\ell]$, we have $0 \le x \le 1$. The function $f(x)=1-e^{-x}$ is concave on $[0,1]$, so it lies above the chord connecting $(0,0)$ and $(1, 1-1/e)$. Therefore, $1-e^{-x} \ge (1-1/e)x$ for $x \in [0,1]$. Applying this, we get:
\[
\mathbb{E}[\bar{Z}_i^\ell] \ge c_i^\ell \left( (1-1/e) \frac{\mathbb{E}[Z_i^\ell]}{c_i^\ell} \right) = (1-1/e)\mathbb{E}[Z_i^\ell].
\]
By definition, $\mathbb{E}[\bar{Z}_i^\ell] = \mathbb{E}[Z_i^\ell] - \mathbb{E}[(Z_i^\ell - c_i^\ell)^+]$. Substituting this into the inequality and rearranging yields the desired result.

\section{Technical results from Section \ref{sec:proofadv}}

\subsection{Proof of Claim \ref{claimadv}}\label{app:claimadv}
Suppose for contradiction that the sum is less than $\frac{1}{2}\beta'y$. The expected total profit  loss from the $h$-parameters can be decomposed and bounded using the definitions of $u'$ and the case condition:
\begin{align*}
\alpha r &= \sum_{s=1}^{u'-1} r^\ell_{ij_s}h^\ell_{ij_s}\hat{y}^\ell_{ij_s} + \sum_{s=u'}^u r^\ell_{ij_s}h^\ell_{ij_s}\hat{y}^\ell_{ij_s} \\
&< (1-\epsilon(\beta'))\frac{\alpha r}{\beta'y} \sum_{s=1}^{u'-1}h^\ell_{ij_s}\hat{y}^\ell_{ij_s} + (1+\epsilon(\beta'))\frac{\alpha r}{\beta'y} \sum_{s=u'}^u h^\ell_{ij_s}\hat{y}^\ell_{ij_s}.
\end{align*}
Note that $\sum_{s=1}^{u'-1}h^\ell_{ij_s}\hat{y}^\ell_{ij_s} = \beta'y - \sum_{s=u'}^u h^\ell_{ij_s}\hat{y}^\ell_{ij_s} > \beta'y/2$ by the assumption, and we know $(1-\epsilon(\beta')) < (1+\epsilon(\beta'))$. Because the two sums total $\beta'y$ and the larger coefficient $(1+\epsilon(\beta'))$ multiplies the smaller sum $\sum_{s=u'}^u h^\ell_{ij_s}\hat{y}^\ell_{ij_s}$, replacing each sum by $\beta'y/2$ can only increase the right-hand side, so we have
\[
\alpha r < (1-\epsilon(\beta'))\frac{\alpha r}{\beta'y}(\beta'y/2) + (1+\epsilon(\beta'))\frac{\alpha r}{\beta'y}(\beta'y/2) = \frac{\alpha r}{2}(1-\epsilon(\beta') + 1+\epsilon(\beta')) = \alpha r.
\]
This is a contradiction, so the claim holds.

\subsection{Proof of Lemma \ref{main ineq}}\label{app:conc}
The proof consists of two steps. First, we establish an inequality for random variables, then we take expectations. The quadratic function $f(x) = (x-\mathbb{E}[Z])^2 - 4(x-c)(c-\mathbb{E}[Z])$ has its minimum at $x = 2c - \mathbb{E}[Z]$. At this point, $f(x)=0$, which implies $f(x) \ge 0$ for all $x$.  For any realization of $Z$, the term $(Z-c)^+$ is non-zero only if $Z>c$. For this case, rearranging $f(Z)\ge 0$ gives the pointwise inequality:
\[
(Z - c)^+ \le \frac{(Z-\mathbb{E}[Z])^2}{4(c - \mathbb{E}[Z])}.
\]
Here we use the assumption $c > \mathbb{E}[Z]$, which keeps the denominator $4(c-\mathbb{E}[Z])$ positive so that dividing by it preserves the inequality. Now we take expectations of both sides. Let $Z = \sum_{k=1}^n Z_k$, where $Z_k$ are independent $[0,1]$-valued random variables.
\begin{align*}
\mathbb{E}[(Z-c)^+] &\le \frac{\mathbb{E}[(Z-\mathbb{E}[Z])^2]}{4(c - \mathbb{E}[Z])} = \frac{\mathrm{Var}(Z)}{4(c-\mathbb{E}[Z])} = \frac{\sum_{k=1}^n \mathrm{Var}(Z_k)}{4(c-\mathbb{E}[Z])} \\
&\le \frac{\sum_{k=1}^n \mathbb{E}[Z_k^2]}{4(c-\mathbb{E}[Z])} \le \frac{\sum_{k=1}^n \mathbb{E}[Z_k]}{4(c-\mathbb{E}[Z])} = \frac{\mathbb{E}[Z]}{4(c-\mathbb{E}[Z])}.
\end{align*}
The last inequality holds because $Z_k \in [0,1]$ implies $Z_k^2 \le Z_k$. Dividing by $\mathbb{E}[Z]$ yields the result.

\section{Performance of Policy 1 in the Non-Stationary Case}\label{app:simgen}

To further motivate the need for the advanced policy, we analyze the performance of Policy 1 when applied to the non-stationary case with a uniform scaling parameter $\alpha$. The following lemma shows that this simpler approach has a significantly weaker performance guarantee.

\begin{lemma}[Performance of Policy 1]\label{lem:simple_policy_bound}
Let $Alg_1(\boldsymbol{c}, \alpha)$ be the expected profit of Policy 1 using a fluid solution scaled by $\alpha \in (0,1]$. The performance is bounded by $\mathbb{E}[Alg_1(\boldsymbol{c}, \alpha)] \ge \alpha(1-\alpha)J_1^*(\boldsymbol{c})$. This bound is tight.
\end{lemma}
\begin{proof}
The proof consists of two parts: establishing the lower bound and then showing its tightness with a worst-case example.

\paragraph{Proof of the Lower Bound.}
The expected $(i,\ell)$-sales to customer type $j$ are given by
\[
\mathbb{E}[\bar{Z}_{ij}^\ell] = \sum_{t \in \mathcal{T}} \lambda_{jt}\sum_{S \subseteq \mathcal{N}}\frac{\hat{w}_j(S)}{\tau_j}\phi_{ij}(S)\frac{\hat{y}^{\ell}_{ij}}{\sum_{h \in \mathcal{L}_+} \hat{y}^h_{ij}} \mathbb{E}[A^{\ell}_{it}],
\]
where $A_{it}^\ell$ is the indicator that stock of product $i$ is available at the beginning of time period $t$ on warehouse $\ell$. The factor $\sum_{t \in \mathcal{T}} \lambda_{jt}\sum_{S \subseteq \mathcal{N}}\frac{\hat{w}_j(S)}{\tau_j}\phi_{ij}(S)\frac{\hat{y}^{\ell}_{ij}}{\sum_{h \in \mathcal{L}_+} \hat{y}^h_{ij}}$ appearing here equals $\mathbb{E}[Z_{ij}^\ell]$ by Lemma~\ref{lem:expected_demand_match}, so a lower bound on $\mathbb{E}[A_{it}^\ell]$ directly lower-bounds $\mathbb{E}[\bar{Z}_{ij}^\ell]$. The probability of having stock at the beginning of time period $t$ is at least the probability of not stocking out over the entire horizon: $\mathbb{E}[A_{it}^\ell] \ge \mathbb{P}(Z_i^\ell < c_i^\ell)$. By Markov's inequality:
\[
\mathbb{P}(Z_i^\ell \ge c_i^\ell) \le \frac{\mathbb{E}[Z_i^\ell]}{c_i^\ell} = \frac{\sum_{j \in \mathcal{M}}\hat{y}^\ell_{ij}}{c_i^\ell} \le \frac{\alpha \sum_{j \in \mathcal{M}}y^{*\ell}_{ij}}{c_i^\ell} \le \frac{\alpha c_i^\ell}{c_i^\ell} = \alpha.
\]
Thus, $\mathbb{P}(Z_i^\ell < c_i^\ell) \ge 1-\alpha$. This provides a lower bound on expected sales:
\[
\mathbb{E}[\bar{Z}^\ell_{ij}] \ge \mathbb{E}[Z^\ell_{ij}] \cdot (1-\alpha) = (1-\alpha)\hat{y}^\ell_{ij}.
\]
The total expected profit is therefore bounded by:
\[
 \sum_{i \in \mathcal{N}}\sum_{j\in \mathcal{M}}\sum_{\ell \in \mathcal{L}}  r^\ell_{ij}\mathbb{E}[\bar{Z}^\ell_{ij}] \ge (1-\alpha) \sum_{i \in \mathcal{N}}\sum_{j\in \mathcal{M}}\sum_{\ell \in \mathcal{L}} r^\ell_{ij}\hat{y}^\ell_{ij} = (1-\alpha)\alpha \cdot \text{LP}(\boldsymbol{c}) \ge \alpha(1-\alpha)J_1^*(\boldsymbol{c}).
\]
The quadratic $\alpha(1-\alpha)$ is maximized at $\alpha=1/2$, yielding a best-possible guarantee of $0.25$ for this simple policy.
\end{proof}

\paragraph{Proof of Tightness.} We show the $\alpha(1-\alpha)$ bound is tight by constructing a worst-case instance. Consider an instance with a single product ($n=1$), two customer types ($m=2$), a single warehouse ($L=1$), and a time horizon of two periods ($T=2$), with an initial capacity of $c_1^1=1$. The arrivals are non-stationary:
\begin{itemize}
    \item At $t=1$: A type 1 customer arrives with probability $\lambda_{11}=1-1/k$.
    \item At $t=2$: A type 2 customer arrives with probability $\lambda_{22}=1/k$.
\end{itemize}
All other arrival probabilities are zero. The profits are highly skewed, with $r_{11}^1=1$ and $r_{12}^1=k^2$ for some large $k$. The choice probabilities are $\phi_{11}(\{1\}) = \phi_{12}(\{1\}) = 1$.

\vspace{0.5\baselineskip}
\noindent\textit{Fluid and True Optimal Solutions.}
The optimal fluid solution does not account for the sequential nature of inventory depletion. Its optimal solution is to set $w_1(\{1\}) = 1 - 1/k$, $w_2(\{1\}) = 1/k$, $y^1_{11} = 1 - 1/k$, and $y^1_{12} = 1/k$, with zero sales from the virtual warehouse ($y_{11}^2=y_{12}^2=0$). The corresponding objective value is $\text{LP}(\boldsymbol{c}) = 1\cdot(1-1/k) + k^2\cdot(1/k) = k+1-1/k$.
The true optimal policy, however, must decide at $t=1$ whether to sell to the low profit customer or wait. The optimal strategy is to always wait for the type-$2$ customer, which yields an expected profit of $J_1^*(\boldsymbol{c}) = \mathbb{P}(\text{arrival at } t=2) \cdot r^1_{12} = (1/k) \cdot k^2 = k$.

\vspace{0.5\baselineskip}
\noindent\textit{Policy 1 Performance.}
We now analyze the performance of Policy 1 with parameter $\alpha$. Its expected total profit is the sum of the expected profits from each time period:
\begin{itemize}
    \item \textit{Expected profit at $t=1$:} A type 1 customer arrives with probability $1-1/k$. The policy offers the product with probability $\alpha$, yielding a profit of 1. The expected profit from this time period is $(1-1/k) \cdot \alpha \cdot 1 = \alpha(1-1/k)$.
    
    \item \textit{Expected profit at $t=2$:} A type 2 customer arrives with probability $1/k$. A sale is possible only if no sale occurred at $t=1$ (which happens with probability $1-\alpha(1-1/k)$). The policy offers the product with probability $\alpha$, yielding a profit of $r^1_{12}=k^2$. The expected profit from this time period is $\frac{1}{k} \cdot (1-\alpha(1-1/k)) \cdot \alpha \cdot k^2$.
\end{itemize}
Summing these two terms, the expected total profit for the policy is:
\[
\mathbb{E}[Alg_1(\boldsymbol{c}, \alpha)] = \alpha(1-1/k) + \alpha k(1-\alpha(1-1/k)) = \alpha(1-\alpha)k + \alpha^2 + \alpha - \frac{\alpha}{k}.
\]
The ratio of the policy's profit to the optimum, $J_1^*(\boldsymbol{c})=k$, is therefore:
\[
\frac{\mathbb{E}[Alg_1(\boldsymbol{c}, \alpha)]}{J_1^*(\boldsymbol{c})} = \frac{\alpha(1-\alpha)k + \alpha^2 + \alpha - \frac{\alpha}{k}}{k} = \alpha(1-\alpha) + \frac{\alpha^2 + \alpha}{k} - \frac{\alpha}{k^2}.
\]
As $k \to \infty$, this ratio approaches $\alpha(1-\alpha)$, proving the bound is tight. This result highlights the limitations of Policy 1 and motivates the need for the more sophisticated machinery of Policy 2 to achieve a better guarantee in the non-stationary case. 

\section{Tightness of the Fluid Bound \texorpdfstring{\eqref{eq:fluid}}{(LP)}}
In this section, we analyze tightness of the fluid approximation both under stationarity assumption and without it.
\subsection{Tightness of the Fluid Bound under Stationarity.}
We show that the $1-\frac{1}{e}$ component of the approximation guarantee is tight with respect to the fluid benchmark $\text{LP}(\boldsymbol{c})$. To do this, we construct an instance where the ratio of the optimal expected profit, $J_1^*(\boldsymbol{c})$, to the value of the fluid solution approaches $1-\frac{1}{e}$. This implies two important consequences: (1) no policy can achieve a constant factor better than $1-\frac{1}{e}$ when its performance is measured against the specific upper bound provided by $\text{LP}(\boldsymbol{c})$, and, consequently, (2) any attempt to prove a better guarantee would require a different, stronger benchmark.
\vspace{0.2cm}

\noindent\textit{Instance Construction.} Consider an instance with a single product, customer type, and warehouse ($n=m=L=1$), with capacity $c_1^1=1$ and profit $r_{11}^1=1$. Arrivals are stationary with $\lambda_{1t} = 1/T$ for all $t \in \mathcal{T}$, so the total expected arrivals are $\tau_1=1$. The customer always chooses the product if it is offered, i.e., $\phi_{11}(\{1\})=1$.

The optimal solution to the fluid LP \eqref{eq:fluid} for this instance is to set $w_1(\{1\}) = 1$, $y_{11}^1=1$, and $y_{11}^2=0$ (where warehouse 2 is the virtual one), which yields an objective value of $\text{LP}(\boldsymbol{c}) = 1$.

However, the true optimal dynamic policy is to offer the product in every period until it is sold. A sale occurs if at least one customer arrives during the $T$ periods. The probability of this event is $1 - (1-1/T)^T$. The optimal expected profit is therefore:
\[
J_1^*(\boldsymbol{c}) = 1 \cdot \left(1 - \left(1-\frac{1}{T}\right)^T\right).
\]
As the time horizon $T \to \infty$, the optimal expected profit $J_1^*(\boldsymbol{c})$ approaches $1-\frac{1}{e}$. Thus, the ratio $J_1^*(\boldsymbol{c})/\text{LP}(\boldsymbol{c})$ can be arbitrarily close to $1-\frac{1}{e}$, demonstrating the tightness of the analysis with respect to the fluid bound.

\subsection{Tightness of the Fluid Bound under Non-Stationarity}
We show that for the non-stationary case, there is a larger gap between the fluid approximation and the optimal expected profit. We construct an instance where the ratio $J^*_1(\boldsymbol{c})/\text{LP}(\boldsymbol{c})$ can be arbitrarily close to $1/2$.

\noindent\textit{Instance.}
Consider an instance with one product ($n=1$), one physical warehouse ($L=1$), a capacity of $c_1^1=1$, and two customer types ($m=2$) arriving over two periods ($T=2$). The arrivals are non-stationary:

At $t=1$: A type 1 customer arrives with probability $\lambda_{11}=1-1/k$. No type 2 customer arrives ($\lambda_{21}=0$).

At $t=2$: A type 2 customer arrives with probability $\lambda_{22}=1/k$. No type 1 customer arrives ($\lambda_{12}=0$).

The profits are skewed: selling to a type 1 customer yields $r_{11}^1=1$, while selling to a type 2 customer yields $r_{12}^1=k$ for a large $k$. Customers always purchase if the product is offered: $\phi_{11}(\{1\}) = \phi_{12}(\{1\}) = 1$.
\vspace{0.2cm}

\noindent\textit{Fluid Solution Value.}
The fluid LP \eqref{eq:fluid} does not account for the sequential nature of the problem and the fact that a sale at $t=1$ depletes inventory for $t=2$. Its optimal solution serves both expected customer arrivals. Thus, the optimal fluid expected sales are $y_{11}^1 = 1-1/k$ and $y_{12}^1 = 1/k$. The complete solution specifies $w_1(\{1\})=1-1/k$, $w_2(\{1\})=1/k$, and zero expected sales from the virtual warehouse ($y_{11}^2=y_{12}^2=0$). The optimal value of the fluid solution is therefore:
\[
\text{LP}(\boldsymbol{c}) = 1 \cdot y_{11}^1 + k \cdot y_{12}^1 = 1 \cdot (1-1/k) + k \cdot (1/k) = 2 - 1/k.
\]

\noindent\textit{True Optimal Expected Profit.}
The true optimal policy faces a crucial decision at $t=1$: sell to the arriving type 1 customer or save the single unit of inventory for the potentially more profitable type 2 customer. We compare the expected profits of these two actions:

\textit{Strategy 1: Sell at $t=1$.} If the type 1 customer arrives (probability $1-1/k$), the retailer sells and gets a profit of 1. The inventory is then depleted. The total expected profit from this strategy is $1 \cdot (1-1/k) + k \cdot (1/k)^2 = 1$.

\textit{Strategy 2: Wait for $t=2$.} The retailer forgoes the sale to the type 1 customer. At $t=2$, a type 2 customer arrives with probability $1/k$, yielding a profit of $k$. The total expected profit from this strategy is $k \cdot (1/k) = 1$.

Thus, the two strategies tie, each yielding expected profit $1$, and the optimal expected profit is $J_1^*(\boldsymbol{c})=1$.
\vspace{0.2cm}

\noindent The ratio of the optimal expected profit to the fluid bound is equal to $ \frac{1}{2-1/k}$. As $k \to \infty$, this ratio approaches $1/2$.

\section{Technical Results For Section \ref{sec:joint}}

\vspace{0.5\baselineskip}

\subsection{Verifying Feasibility of the Stocking Plan in Theorem \ref{joint stationarity}}\label{app:jfs}
We show that $\hat{\boldsymbol{c}}$ satisfies the supply and capacity constraints. For any product $i \in \mathcal{N}$:
\begin{align*}
\sum_{\ell \in \mathcal{L}} \hat{c}^\ell_i &\le \sum_{\ell \in \mathcal{L}} (\alpha c^{*\ell}_i + \gamma) = \alpha \sum_{\ell \in \mathcal{L}} c^{*\ell}_i + L\gamma \\
&\le \alpha C_i + L \min\left(\sqrt{\frac{K^{\min}}{n}},\sqrt{\frac{C_{\min}}{L}}\right) \le \left(1-\sqrt{\frac{L}{C_{\min}}}\right)C_i + L\sqrt{\frac{C_{\min}}{L}} \\
&\le C_i - \sqrt{L C_{\min}} + \sqrt{L C_{\min}} = C_i.
\end{align*}
Similarly, for any warehouse $\ell \in \mathcal{L}$:
\begin{align*}
\sum_{i \in \mathcal{N}} \hat{c}^\ell_i &\le \sum_{i \in \mathcal{N}} (\alpha c^{*\ell}_i + \gamma) = \alpha \sum_{i \in \mathcal{N}} c^{*\ell}_i + n\gamma \\
&\le \alpha K^\ell + n \min\left(\sqrt{\frac{K^{\min}}{n}},\sqrt{\frac{C_{\min}}{L}}\right) \le \left(1-\sqrt{\frac{n}{K^{\min}}}\right)K^\ell + n\sqrt{\frac{K^{\min}}{n}} \\
&\le K^\ell - \sqrt{n K^{\min}} + \sqrt{n K^{\min}} \le K^\ell.
\end{align*}
Thus, the stocking plan $\hat{\boldsymbol{c}}$ is feasible.

\subsection{Verifying Feasibility of the Stocking Plan in Theorem \ref{joint general}} \label{app:jfg}
We first show that the aggregate stocking plan $\hat{\boldsymbol{c}}$ satisfies the physical capacity and supply constraints. For any product $i \in \mathcal{N}$:
\begin{align*}
\sum_{\ell \in \mathcal{L}} \hat{c}^\ell_i &\le \sum_{\ell \in \mathcal{L}} (\alpha c^{*\ell}_i + m\gamma) = \alpha \sum_{\ell \in \mathcal{L}} c^{*\ell}_i + Lm\gamma \\
&\le \alpha C_i + Lm \min\left(\sqrt{\frac{K^{\min}}{mn}},\sqrt{\frac{C_{\min}}{mL}}\right) \le \left(1-\sqrt{\frac{Lm}{C_{\min}}}\right)C_i + Lm\sqrt{\frac{C_{\min}}{Lm}} \le C_i.
\end{align*}
Similarly, for any warehouse $\ell \in \mathcal{L}$, the capacity constraint is met:
\begin{align*}
\sum_{i \in \mathcal{N}} \hat{c}^\ell_i &\le \sum_{i \in \mathcal{N}} (\alpha c^{*\ell}_i + m\gamma) = \alpha \sum_{i \in \mathcal{N}} c^{*\ell}_i + nm\gamma \\
&\le \alpha K^\ell + nm \min\left(\sqrt{\frac{K^{\min}}{mn}},\sqrt{\frac{C_{\min}}{mL}}\right) \le \left(1-\sqrt{\frac{nm}{K^{\min}}}\right)K^\ell + nm\sqrt{\frac{K^{\min}}{nm}} \le K^\ell.
\end{align*}
Next, we show that the customer-specific capacities $\hat{c}^\ell_{ij}$ are consistent with the aggregate capacities $\hat{c}^\ell_i$. Using the property that $\sum \lfloor x_k \rfloor \le \left\lfloor \sum x_k \right\rfloor$:
\[
\sum_{j \in \mathcal{M}} \hat{c}^\ell_{ij} = \sum_{j \in \mathcal{M}} \lfloor \hat{y}^\ell_{ij} + \gamma \rfloor \le \left\lfloor \sum_{j \in \mathcal{M}} (\hat{y}^\ell_{ij} + \gamma) \right\rfloor \mathrel{\le} \lfloor \alpha c^{*\ell}_i + m\gamma \rfloor = \hat{c}^\ell_i.
\]
Thus, the stocking plan is feasible at both the aggregate and customer-specific levels.

\section{Proof of Lemma \ref{lem:mnl_lp_equiv}}\label{app:eqlp}
The proof proceeds by a chain of duality arguments. We begin with the dual of \eqref{eq:fluid}, which is:
\begin{equation}\label{eq:proof_dual_1}
\min_{{\boldsymbol{\gamma\ge0},\boldsymbol{\sigma,\beta}}} \left\{ \begin{array}{l@{\quad}l}
     \displaystyle \sum_{i \in \mathcal{N}}\sum_{\ell \in \mathcal{L}} c^\ell_i \gamma_i^\ell + \sum_{j \in \mathcal{M}} \tau_j \sigma_j : & \gamma^\ell_i - \beta_{ij} \ge r^\ell_{ij}, \quad \forall i \in \mathcal{N}, j \in \mathcal{M}, \ell \in \mathcal{L}; \\
     & \sigma_j \ge \displaystyle\max_{S\subseteq\mathcal{N}}\left(-\sum_{i \in \mathcal{N}}\phi_{ij}(S)\beta_{ij}\right), \quad \forall j \in \mathcal{M}
\end{array} \right\}
\end{equation}
For each $j \in \mathcal{M}$, the value of the inner maximization over $S$ is equal to the value of the standard LP relaxation for the single-customer MNL assortment optimization problem, similar to \citet{topaloglu2013joint}:
\begin{equation}\label{eq:proof_inner_primal}
\max_{\boldsymbol{z\ge0}} \left\{ -\sum_{i \in \mathcal{N}} \beta_{ij}z_{ij} : \quad \sum_{i \in \mathcal{N}} z_{ij} + z_{0j} \le 1, \quad z_{ij} \le v_{ij}z_{0j}, \quad \forall i \in \mathcal{N} \right\}
\end{equation}
By strong duality, we can replace the value of each such LP with the value of its dual:
\begin{equation}\label{eq:proof_inner_dual}
\min_{\boldsymbol{w\ge0} } \left\{ w_{0j} : \quad w_{0j} - \sum_{i \in \mathcal{N}} v_{ij}w_{ij} \ge 0, \quad w_{0j} + w_{ij} \ge -\beta_{ij}, \quad \forall i \in \mathcal{N} \right\}
\end{equation}
Substituting this back into the main dual formulation \eqref{eq:proof_dual_1} gives the equivalent dual problem:
\begin{equation}\label{eq:proof_master_dual}
\min_{\boldsymbol{\gamma \ge 0,w\ge0}, \boldsymbol{\sigma,\beta}}  \left\{ \begin{array}{l@{\quad}l}
    \displaystyle \sum_{i \in \mathcal{N}}\sum_{\ell \in \mathcal{L}} c^\ell_i\gamma_i^\ell + \sum_{j \in \mathcal{M}} \tau_j \sigma_j : & \gamma^\ell_i - \beta_{ij} \ge r^\ell_{ij}, \quad \forall i \in \mathcal{N}, j \in \mathcal{M}, \ell \in \mathcal{L}; \\
    & \sigma_j \ge w_{0j}, \quad \forall j \in \mathcal{M}; \\
    & \displaystyle w_{0j} - \sum_{i \in \mathcal{N}} v_{ij}w_{ij} \ge 0, \quad \forall j \in \mathcal{M}; \\
    & w_{0j} + w_{ij} \ge -\beta_{ij}, \quad \forall i \in \mathcal{N}, j \in \mathcal{M}
\end{array} \right\}
\end{equation}
Finally, taking the dual of \eqref{eq:proof_master_dual} yields an LP that directly simplifies to \eqref{eq:mnl_lp_compact}.

\section{Proof of Lemma~\ref{lem:fapp_approx}}\label{app:fapp_approx}

For the lower bound, let $(\boldsymbol{y^*, y_{0}^*})$ be an optimal solution to the LP \eqref{eq:mnl_lp_compact} defining $\text{LP}(\boldsymbol{c})$ (as our choice model is MNL). Consider the scaled solution $\bar{\boldsymbol{y}} = \frac{1}{2}\boldsymbol{y}^*$. This solution is feasible for the LP \eqref{eq:f_app}; the first constraint is met by scaling, and the second and third constraints are met because $\sum_{i \in \mathcal{N},\ell \in \mathcal{L}_+} \bar{y}^\ell_{ij} = \frac{1}{2}\sum_{i \in \mathcal{N},\ell \in \mathcal{L}_+} y^{*\ell}_{ij} \le \frac{1}{2} \tau_j$ and $\sum_{\ell \in \mathcal{L}_+}  \bar{y}^\ell_{ij} = \frac{1}{2}\sum_{\ell \in \mathcal{L}_+} y^{*\ell}_{ij} \le \frac{1}{2}v_{ij}y^*_{0j} \le \frac{1}{2}v_{ij}\tau_j$ (last inequalities in both chains follow directly from constraints in  \eqref{eq:mnl_lp_compact}). The objective value of this feasible solution is $\frac{1}{2}\text{LP}(\boldsymbol{c})$, which proves the lower bound.

For the upper bound, let $\bar{\boldsymbol{y}}$ be an optimal solution to the LP \eqref{eq:f_app} defining $f_{app}(\boldsymbol{c})$. We can construct a feasible solution for the LP \eqref{eq:mnl_lp_compact} by setting $\boldsymbol{y} = \bar{\boldsymbol{y}}$ and $y_{0j} = \tau_j/2$ for all $j \in \mathcal{M}$. This solution has an objective value of $f_{app}(\boldsymbol{c})$. Since $\text{LP}(\boldsymbol{c})$ is the optimal value of \eqref{eq:mnl_lp_compact}, we must have $f_{app}(\boldsymbol{c}) \le \text{LP}(\boldsymbol{c})$.

\section{Proof of Lemma \ref{lem:fapp_dr}}\label{app:dr}

We begin our proof by analyzing the dual of the LP \eqref{eq:f_app} that defines $f_{app}(\boldsymbol{c})$.

\paragraph{Dual Formulation.}
Let $\mu_i^\ell \ge 0$, $\sigma_j \ge 0$, and $\theta_{ij} \ge 0$ be the dual variables for the three sets of constraints in \eqref{eq:f_app}, respectively. The dual LP is:
\begin{equation*}
\min_{\bm{\mu \ge 0, \sigma\ge 0, \theta \ge 0}}  \left\{ \begin{aligned}
&\sum_{i \in \mathcal{N}}\sum_{\ell \in \mathcal{L}} c_i^{\ell} \mu_i^{\ell} + \sum_{j \in \mathcal{M}} \frac{\tau_j}{2} \sigma_j + \sum_{i \in \mathcal{N}}\sum_{j \in \mathcal{M}} \frac{v_{ij}\tau_j}{2} \theta_{ij} : \\
&\quad \mu_i^{\ell} + \sigma_j + \theta_{ij} \ge r^{\ell}_{ij}, \;\; \forall i \in \mathcal{N},\, j \in \mathcal{M},\, \ell \in \mathcal{L}
\end{aligned} \right\}.
\end{equation*}

For fixed $\bm{\mu}$, this problem can be decomposed by customer types. For notational convenience, let us first define the inner minimization problem for each customer type $j \in \mathcal{M}$ as a function of $\bm{\mu}$:
\begin{equation}\label{eq:G_j_def}
G_j(\bm{\mu}) \triangleq \min_{\sigma_j \ge 0, \bm{\theta}_j \boldsymbol{\ge 0}} \left\{ \sigma_j + \sum_{i \in \mathcal{N}} v_{ij} \theta_{ij} : \quad \sigma_j + \theta_{ij} \ge \left(r^{\ell}_{ij} - \mu_i^{\ell}\right)^+ , \; \forall i \in \mathcal{N},  \ell \in \mathcal{L}\right\}.
\end{equation}
Using this definition, we can express $f_{app}(\boldsymbol{c})$ more compactly. Let $L(\boldsymbol{c}, \bm{\mu}) = \sum_{i \in \mathcal{N}}\sum_{\ell \in \mathcal{L}} c_i^{\ell} \mu_i^{\ell}$ and define the auxiliary function $F(\boldsymbol{c}, \bm{\mu})$ as:
\begin{equation}\label{eq:F_def}
F(\boldsymbol{c}, \bm{\mu}) \triangleq L(\boldsymbol{c}, \bm{\mu}) + \sum_{j \in \mathcal{M}} \frac{\tau_j}{2} G_j(\bm{\mu}).
\end{equation}
Then, the surrogate function is the minimization of $F$ over $\bm{\mu}$: $f_{app}(\boldsymbol{c}) = \min_{\bm{\mu}\ge \bf{0}} F(\boldsymbol{c}, \bm{\mu})$.

\subsection{Main Proof Argument}
The proof relies on showing the following property of $F$:

\begin{lemma}\label{lem:F_weak_submodular}
For any real vectors $\boldsymbol{c} \ge \boldsymbol{b} \ge \bf{0}$ and $\bm{\mu}, \bm{\eta} \ge \bf{0}$, we have  $F(\boldsymbol{c},\bm{\mu}) + F(\boldsymbol{b},\bm{\eta}) \ge F(\boldsymbol{c},\bm{\mu} \wedge \bm{\eta}) + F(\boldsymbol{b},\bm{\mu} \vee \bm{\eta})$.
\end{lemma}

This lemma directly follows from the following properties of $L$ and $G_j$, as $F$ is a non-negative linear combination of these functions. 

\begin{lemma}\label{lem:L_submodular_full}
For any real vectors $\boldsymbol{c} \ge \boldsymbol{b} \ge \bf{0}$ and $\bm{\mu}, \bm{\eta} \ge \bf{0}$, we have $L(\boldsymbol{c},\bm{\mu}) + L(\boldsymbol{b},\bm{\eta}) \ge L(\boldsymbol{c},\bm{\mu} \wedge \bm{\eta}) + L(\boldsymbol{b},\bm{\mu} \vee \bm{\eta})$.
\end{lemma}
\begin{lemma}\label{lem:G_submodular_full}
The function $G_j(\bm{\mu})$ is weak-DR submodular i.e. for any real vectors $\bm{\mu}, \bm{\eta} \ge 0$,  we have $G_j(\bm{\mu}) + G_j(\bm{\eta}) \ge G_j(\bm{\mu} \wedge \bm{\eta}) + G_j(\bm{\mu} \vee \bm{\eta})$.
\end{lemma}
These technical lemmas are proven in the next subsections. We now move to the final stage of proving DR-submodularity of $f_{app}$. For any vector $\boldsymbol{x}\ge\bf{0}$ we define $\tilde{\bm{\mu}}_{\boldsymbol{x}} = \arg\min_{\bm{\mu\ge0}} F(\boldsymbol{x}, \bm{\mu})$ and $\tilde{\bm{\mu}}_{\boldsymbol{x}}^+ = \arg\min_{\bm{\mu\ge0}} F(\boldsymbol{x} + \boldsymbol{e}_i^{\ell}, \bm{\mu})$. Consider any integer vectors $\boldsymbol{c} \ge \boldsymbol{b} \ge \bf{0}$ and a unit integer vector $\boldsymbol{e}^{\ell}_i$. The desired inequality is established by the following chain of relations:
\begin{align*}
f_{app}(\boldsymbol{c}) + f_{app}(\boldsymbol{b} + \boldsymbol{e}_i^{\ell}) &\stackrel{(a)}{=} F(\boldsymbol{c},\tilde{\bm{\mu}}_{\boldsymbol{c}}) + F(\boldsymbol{b} + \boldsymbol{e}_i^{\ell},\tilde{\bm{\mu}}_{\boldsymbol{b}}^+) \\
&\stackrel{(b)}{=} F(\boldsymbol{c},\tilde{\bm{\mu}}_{\boldsymbol{c}}) + F(\boldsymbol{b},\tilde{\bm{\mu}}_{\boldsymbol{b}}^+) + (\tilde{\bm{\mu}}_{\boldsymbol{b}}^+)_i^\ell \\
&\stackrel{(c)}{\ge} F(\boldsymbol{c},\tilde{\bm{\mu}}_{\boldsymbol{c}} \wedge \tilde{\bm{\mu}}_{\boldsymbol{b}}^+) + F(\boldsymbol{b},\tilde{\bm{\mu}}_{\boldsymbol{c}} \vee \tilde{\bm{\mu}}_{\boldsymbol{b}}^+) + (\tilde{\bm{\mu}}_{\boldsymbol{b}}^+)_i^\ell \\
&\stackrel{(d)}{\ge} F(\boldsymbol{c},\tilde{\bm{\mu}}_{\boldsymbol{c}} \wedge \tilde{\bm{\mu}}_{\boldsymbol{b}}^+) + F(\boldsymbol{b},\tilde{\bm{\mu}}_{\boldsymbol{c}} \vee \tilde{\bm{\mu}}_{\boldsymbol{b}}^+) + (\tilde{\bm{\mu}}_{\boldsymbol{c}} \wedge \tilde{\bm{\mu}}_{\boldsymbol{b}}^+)_i^\ell \\
&\stackrel{(e)}{=} F(\boldsymbol{c} + \boldsymbol{e}_i^{\ell},\tilde{\bm{\mu}}_{\boldsymbol{c}} \wedge \tilde{\bm{\mu}}_{\boldsymbol{b}}^+) + F(\boldsymbol{b},\tilde{\bm{\mu}}_{\boldsymbol{c}} \vee \tilde{\bm{\mu}}_{\boldsymbol{b}}^+) \\
&\stackrel{(f)}{\ge} F(\boldsymbol{c} + \boldsymbol{e}_i^{\ell},\tilde{\bm{\mu}}_{\boldsymbol{c}}^+) + F(\boldsymbol{b},\tilde{\bm{\mu}}_{\boldsymbol{b}}) \\
&\stackrel{(g)}{=} f_{app}(\boldsymbol{c} + \boldsymbol{e}_i^{\ell}) + f_{app}(\boldsymbol{b}).
\end{align*}
Here, (a) and (g) follow from the definition of $f_{app}$ and $\tilde{\bm{\mu}}$. Equalities (b) and (e) use the property that $F(\boldsymbol{x} + \boldsymbol{e}_i^{\ell}, \bm{\mu}) = F(\boldsymbol{x}, \bm{\mu}) + \mu_i^\ell$. Inequality (c) is a direct application of Lemma \ref{lem:F_weak_submodular}. Inequality (d) holds because  $(\tilde{\bm{\mu}}_{\boldsymbol{b}}^+)_i^\ell \ge (\tilde{\bm{\mu}}_{\boldsymbol{c}} \wedge \tilde{\bm{\mu}}_{\boldsymbol{b}}^+)_i^\ell$. Finally, step (f) follows from the definition of $\tilde{\bm{\mu}}$ as a minimizer. This completes the main proof.

\subsection{Proof of Lemma~\ref{lem:L_submodular_full}}
The lemma follows from a simple identity: $x+y = (x \wedge y) + (x \vee y)$. Indeed, let $\delta_i^\ell = \mu_i^\ell - (\mu_i^\ell \wedge \eta_i^\ell) = (\mu_i^\ell \vee \eta_i^\ell) - \eta_i^\ell \ge 0$. We have:
\begin{align*}
L(\boldsymbol{c}, \bm{\mu}) - L(\boldsymbol{c}, \bm{\mu} \wedge \bm{\eta}) &= \sum_{i \in \mathcal{N}}\sum_{\ell \in \mathcal{L}} c_i^\ell (\mu_i^\ell - (\mu_i^\ell \wedge \eta_i^\ell)) = \sum_{i \in \mathcal{N}}\sum_{\ell \in \mathcal{L}} c_i^\ell \delta_i^\ell \\
&\ge \sum_{i \in \mathcal{N}}\sum_{\ell \in \mathcal{L}} b_i^\ell \delta_i^\ell 
= \sum_{i \in \mathcal{N}}\sum_{\ell \in \mathcal{L}} b_i^\ell ((\mu_i^\ell \vee \eta_i^\ell) - \eta_i^\ell)\\&=  L(\boldsymbol{b}, \bm{\mu} \vee \bm{\eta}) - L(\boldsymbol{b}, \bm{\eta}),
\end{align*}
where the inequality follows from $\boldsymbol{c} \ge \boldsymbol{b} \text{ and } \delta_i^\ell \ge 0$. Rearranging the terms gives the desired inequality.

\subsection{Proof of Lemma~\ref{lem:G_submodular_full}}
To prove that $G_j(\bm{\mu})$ is weak-DR submodular, we use an equivalent definition from \cite{bian2017continuous}: $G_j$ is weak-DR submodular if and only if for any non-negative real vectors $\bm{\mu} \ge \bm{\eta}$ and any single component $(i,\ell)$ where they are equal ($\mu_i^\ell = \eta_i^\ell$), the following inequality holds for any $h > 0$:
\[
G_j(\bm{\mu} + h \boldsymbol{e}_i^{\ell}) - G_j(\bm{\mu}) \le G_j(\bm{\eta} + h \boldsymbol{e}_i^{\ell}) - G_j(\bm{\eta}).
\]
Our strategy is to analyze the rate of change of the function $G_j$ along the direction $\boldsymbol{e}_i^{\ell}$. To do this, let us fix a component $(i,\ell)$ and define two single-variable helper functions: $g_{\bm{\mu}}(t) = G_j(\bm{\mu} + t \boldsymbol{e}_i^{\ell})$ and $g_{\bm{\eta}}(t) = G_j(\bm{\eta} + t \boldsymbol{e}_i^{\ell})$. As value functions of a linear program with bounded feasible region, these helper functions are continuous and piecewise-linear, which implies they are differentiable everywhere except for a finite number of points. Where the derivative exists, it is given by the negative of the optimal value of the dual variable corresponding to the $(i,\ell)$-constraint. The dual of the linear program for $G_j(\bm{\mu})$ in \eqref{eq:G_j_def} is the \textit{(modified) continuous knapsack problem}:
\begin{equation}\label{eq:knapsack_dual_full}
\max_{\boldsymbol{z}_j \ge \bf{0}} \left\{ \sum_{k \in \mathcal{N}} \sum_{w \in \mathcal{L}} \left(r^w_{kj} - \mu_k^w\right)^+ \cdot z_{kj}^w : \sum_{k \in \mathcal{N},w \in \mathcal{L}} z_{kj}^w \le 1, \; \sum_{w \in \mathcal{L}} z_{kj}^w \le v_{kj}, \; \forall k \in \mathcal{N} \right\}.
\end{equation}
Let $z_{ij}^{\ell*}(\bm{\mu})$ be the maximal value that the variable $z_{ij}^\ell$ takes across the set of all optimal solutions to this LP. The derivatives of our helper functions, where they exist, are given by $g'_{\bm{\mu}}(t) = -z_{ij}^{\ell*}(\bm{\mu} + t \boldsymbol{e}_i^{\ell})\,\mathbbm{1}[\mu_i^{\ell}+t < r^{\ell}_{ij}]$ and $g'_{\bm{\eta}}(t) = -z_{ij}^{\ell*}(\bm{\eta} + t \boldsymbol{e}_i^{\ell})\,\mathbbm{1}[\eta_i^{\ell}+t < r^{\ell}_{ij}]$. The proof now hinges on the monotonicity of this optimal solution, which we establish in the following claim.

\begin{claim}
If $\bm{\mu} \ge \bm{\eta}$ and $\mu_i^\ell = \eta_i^\ell$ for some $(i,\ell)$, then $z_{ij}^{\ell*}(\bm{\mu}) \ge z_{ij}^{\ell*}(\bm{\eta})$.
\end{claim}
\begin{proof}
Let the dual objective coefficient for a variable $z_{kj}^w$ be $\pi_{kj}^w(\bm{\mu}) = (r^w_{kj} - \mu_k^w)^+$. Since $\mu_i^\ell = \eta_i^\ell$, the objective coefficient $\pi_{ij}^\ell$ of $(i,\ell)$ is the same in both problems: $\pi_{ij}^\ell(\bm{\mu}) = \pi_{ij}^\ell(\bm{\eta})$. For any other product-warehouse pair $(k,w)$, the condition $\bm{\mu} \ge \bm{\eta}$ implies $\pi_{kj}^w(\bm{\mu}) \le \pi_{kj}^w(\bm{\eta})$. These two properties of $\pi$ will be key in our proof.

If the set of warehouses offering better value for product $i$ than $\ell$ under $\bm{\mu}$, $B_{\bm{\mu}} = \{ w \in \mathcal{L} \mid\pi_{ij}^w(\bm{\mu}) > \pi_{ij}^\ell(\bm{\mu}) \} $, is non-empty, same is true for the similarly defined $B_{\bm{\eta}}$: for any $w \in B_{\bm{\mu}}$, the equality $\pi_{ij}^\ell(\bm{\mu}) = \pi_{ij}^\ell(\bm{\eta})$ and the inequality $\pi_{ij}^w(\bm{\eta}) \ge \pi_{ij}^w(\bm{\mu})$ give $\pi_{ij}^w(\bm{\eta}) \ge \pi_{ij}^w(\bm{\mu}) > \pi_{ij}^\ell(\bm{\mu}) = \pi_{ij}^\ell(\bm{\eta})$, so $w \in B_{\bm{\eta}}$. Moreover, whenever either $B_{\bm{\mu}}$ or $B_{\bm{\eta}}$ is non-empty, the corresponding optimal value $z^{\ell*}_{ij}(\bm{\mu})$ or $z^{\ell*}_{ij}(\bm{\eta})$ equals zero: the constraint $\sum_{w \in \mathcal{L}} z_{ij}^w \le v_{ij}$ pools all warehouses serving product $i$ under a single budget, and every warehouse in $B_{\bm{\mu}}$ (respectively $B_{\bm{\eta}}$) carries a strictly larger objective coefficient than $\ell$, so any optimal solution directs the entire allocation for product $i$ away from $\ell$. Consequently, if $B_{\bm{\eta}} \neq \varnothing$ then $z^{\ell*}_{ij}(\bm{\eta}) = 0 \le z^{\ell*}_{ij}(\bm{\mu})$ and the claim holds, while if $B_{\bm{\eta}} = \varnothing$ then $B_{\bm{\mu}} = \varnothing$ as well; it therefore remains to treat the case $B_{\bm{\mu}} = B_{\bm{\eta}} = \varnothing$. Suppose, then, that both sets are empty.

Let $A_{\bm{\mu}} = \{ k \in \mathcal{N} \mid \exists w \in \mathcal{L}: \pi_{kj}^w(\bm{\mu}) > \pi_{ij}^\ell(\bm{\mu}) \}$ be the set of product-warehouse pairs that outperform $(i,\ell)$ under $\bm{\mu}$. Similarly, define $A_{\bm{\eta}}$. Key properties of $\pi$ imply that $A_{\bm{\mu}} \subseteq A_{\bm{\eta}}$. Problem \eqref{eq:knapsack_dual_full} is a continuous (fractional) knapsack with a single unit budget $\sum_{k \in \mathcal{N}, w \in \mathcal{L}} z_{kj}^w \le 1$, so it is solved greedily: the products are sorted in decreasing order of their best coefficient $\max_{w \in \mathcal{L}} \pi_{kj}^w$ and each is allocated up to $v_{kj}$ until the budget is exhausted. The maximal allocation to $(i,\ell)$ is therefore determined by the residual capacity left after filling the products with higher maximum value:
\[
z_{ij}^{\ell*}(\bm{\mu}) = \min \left\{ v_{ij}, \left(1 - \sum_{k \in A_{\bm{\mu}}} v_{kj}\right)^+ \right\} \quad \text{and} \quad z_{ij}^{\ell*}(\bm{\eta}) = \min \left\{ v_{ij}, \left(1 - \sum_{k \in A_{\bm{\eta}}} v_{kj}\right)^+ \right\}.
\]
Since $A_{\bm{\mu}} \subseteq A_{\bm{\eta}}$, the sum over $A_{\bm{\mu}}$ is smaller, leaving more residual capacity. Thus, $z_{ij}^{\ell*}(\bm{\mu}) \ge z_{ij}^{\ell*}(\bm{\eta})$.
\end{proof}
The claim implies that $g'_{\bm{\mu}}(t) \le g'_{\bm{\eta}}(t)$ for all $t$ where both derivatives exist. Since this inequality holds almost everywhere on the interval $[0, h]$ and the functions are continuous, we can integrate both sides of $g'_{\bm{\mu}}(t) \le g'_{\bm{\eta}}(t)$ over $t \in [0,h]$ to obtain $g_{\bm{\mu}}(h) - g_{\bm{\mu}}(0) \le g_{\bm{\eta}}(h) - g_{\bm{\eta}}(0)$. Substituting back the definitions of the helper functions, $g_{\bm{\mu}}(h) = G_j(\bm{\mu} + h \boldsymbol{e}_i^{\ell})$ and $g_{\bm{\mu}}(0) = G_j(\bm{\mu})$ (and similarly for $\bm{\eta}$), yields the desired inequality:
\[
G_j(\bm{\mu} + h \boldsymbol{e}_i^{\ell}) - G_j(\bm{\mu}) \le G_j(\bm{\eta} + h \boldsymbol{e}_i^{\ell}) - G_j(\bm{\eta}).
\]
This is precisely the condition we set out to prove, which completes the proof of the lemma.

\section{Final Details of the Proof of Theorem \ref{thm:jointconst}}\label{app:pg}

In this appendix we explicitly write out the factors in the performance guarantees of the algorithm.

\noindent\textit{Case 1: General Case (No special assumptions).}
The guarantee is the product of:
\begin{enumerate}[label=(\roman*)]
    \item The $(1/2-\epsilon)$ factor for maximizing $f_{app}$ over a two-matroid intersection.
    \item The $1/2$ factor from Lemma~\ref{lem:fapp_approx}, relating $f_{app}(\boldsymbol{c})$ to $\text{LP}(\boldsymbol{c})$.
    \item The $0.322$ factor from Policy 2 for the general non-stationary case.
\end{enumerate}
The combined factor is $\beta = (1/2-\epsilon) \cdot 1/2 \cdot 0.322$ . We therefore have a $(0.0805 - \epsilon)$-approximation.

\vspace{0.5\baselineskip}
\noindent\textit{Case 2: Unlimited Product Supplies ($C_{\min} = +\infty$).}
The guarantee is the product of:
\begin{enumerate}[label=(\roman*)]
    \item The $(1-1/e)$ factor for maximizing $f_{app}$ over a single matroid.
    \item The $1/2$ factor from Lemma~\ref{lem:fapp_approx}.
    \item The $0.322$ factor from Policy 2.
\end{enumerate}
The combined factor is  $\beta = (1-1/e) \cdot 1/2 \cdot 0.322$. This gives a $0.161(1-1/e)$-approximation.

\vspace{0.5\baselineskip}
\noindent\textit{Case 3: Stationary Arrivals.}
The guarantee is the product of:
\begin{enumerate}[label=(\roman*)]
    \item The $(1/2-\epsilon)$ factor for maximizing $f_{app}$ over a two-matroid intersection.
    \item The $1/2$ factor from Lemma~\ref{lem:fapp_approx}.
    \item The $(1-1/e)$ factor from Policy 1 under stationarity.
\end{enumerate}
The combined factor is $\beta = (1/2-\epsilon) \cdot \frac{1}{2} \cdot (1-1/e)$. This gives a $(\frac{1}{4}(1-1/e) - \epsilon)$-approximation.

\vspace{0.5\baselineskip}
\noindent\textit{Case 4: Stationary Arrivals and Unlimited Supplies.}
The guarantee is the product of:
\begin{enumerate}[label=(\roman*)]
    \item The $(1-1/e)$ factor for maximizing $f_{app}$ over a single matroid.
    \item The $1/2$ factor from Lemma~\ref{lem:fapp_approx}.
    \item The $(1-1/e)$ factor from Policy 1.
\end{enumerate}
The combined factor is $\beta = (1-1/e)  \cdot \frac{1}{2} \cdot (1-1/e) = \frac{1}{2}(1-1/e)^2$. This gives a $\frac{1}{2}(1-1/e)^2$-approximation.

\section{Proof of Theorem \ref{thm:ib}}\label{app:ib}
We use a coupling argument and prove by induction on $t$ that the distribution of the vector of total \textit{(non-virtual)} sales up to time period $t$ is identical for both policies. Let $\bar{\boldsymbol{Z}}_t = (\bar{Z}^\ell_{ijt})_{i,j,\ell}$ be the sales vector for the inventory-agnostic policy and $\hat{\boldsymbol{Z}}_t = (\hat{Z}^\ell_{ijt})_{i,j,\ell}$ be the sales vector for the inventory-aware policy.

\textbf{Base Case $(t=0)$:} Both sales vectors are zero, so the property holds.

\textbf{Inductive Step:} Assume that the distributions of $\bar{\boldsymbol{Z}}_{t-1}$ and $\hat{\boldsymbol{Z}}_{t-1}$ are identical. Fix a realizable value $\boldsymbol{z} = (z^\ell_{ik})_{i,k,\ell}$ of the cumulative sales through period $t-1$, and condition each policy on its cumulative sales equalling $\boldsymbol{z}$, that is, on the events $\{\bar{\boldsymbol{Z}}_{t-1} = \boldsymbol{z}\}$ and $\{\hat{\boldsymbol{Z}}_{t-1} = \boldsymbol{z}\}$, respectively. We show that the conditional distribution of the period-$t$ sales given this event is the same function of $\boldsymbol{z}$ under both policies; since $\bar{\boldsymbol{Z}}_{t-1}$ and $\hat{\boldsymbol{Z}}_{t-1}$ have the same distribution, this yields that $\bar{\boldsymbol{Z}}_{t}$ and $\hat{\boldsymbol{Z}}_{t}$ have the same distribution.

Consider a sale of product $i$ to a type-$j$ customer from warehouse $\ell$ at time period $t$. Let $\bar{A}^\ell_{it}$ and $\hat{A}^\ell_{it}$ be the stock availability indicators for the agnostic and aware policies, respectively. Given the cumulative sales $\boldsymbol{z}$, both are the same deterministic function of $\boldsymbol{z}$:
\[ \bar{A}^\ell_{it} = \hat{A}^\ell_{it} = A^\ell_{it}(\boldsymbol{z}) := \mathbbm{1}\!\Big[\, \textstyle\sum_{k \in \mathcal{M}} z^\ell_{ik} < c^\ell_i \,\Big], \]
since product $i$ is in stock at warehouse $\ell$ exactly when fewer than its stocked quantity $c^\ell_i$ units have been sold from $\ell$.

\textit{For the agnostic policy}, a sale occurs if a demand for $(i,j,\ell)$ is generated and the item is in stock. The probability of demand for $(i,j,\ell)$ at time period $t$ is:
    \[  \lambda_{jt} \sum_{S\subseteq\mathcal{N}} \frac{\hat{w}_j(S)}{\tau_j} \phi_{ij}(S) \frac{\hat{y}^\ell_{ij}}{\sum_{\ell' \in \mathcal{L}_+} \hat{y}^{\ell'}_{ij}} = \frac{\lambda_{jt}}{\tau_j} \hat{y}^\ell_{ij}. \]
    Since this demand is generated independently of the history and, given $\bar{\boldsymbol{Z}}_{t-1} = \boldsymbol{z}$, product $i$ is in stock at warehouse $\ell$ exactly when $A^\ell_{it}(\boldsymbol{z}) = 1$, we obtain
    \[ \mathbb{P}\big(\text{sale under agnostic policy} \,\big|\, \bar{\boldsymbol{Z}}_{t-1} = \boldsymbol{z}\big) = \frac{\lambda_{jt}}{\tau_j} \hat{y}^\ell_{ij} \cdot A^\ell_{it}(\boldsymbol{z}). \]
    
\textit{For the aware policy}, a sale occurs if customer of type $j$ arrives, warehouse $\ell$ is pre-selected for product $i$, the product is in stock there, and the customer chooses product $i$ from the offered assortment. Warehouse $\ell$ is pre-selected for product $i$ (for a type-$j$ customer) with probability $\rho_{ij}(\ell)$. Given $\hat{\boldsymbol{Z}}_{t-1} = \boldsymbol{z}$, product $i$ is in stock on warehouse $\ell$ exactly when $A^\ell_{it}(\boldsymbol{z}) = 1$. Finally, by construction and the law of total probability, the probability that the customer chooses product $i$ given it is available (on the pre-selected warehouse $\ell$) is precisely the target marginal probability, which is $ \sum_{\ell' \in \mathcal{L}_+} \hat{y}^{\ell'}_{ij} / \tau_j$, for every realization $\boldsymbol{z}$. The conditional probability of a sale is therefore:
    \begin{align*}
    \mathbb{P}\big(\text{sale under aware policy} \,\big|\, \hat{\boldsymbol{Z}}_{t-1} = \boldsymbol{z}\big) &= \lambda_{jt} \cdot \rho_{ij}(\ell) \cdot A^\ell_{it}(\boldsymbol{z}) \cdot \mathbb{P}(i \text{ is chosen}  \mid i \text{ is available}) \\
    &= \lambda_{jt} \cdot \left(\frac{\hat{y}^\ell_{ij}}{\sum_{\ell' \in \mathcal{L}_+} \hat{y}^{\ell'}_{ij}}\right) \cdot A^\ell_{it}(\boldsymbol{z}) \cdot \left(\frac{\sum_{\ell' \in \mathcal{L}_+} \hat{y}^{\ell'}_{ij}}{\tau_j}\right) \\
    &= \frac{\lambda_{jt}}{\tau_j} \hat{y}^\ell_{ij} \cdot A^\ell_{it}(\boldsymbol{z}).
    \end{align*}

Thus, for every realizable $\boldsymbol{z}$, both policies assign the same probability to a sale of each fixed triple $(i,j,\ell)$ in period $t$; as at most one sale can occur in a period, these probabilities determine the entire conditional distribution of the period-$t$ sales given that the cumulative sales through period $t-1$ equal $\boldsymbol{z}$, which is therefore the same function of $\boldsymbol{z}$ for both policies. This completes our induction step.

Equality of distributions of $\bar{\boldsymbol{Z}}_T$ and $\hat{\boldsymbol{Z}}_T$ directly implies equality of the total expected profits of the two policies.

\end{document}